\documentclass[12pt]{amsart}
\usepackage{epsfig}
\usepackage[dvipsnames]{xcolor}

\usepackage{hyperref}
\hypersetup{
    colorlinks=true, 
    linktoc=all,     
    linkcolor=blue,  
}

\usepackage{amsrefs}

\usepackage{mathrsfs}
\usepackage{amssymb}

\theoremstyle{definition}

\newtheorem{theorem}{Theorem}[section]
\newtheorem{definition}[theorem]{Definition}
\newtheorem{conjecture}[theorem]{Conjecture}
\newtheorem{proposition}[theorem]{Proposition}
\newtheorem{lemma}[theorem]{Lemma}
\newtheorem{remark}[theorem]{Remark}
\newtheorem{corollary}[theorem]{Corollary}

\newtheorem{question}[theorem]{Question}

\numberwithin{equation}{section}

\usepackage{enumerate}

\DeclareMathOperator*{\ind}{ind}
\DeclareMathOperator*{\nul}{null}
\DeclareMathOperator*{\spa}{span}

\newcommand{\tr}{\operatorname{tr}}

\newcommand{\wt}{\widetilde}
\newcommand{\pr}{\partial}

\newcommand{\Lap}{\Delta}

\newcommand{\Ric}{\operatorname{Ric}}

\DeclareMathOperator{\ct}{ct}

\newcommand{\pit}{\frac{\pi}{2}}

\usepackage{comment}

\usepackage{esint}

\usepackage{subcaption}

\def\MR#1{}

\title[Capillary minimal surfaces in spherical caps: Low genus]{Free boundary and capillary minimal surfaces in spherical caps I: Low genus}

\author{Keaton Naff}
\address{Lehigh University, Bethlehem, PA, USA}
\email{ken424@lehigh.edu}
\author{Jonathan J. Zhu}
\address{University of Washington, Seattle, WA, USA}
\email{jonozhu@uw.edu}

\begin{document}
\begin{abstract}
This is the first of two articles in which we investigate the geometry of free boundary and capillary minimal surfaces in balls $B_R\subset\mathbb{S}^3$. In this article, we extend our previous half-space intersection properties to warped products, and extend (non-)umbilicity of discs and annuli to capillary minimal surfaces in high codimension. We establish a dual operation relating free boundary and capillary minimal surfaces. These results are discussed in a continuous, unified framework, particularly in relation to uniqueness of minimal annuli. 
\end{abstract}
\date{\today}
\maketitle


\vspace{-2em}

\section{Introduction}

This article is the first of two papers (together with \cite{NZ25b}) in which we investigate the geometry of free boundary and capillary minimal surfaces $\Sigma$ in geodesic balls $B^3_R\subset \mathbb{S}^3$. We endeavour to study these surfaces under a unified framework, with key parameters given by the cap radius $R$ and the contact angle $\gamma$. In this article, we focus on developing this framework, as well as rigidity phenomena for low topology. 

Under this framework, different settings are related according to the continuous parameters $R,\gamma$. The best-understood case is $(R, \gamma) = (\pit,  \pit)$, i.e. free-boundary minimal surfaces (FBMS) in a hemisphere - equivalently, closed minimal surfaces in $\mathbb{S}^3$ with a reflection symmetry. This special case serves as a model for developing theory in less-understood cases. The more established setting of FBMS in a Euclidean ball $\mathbb{B}_1 \subset \mathbb{R}^3$ naturally arises as $R \to 0$ (fixing $\gamma = \pit$). The study of FBMS in spherical caps was recently initiated in \cite{LM23, Me23, dO24}.

To illustrate the relationship hinted at above, we consider the striking classification results for minimal surfaces in $\mathbb{S}^3$ of simplest topology: Almgren's result that the only minimal spheres in $\mathbb{S}^3$ are equators \cite{Alm66}, and the Lawson conjecture - proven by Brendle \cite{Br12} - that the only embedded minimal tori are Clifford tori. For the $(R,\pit)$ setting, the analogous questions become the classification of free boundary minimal discs, and of embedded free boundary minimal annuli. In particular, the analogous question as $R\to0$ is  Nitsche's conjecture on the uniqueness of the critical catenoid. 

A classification of free boundary minimal discs in space forms was proven by Fraser-Schoen \cite{FS15} using a Hopf differential method. For annuli, the situation is much richer, as evidenced by the constructions of \cite{FHM23, CFM25}, and the apparent difficulty of proving Nitsche's conjecture. Indeed, complementing our unified discussion, we will show that minimal annuli with parameters $(R,\gamma)$ admit \textit{polar dual} minimal immersions with dual parameters $(\tilde{R},\tilde{\gamma})$ - this polar dual is in fact given by the Gauss map. We find that this duality is most effective from embedded $(R,\pit)$-minimal annuli to embedded $(\pit, R)$-minimal annuli:

\begin{theorem}
\label{thm:dual-intro}
Let $R\in (0,\pit]$ and consider a ball $B_R\subset \mathbb{S}^3$. Let $x : (\Sigma , \partial \Sigma) \hookrightarrow (B_R, \partial B_R)$ be an embedded free boundary minimal annulus. Then its polar dual is an embedded capillary minimal annulus in the hemisphere $(\mathbb{S}^3_+, \pr \mathbb{S}^3_+)$ with contact angle $R$.
\end{theorem}

In this way, by studying these problems together, we find a new correspondence between capillary minimal surfaces in the hemisphere and free boundary minimal surfaces, which we expect will provide an interesting bridge and give each setting newfound significance. (Note that capillary minimal surfaces in the hemisphere model capillary minimal cones.) We also note that the limiting setting $\gamma\to 0$ (keeping $R=\pit$) in our framework corresponds to the overdetermined Serrin problem (equivalently, links of homogeneous solutions of the one-phase free boundary problem) studied, for instance, in the recent works \cite{EM23, EM24, CKM25}. 

In general, Lawson \cite{Law70} showed that the Gauss map $\nu: \Sigma \to \mathbb{S}^3$ of a minimal surface in $\mathbb{S}^3$ defines a (branched) minimal immersion. In addition to checking the boundary behaviour, Theorem \ref{thm:dual-intro} uses two key ingredients, which also explain why the dual operation is most effective for embedded annuli:

\begin{itemize}
    \item Non-umbilicity for minimal annuli, which ensures an absence of branch points (and a consistent contact direction along the boundary); 
    \item A two-piece property for embedded FBMS of genus 0, which is used to control the topology of the polar dual. 
\end{itemize}

Theorem \ref{thm:dual-intro} is an analogue of a classical result of Ros \cite{Ros95}, who showed that the polar dual of an embedded minimal torus in $\mathbb{S}^3$ is itself an embedded minimal torus. (Of course, Ros' result may now be considered a consequence of the Lawson conjecture.) An important tool in Ros' work \cite{Ros95} was his so-called two-piece property (for closed minimal surfaces in $\mathbb{S}^3)$. As we discussed in \cite{NZ24}, two-piece properties often follow from \textit{half-space intersection properties} for minimal surfaces; in that paper we also collected proofs of such properties in various settings: closed minimal hypersurfaces in $\mathbb{S}^n$, FBMS in geodesic balls $B_R \subset M \in \{\mathbb{H}^n,\mathbb{R}^n, \mathbb{S}^n\}$, and complete self-shrinkers in $\mathbb{R}^n$. In this paper, we take the opportunity to further generalise our techniques to a rather more general class of warped-product manifolds:

\begin{theorem}[Half-space Frankel in warped-products]
\label{thm:frankel-intro}
Let $(M^{n+1}, \bar{g}) = (\mathbb{B}^{n+1}_{\bar{r}}, e^{2\phi(r)} \delta)$ be a radially-symmetric, conformally Euclidean metric for some $\bar{r} \in (0, \infty]$, where $\mathbb{B}_{\bar{r}}$ denotes the Euclidean ball. Assume that the Ricci curvature $\Ric_{\bar{g}}$ is either (strictly) largest in the radial direction, or is constant, in which case $(M^{n+1}, \bar{g})$ is a geodesic ball in a space form.

Suppose $\Sigma_0, \Sigma_1$ are compact, properly embedded free boundary minimal hypersurfaces in $\mathbb{B}_{r_0}$ for some $r_0 \in (0, \bar{r}]$. Then $\Sigma_0$ and $\Sigma_1$ intersect in every closed half ball of $\overline{\mathbb{B}_{r_0}}$.
\end{theorem}

The resulting two-piece property (see Corollary \ref{cor:two-piece}) has geometric implications for genus-zero minimal surfaces. In addition to Ros' work, two-piece properties played an important role in Brendle's classification \cite{Bre16} of embedded genus-zero self-shrinkers (Gaussian minimal surfaces) in $\mathbb{R}^3$. In Brendle's work, the two-piece property was used to show genus-zero self-shrinkers are star-shaped, i.e. radial graphs. (Note that in the shrinker setting, star-shapedness is equivalent to mean convexity.) In Section \ref{sec:nodal}, we show that genus-zero minimal surfaces in any of the warped-product manifolds $(M^3, \bar{g})$ above are, likewise, radial graphs. 

Let us return to discussion of (non-)umbilicity. Classically, the Hopf differential technique was used to classify closed genus-zero minimal surfaces in $\mathbb{S}^n$ as totally geodesic, but also yields that a closed genus-one minimal surface in $\mathbb{S}^n$ has no umbilic points. (The latter was a key prerequisite to Brendle's proof \cite{Br12} of the Lawson conjecture.) The Hopf differential method has previously been extended to classify free boundary minimal discs in $\mathbb{B}^3$ by Nitsche \cite{Ni85}, free boundary minimal discs in space form balls by Fraser-Schoen \cite{FS15}, and capillary minimal discs in a hemisphere $\mathbb{S}^3_+$ by Chodosh-Edelen-Li \cite{CEL24}. 

In this article, we develop the Hopf differential method further, introducing a new class of surfaces in higher codimension that generalises the notion of capillary minimal surfaces. We extend the classification of discs to this new class, and also show that capillary minimal annuli in any space form ball (of any codimension) have no umbilic points. (See Section \ref{sec:hopf} for the precise statements.) This extension unifies all previous results within a single framework. 

With the above two ingredients in hand, we prove Theorem \ref{thm:dual-intro} following the general strategy of Ros \cite{Ros95}, with some key adaptations to the free boundary setting. We then return to our discussion of uniqueness problems for embedded minimal annuli. 

In the final portion of our note, we will describe the continuous connection between Lawson's conjecture and Nitsche's conjecture mentioned above. In particular, we outline how the former would imply the latter, if it held that all Jacobi fields of a minimal annulus arise via the action of ambient symmetries. (See Section \ref{sec:fR-surfaces} for the precise statements.) The idea, naturally, is to consider the set $I$ of radii $R \in [0, \pit]$ for which the uniqueness of embedded minimal annuli in $B_R \subset \mathbb{S}^3$ holds. Compactness theory for free boundary minimal surfaces (see Appendix \ref{sec:compactness}) and continuity properties of the (spectral) index suggest $I$ is open. A perturbation argument may show that $I$ is closed; this is the part that requires the above assumption on Jacobi fields.

Here and in our companion paper \cite{NZ25b}, we (generally) take a parametric approach towards minimal surfaces. Several choices of convention arise in this setting - for instance, whether we insist that our surfaces remain in the ball $B_R$ (which we term \textit{constrained}), or whether they may leave the ball, so long as the boundary contacts $\pr B_R$ appropriately. We have striven to give a precise account of these distinctions - we include several such remarks in Section \ref{sec:capillary-notation}, and also observe the equivalence of constraint and embeddedness for rotationally symmetric surfaces in Section \ref{sec:rot-sym-annuli}. 

The connections between settings described above motivate the unified approach we pursue throughout this work. It is natural to wonder whether other continuity frameworks can be developed. For instance, a folklore principle is that closed minimal surfaces in $\mathbb{S}^3$ and self-shrinkers in $\mathbb{R}^3$ share many geometric properties. We speculate that it may be possible to relate them through a continuous family of weighted minimal surfaces (in ellipsoids, for instance). We do not pursue such a theory here, but we nevertheless adopt a unified perspective wherever possible, with a view towards such generalisations.

\begin{remark}
Numerical simulations suggest that there should be embedded rotationally symmetric $(R,\gamma)$-minimal surfaces for any $R,\gamma$. Profile curves of some $(R,\pit)$-minimal surfaces, and their dual $(\pit, \gamma)$-minimal surfaces are shown in Figure \ref{fig:examples}. (Here $\gamma$ will be the smaller of $R$ and $\pi-R$.) These were plotted by taking the appropriate range of the rotationally symmetric solution with a varying parameter $a$ as in \cite{LM23, dO24} (up to sign; in our implementation, moving left to right, $a$ decreases from $\frac{1}{2}$ to $-\frac{1}{2}$.)

Initially, the cap radius increases from $0$ to $\pit$, and both surfaces are radial graphs. The central figures, where $R=\pit$, depict the profile curve of the Clifford torus, which is self-dual. As $a$ decreases further, $R$ reaches a maximum $\bar{R}>\pit$ and then returns to $\pit$; for these surfaces, it seems that the $(\pit,\gamma)$-minimal surfaces are no longer radial graphs, and will collapse to a geodesic segment (so the angle also returns to $\pit$). In particular, uniqueness appears to fail even for rotationally symmetric embedded $(R,\pit)$-minimal annuli, with $R\in (\pit, \bar{R})$, and for embedded $(\pit,\gamma)$-minimal annuli, $\gamma \in (\pi-\bar{R},\pit)$.

\begin{figure}
\begin{subfigure}[b]{0.2\textwidth}
        \includegraphics[scale=0.3]{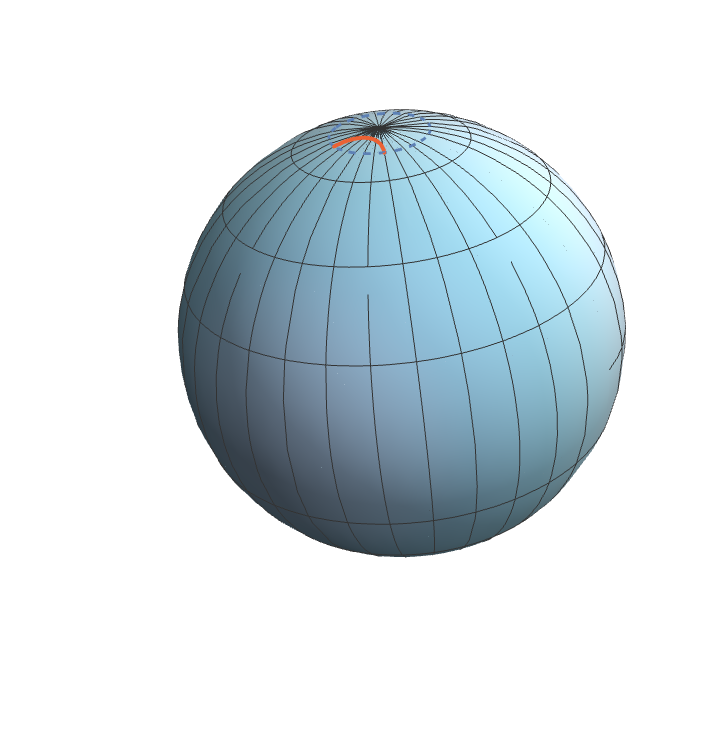}
    \end{subfigure}%
\begin{subfigure}[b]{0.2\textwidth}
        \includegraphics[scale=0.3]{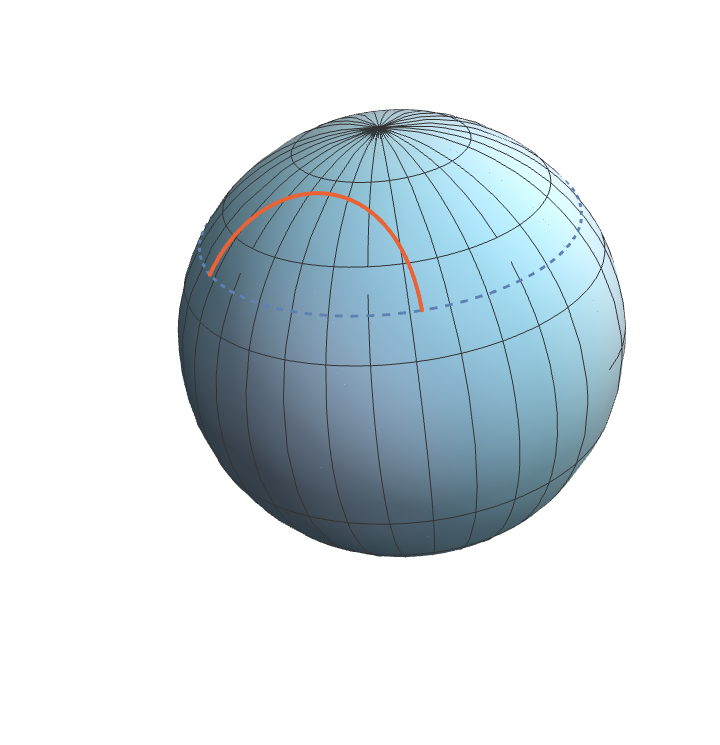}
    \end{subfigure}%
\begin{subfigure}[b]{0.2\textwidth}
        \includegraphics[scale=0.3]{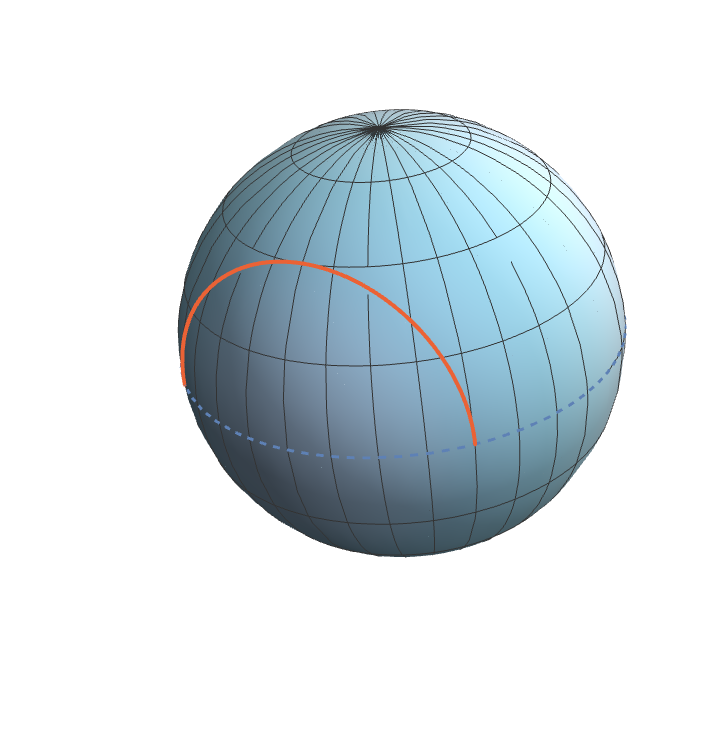}
    \end{subfigure}%
\begin{subfigure}[b]{0.2\textwidth}
        \includegraphics[scale=0.3]{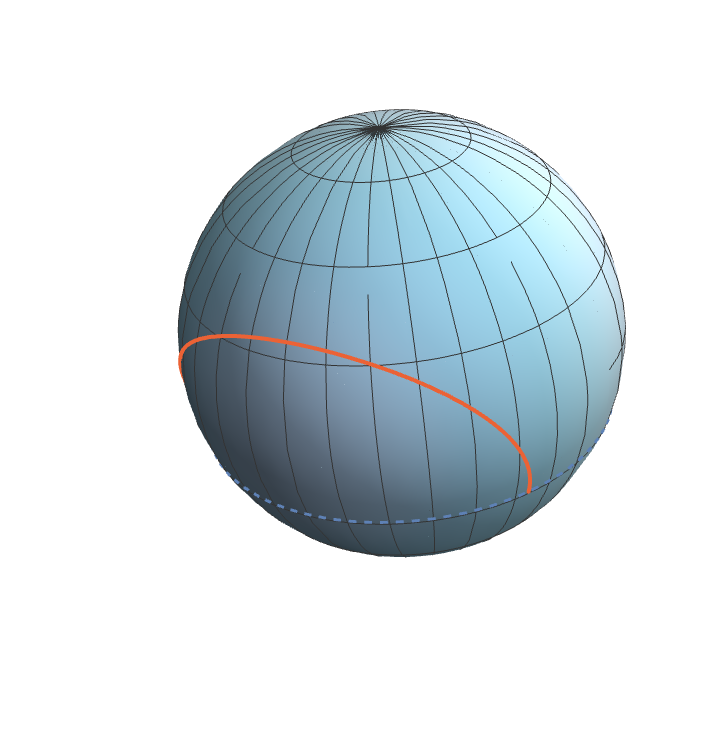}
    \end{subfigure}%
\begin{subfigure}[b]{0.2\textwidth}
        \includegraphics[scale=0.3]{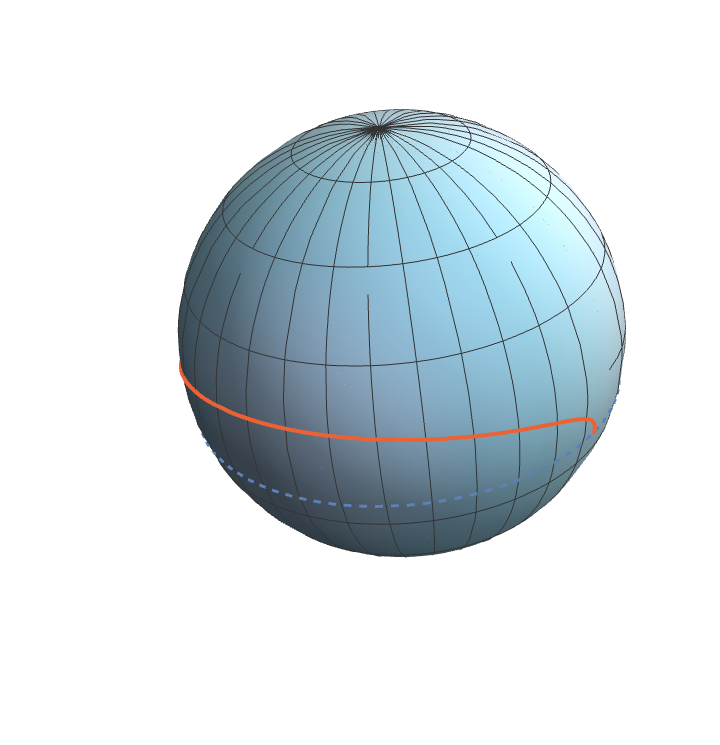}
    \end{subfigure}%
    \\
\begin{subfigure}[b]{0.2\textwidth}
        \includegraphics[scale=0.3]{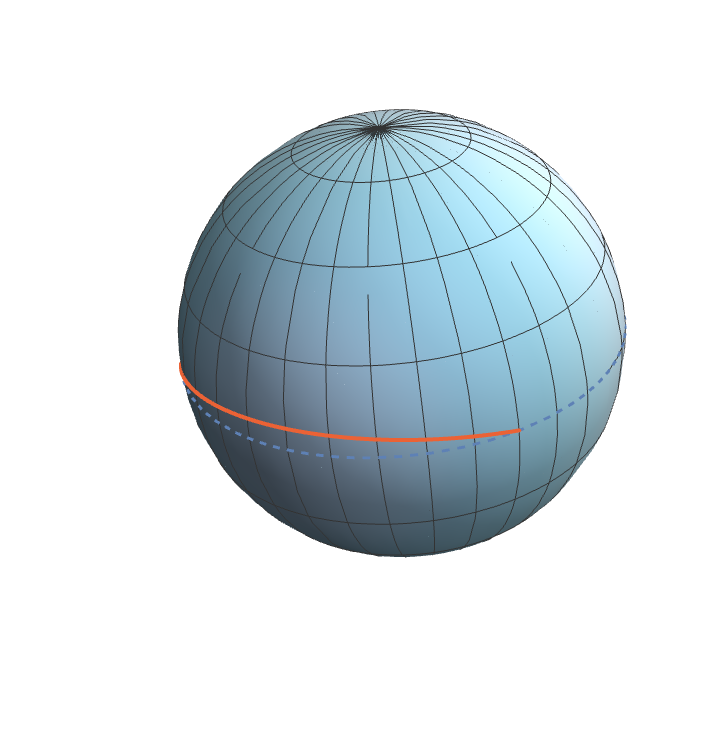}
    \end{subfigure}%
\begin{subfigure}[b]{0.2\textwidth}
        \includegraphics[scale=0.3]{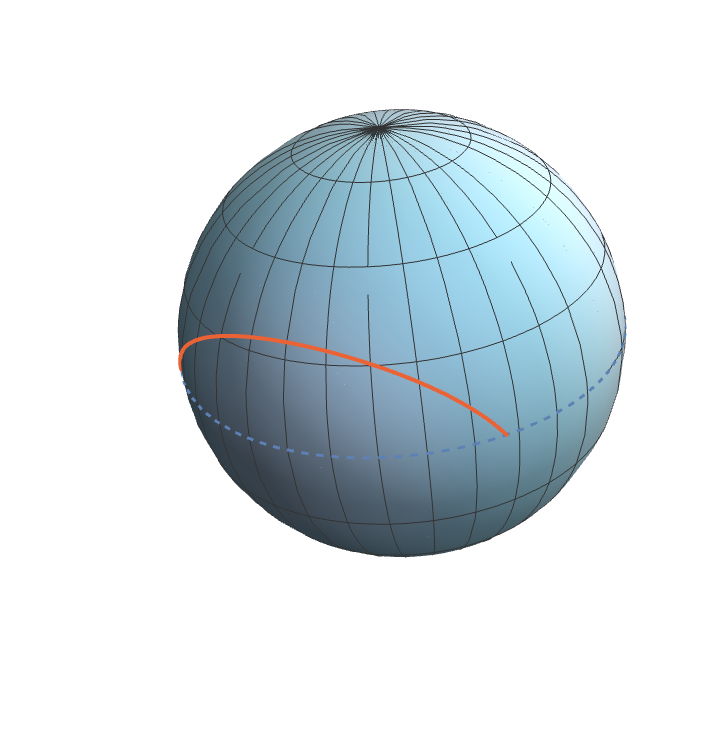}
    \end{subfigure}%
\begin{subfigure}[b]{0.2\textwidth}
        \includegraphics[scale=0.3]{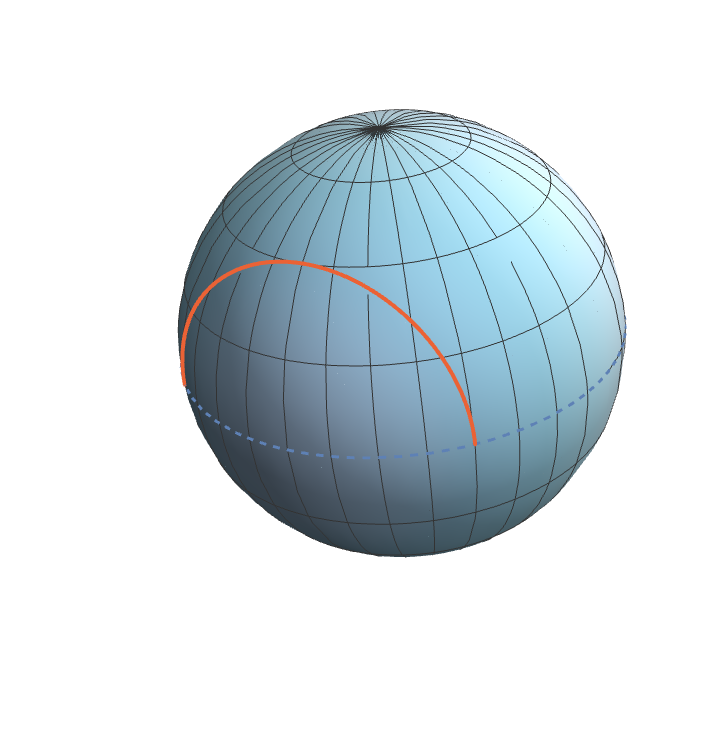}
    \end{subfigure}%
\begin{subfigure}[b]{0.2\textwidth}
        \includegraphics[scale=0.3]{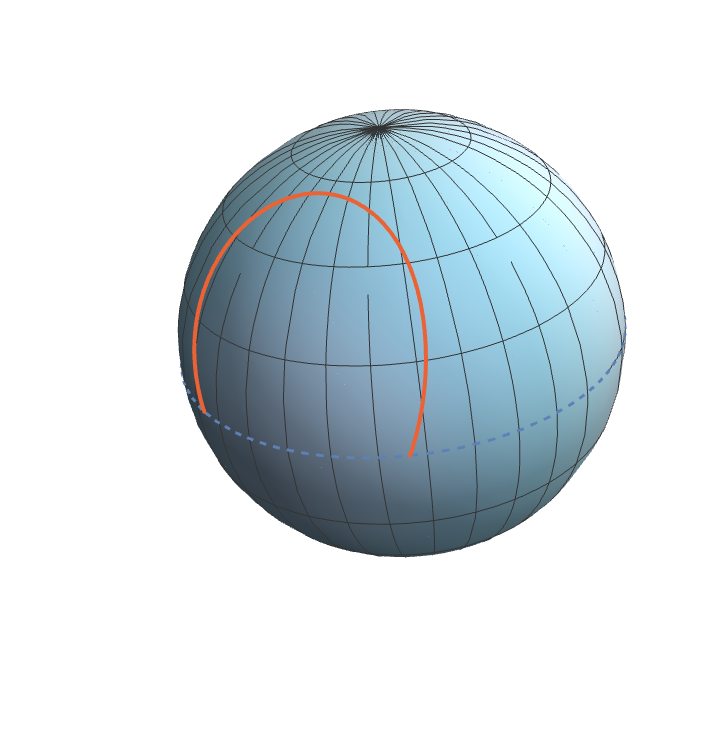}
    \end{subfigure}%
\begin{subfigure}[b]{0.2\textwidth}
        \includegraphics[scale=0.3]{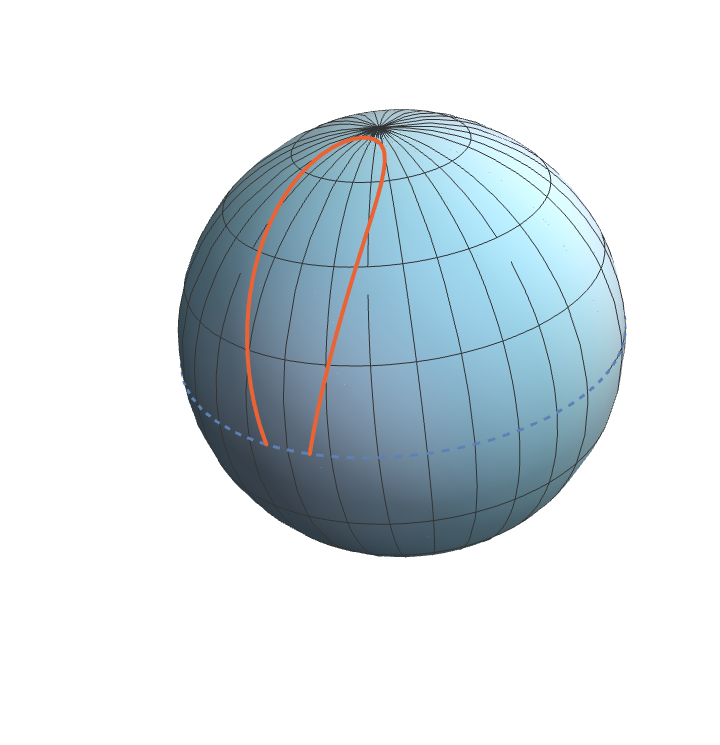}
    \end{subfigure}%
\caption{Profile curves of rotationally symmetric free boundary minimal annuli in spherical caps (top row), and their duals, which are capillary minimal annuli in the hemisphere (bottom row). The caps are all centred at the north pole, and their boundaries represented by dashed lines. From left to right, the radii $R$ are $0.22,0.98, \pit \simeq 1.57, 1.95,1.84$. }
\label{fig:examples}
\end{figure}
\end{remark}

\subsection{Overview of the paper}

We will now give a brief outline of the contents. Section \ref{sec:prelim} sets out our notation as well as some preliminary computations. 

Each main section is intended to be somewhat self-contained, including further background on each topic as well as the relevant proofs. In Section \ref{sec:radial}, we study half-space intersection properties, two-piece properties and topological consequences for free boundary minimal hypersurfaces in our class of warped products. The Hopf differential method is described in Section \ref{sec:hopf}, including our results for annuli, and in higher codimension. In Section \ref{sec:dual}, we recall Lawson's construction of polar dual surfaces, and prove our results for polar duals of capillary minimal annnuli. Finally, in Section \ref{sec:continuity}, we describe the ingredients for a continuity approach to uniqueness questions. 

Appendix \ref{sec:conformal-change} contains some useful formulae for geometric quantities under conformal change. Appendix \ref{sec:compactness} records some compactness results for free boundary minimal surfaces, including some smaller results that are, to the best of our knowledge, new to the literature. Appendices \ref{sec:stable-details}, \ref{sec:ros-details} and \ref{sec:ift-caps} respectively contain additional details on: the stability-based approach to the half-space intersection properties in Section \ref{sec:frankel}; topological observations related to the proof of Theorem \ref{thm:dual-intro}; the implicit function theorem argument in Section \ref{sec:I-closed}.

\subsection*{Acknowledgements}
JZ was supported in part by a Sloan Research Fellowship, and the National Science Foundation under grant DMS-2439945. The authors would like to thank Nick Edelen and Martin Li for insightful conversations about their respective works, as well as Jos\'{e} Espinar for his interest in our work.

\section{Preliminaries}
\label{sec:prelim}

\subsection{Notation}
\label{sec:notation}

In an ambient Riemannian manifold $M^N$, we will use $\rho$ for the distance from a fixed point $o$, and write $B^N_R(o)$ for the geodesic ball of radius $R$ and centre $o$. The decorations $N, o$ will be suppressed if clear from context. We will reserve $r$ for the radial coordinate in Euclidean space $\mathbb{R}^N$, and $\mathbb{B}^n_R$ for the Euclidean ball of radius $R$. If the radius of a Euclidean ball is not specified it will be taken to be 1, that is, $\mathbb{B}^n = \mathbb{B}^n_1$. 

\subsection{Capillary and other surfaces}
\label{sec:capillary-notation}
In this article, we will consider immersed compact submanifolds with boundary $(\Sigma^k,\pr\Sigma) \looparrowright (M^N,S)$, where $M$ is a Riemannian manifold, and $S$ is a complete embedded hypersurface in $M$ with a unit normal $\bar{\eta}$. We emphasise that the metric on $M$ and the normal $\bar{\eta}$ on $S$ are chosen as part of this setup, although their notation may be suppressed. If the barrier is written as the boundary of a region, $S=\pr \Omega$, then the barrier normal $\bar{\eta}$ should be understood to be the \textit{outer} unit normal. We denote by $\eta$ the outer unit conormal of $\pr\Sigma$ in $\Sigma$. We will often conflate $\eta$ with its pushforward $\pr_\eta x$, where $x$ is the immersion map. 

We will $\bar{g}$ to denote the metric on $M$ and $g$ to denote the induced metric on $\Sigma$. Similarly, we use $\bar{\nabla}$ for the ambient connection and $\nabla$ for the induced connection on $\Sigma$.

The (bundle-valued) second fundamental form of $\Sigma$ in $M$ is denoted $\mathbf{A}(X,Y) = (\bar{\nabla}_X Y)^\perp$, where $(\cdot)^\perp$ denotes the projection to the normal bundle of $\Sigma$. The submanifold $\Sigma$ is minimal if $\mathbf{H} = \tr \mathbf{A} =0$. 

If $\Sigma$ is a 2-sided hypersurface, then we denote by $\nu$ the unit normal of $\Sigma$ in $M$. In the hypersurface case, we will work with the scalar-valued second fundamental form, for which our convention is $A(X,Y) = \langle \bar\nabla_X \nu, Y\rangle$. The second fundamental form of the hypersurface $S$ is always $k_S(X,Y) = \langle \bar \nabla_X \bar{\eta}, Y\rangle$.

When $M = \mathbb{R}^{N+1}$ or $\mathbb{S}^N \subset \mathbb{R}^{N+1}$, we let $e_0, \dots, e_N$ denote an orthonormal basis of the Euclidean space and $x_i = \langle x, e_i\rangle$ corresponding coordinates. When we consider a geodesic ball $B_R \subset \mathbb{S}^N$, we assume $B_R$ has centre at $e_0$, unless otherwise indicated.

\begin{definition}
\label{}
Let $(\Sigma^k,\pr\Sigma) \looparrowright (M^{N},S)$ be an immersed submanifold with boundary. We say that $\Sigma$ contacts $S$ with angle $\gamma\in(0,\frac{\pi}{2}]$ along $\pr\Sigma\subset S$ if 
\begin{equation}\label{eq:contact-angle-def}\tag{$\dagger$} \langle \eta,\bar{\eta}\rangle = \sin\gamma>0.
\end{equation}

If $\Sigma$ is a 2-sided hypersurface, then the above condition is equivalent to the existence of a choice of unit normal $\nu$ for which $\langle \nu, \bar{\eta}\rangle = \cos\gamma$. Unless otherwise specified, we take $\nu$ to denote this choice of unit normal.

In the special case $\gamma =\pit$, the condition (\ref{eq:contact-angle-def}) is equivalent to $\eta = \bar{\eta}$ along $\pr\Sigma\subset S$, and we say that $\Sigma$ is a \textit{free boundary} submanifold, and we use the abbreviation `FBMS' for such minimal submanifolds.  

We say that an immersion is \textit{constrained} in $\Omega\subset M$ if the image of (the interior of) $\Sigma$ lies in (the interior of) $\Omega$. 
\end{definition}

\begin{remark}
    Particularly in higher codimension, the constant contact angle condition above may be regarded as analogous to constant \textit{scalar} mean curvature $|\mathbf{H}|$. Whilst most of our focus is on hypersurfaces, in Section \ref{sec:hopf} we discuss a different generalisation which is analogous to the possibly more natural notion of \textit{parallel} mean curvature $\mathbf{H}$.
\end{remark}

\begin{remark}
\label{rmk:constrained}
Elsewhere in the literature, the notion of `free boundary' submanifold often refers to a submanifold $(\Sigma, \pr\Sigma) \looparrowright (\Omega, \pr\Omega)$ that contacts at angle $\pit$ along $\pr\Sigma$ - that is, in our language, a free boundary submanifold \textit{constrained} in $\Omega$, where $S=\pr \Omega$. We emphasise that we do \textit{not} by default assume the constrained property above.

In the case $\Omega = \mathbb{B}^n \subset \mathbb{R}^n$, the unit ball, then there is no ambiguity, as any minimal submanifold $(\Sigma, \pr\Sigma) \looparrowright (\mathbb{R}^n, \pr\mathbb{B}^n)$ must be constrained in $\mathbb{B}^n$, as a simple application of the maximum principle to the function $|x|^2$ shows that it cannot have an interior maximum.

Similarly, consider a minimal submanifold $(\Sigma, \pr\Sigma)\looparrowright (\mathbb{S}^N, \pr B_R)$. If $R\leq \pit$ and $\Sigma$ is \textit{known to lie in the hemisphere} $\{x_0\geq 0\}$, then an application of the (strong) maximum principle to $x_0$ shows that $\Sigma$ must be constrained in $B_R$. 
\end{remark}

\begin{remark}
The term `proper' is often used in the literature as a condition on (what we have termed) \textit{constrained} submanifolds $x:(\Sigma, \pr\Sigma) \looparrowright (\Omega, \pr\Omega)$. This condition is that $x^{-1}(\pr\Omega) = \pr\Sigma$. We find it helpful to clarify that if $\Sigma$ is compact the map $x$ is always proper as a map (on the whole submanifold $\Sigma$). On the other hand, the condition $x^{-1}(\pr\Omega) = \pr\Sigma$ is equivalent to the condition that $x$ restricts to a proper map from the \textit{interior} of $\Sigma$ to the interior of $\Omega$. Note that if $\Sigma$ is constrained and $\Omega$ is mean convex, then by the strong maximum principle we indeed have $x^{-1}(\partial \Omega) = \partial \Sigma$. 

We also remark that if $\Sigma$ is compact, then the map on the compact surface $x :\Sigma \looparrowright M$ is an embedding if and only if it is an injective immersion.

Finally, we emphasise that our definition requires that all the components of $\pr\Sigma$ contact $S$ from the \textit{same} direction. 
(In particular, if $S=\pr\Omega$, then the conormal on each component always points out of $\Omega$.) 

By contrast, de Oliveira \cite{dO24} constructs `free boundary' minimal annuli in $\mathbb{S}^3$ which are allowed to contact the sphere of radius $R$ (orthogonally) from either side.
\end{remark}

The following lemma is a standard calculation but we include a proof to illustrate our conventions.

\begin{lemma}
\label{lem:2ff-diag}
Suppose $(\Sigma,\pr\Sigma)\looparrowright (M,S)$ is a hypersurface that contacts $S$ with angle $\gamma \in (0,\pit]$ along $\pr \Sigma$. Let $\bar{\nu}$ be a unit normal of $\pr\Sigma$ in $S$ for which $\sin\gamma \, \bar{\eta} = \eta + \cos\gamma \, \bar{\nu}$. Then on the boundary, for any vector $v$ tangent to $\pr \Sigma$, we have
\[
A(v,\eta) = -k_S(v,\bar{\nu}).
\]
In particular, if $S$ is totally umbilic then $A(v,\eta)=0$. 
\end{lemma}
\begin{proof}
The capillary condition gives $\bar{\eta} = \cos \gamma \, \nu + \sin \gamma \, \eta$ on $\partial \Sigma$. By the choice of $\bar{\nu}$, it follows that $\bar{\nu} = \frac{1}{\sin \gamma} \,\nu -\cot\gamma \, \bar{\eta}$.
Then
 \[
 \begin{split}
 k_S(v,\bar{\nu}) &= \langle \bar\nabla_v \bar{\eta}\,,\,\bar{\nu}\rangle = \langle \bar\nabla_v \bar{\eta}\,,\, \frac{1}{\sin\gamma}\nu - \cot\gamma \, \bar{\eta}\rangle = \frac{1}{\sin\gamma}\langle \bar\nabla_v \bar{\eta}\, ,\, \nu\rangle
 \\&= \frac{1}{\sin\gamma} \langle \cos\gamma \,\bar \nabla_v \nu + \sin\gamma \,\bar \nabla_v \eta, \nu\rangle = \langle \bar\nabla_v \eta,\nu\rangle = -A(v,\eta). 
 \end{split}
 \]
\end{proof}

\subsection{Variational problems on capillary surfaces}

In this subsection, for the readers' convenience, we reproduce some variational notions for capillary surfaces. For a more thorough treatment, we refer the reader to our companion paper \cite{NZ25b}. 

\subsubsection{Stability}
\label{sec:prelim-stability}

Consider a minimal hypersurface $(\Sigma^n,\pr\Sigma) \looparrowright (M^{n+1},S)$ which contacts $S$ with angle $\gamma$. Let $q=\frac{1}{\sin\gamma}k_S(\bar{\nu},\bar{\nu}) - \cot\gamma\, A(\eta,\eta)$, where $\bar{\nu}$ is a normal of $\pr\Sigma$ in $S$. 

The second variation of wetting energy is given by the index form (cf. \cite{RS97})
\[
\begin{split}
    Q^{\mathrm{A}}(u) &:= \int_\Sigma \left(|\nabla u|^2 - (|A|^2 + \Ric(\nu,\nu))u^2\right) - \int_{\pr \Sigma} qu^2.
    \\& = -\int_\Sigma u\left(\Lap  + (|A|^2 + \Ric(\nu,\nu)\right)u + \int_{\pr \Sigma} u(\pr_\eta-q)u.
\end{split}
\]
The operator $L = L_{\Sigma} = \Lap + |A|^2 + \Ric(\nu,\nu)$ is often referred to as the \textit{stability operator}. 

\subsubsection{Modified Dirichlet form}

Consider a FBMS $(\Sigma^2,\pr\Sigma) \looparrowright (\mathbb{M}^3, \pr B_R)$, where $\mathbb{M}$ is a space form of constant curvature $\kappa$. We introduce the modified Dirichlet form 
\[
\begin{split}
    Q^{\mathrm{S}}(u) &:= \int_\Sigma \left(|\nabla u|^2 - 2 \kappa u^2\right) - \ct_\kappa(R) \int_{\pr \Sigma} u^2.
    \\& = -\int_\Sigma u\left(\Lap  +2\kappa\right)u + \int_{\pr \Sigma} u(\pr_\eta-\ct_\kappa(R))u,
\end{split}
\]
where \[\ct_\kappa(R) = \begin{cases} \sqrt{\kappa} \cot(R\sqrt{\kappa}), & \kappa >0 \\ \frac{1}{R}, &\kappa =0\\ \sqrt{|\kappa|}\coth(r\sqrt{|\kappa|}), &\kappa <0,\end{cases}\] 
is the (constant) curvature of $\partial B_R$ in $\mathbb{M}$. 

\subsubsection{Index and Nullity}

For the quadratic forms $Q= Q^{\mathrm{S}}, Q^{\mathrm{A}}$, there is an associated bilinear form which we continue to denote by $Q$; explicitly $Q(u,v) =\frac{1}{2}( Q(u+ v)-Q(u) -Q(v) )$. The index $\ind(Q)$ is the maximal dimension of a subspace of $C^\infty(\Sigma)$ on which $Q$ is negative definite. We define the nullity $\nul(Q)=\dim\ker Q$, where $\ker Q = \{ u : Q(u,\cdot)\equiv 0\}$. Finally, we let $\ind_0(Q) := \ind(Q)+\nul(Q)$.

\subsection{Conformal minimal surfaces and weighted minimal surfaces}
\label{sec:conf-weight-corres}

Fix an ambient manifold $(M^{n+1},\bar g)$ and a smooth function $\phi$ on $M$. We note that the $n$-dimensional area functional on an immersion $(\Sigma^n, g) \looparrowright (M^{n+1}, \bar{g})$ with respect to the induced conformal metric $\tilde{g}:= e^{2\phi}g$ is precisely the same as the weighted area functional with the weight $e^{-f}$ where
\[
f = -n \phi.
\]
That is, 
\[ 
\int_\Sigma d\mu_{\tilde{g}} = \int_{\Sigma} e^{n\phi} d\mu_g = \int_{\Sigma} e^{-f} d\mu_g.
\]

At various points in this article, we will switch between these perspectives to make certain arguments most efficient. In particular, as the conformal area and weighted area are equivalent, so too are their first and second variations.

For the convenience of the reader, we will make these equivalences explicit as follows. Recall that the undecorated $\nabla,\Lap$ are taken with respect to $g$ on $\Sigma$, while $\bar{\nabla}, \bar{\Lap}$ are taken on $M$. We will decorate operators with $\tilde{g}$ to indicate they are taken with respect to the induced conformal metric. 

We will use well-known formulae for conformal changes, see Appendix \ref{sec:conformal-change}. The weighted (scalar) mean curvature is given by $H^g_f = H^g-\pr_\nu f$. On the other hand, the conformal mean curvature is
\[ H^{\tilde{g}} = e^{-\phi}(H^g + n(\pr_\nu \phi)) = e^{-\phi} H^g_f.\]
In particular, $\Sigma$ is minimal with respect to $\tilde{g}$ (i.e. $H^{\tilde{g}}=0$) if and only if it is $f$-minimal with respect to $g$ (i.e. $H^g_f=0$). 

The stability operator for weighted area is  
\[L_f = \Lap - g(\nabla f, \nabla \cdot) + |A^g|_g^2 + \Ric_M(\nu,\nu) + \bar\nabla^2 f.\]

The lemma below says that $L^g_f$ is essentially equivalent to the stability operator for conformal area, up to conjugation by $e^\phi$. Consequently, 
\[
\int_\Sigma e^\phi u L^{\tilde{g}}(e^\phi u) d\mu_{\tilde{g}} = \int_\Sigma (uL^g_f u) e^{n\phi} d\mu_g = \int_\Sigma (uL^g_f u) e^{-f} d\mu_g,
\]
and one may recover the index form for weighted area from the index form for conformal area.

\begin{lemma}
\label{lem:conf-weight-operator}
Let $(\Sigma, g) \looparrowright (M, \bar{g})$ be a minimal hypersurface as above, $\tilde{g} = e^{2\phi}g$ denote the induced conformal metric, and $f=-n\phi$. Then
\begin{equation}
\label{eq:conf-weight-operator}
    e^\phi L^{\tilde{g}}(e^\phi u) = L^g_f(u).
\end{equation}
\end{lemma}

\begin{proof}

The stability operator in the conformal metric is \[L^{\tilde{g}} = \Lap^{\tilde{g}} + |A^{\tilde{g}}|_{\tilde{g}}^2 + \wt{\Ric}_M( \tilde{\nu}, \tilde{\nu}),\]
where $\wt{\Ric}_M$ is the Ricci curvature of the conformal metric $e^{2\phi} \bar{g}$ on $M$ and $\tilde{\nu}$ is the unit normal of $\Sigma$ with respect to the conformal metric. 
Assuming $H^g = \pr_\nu f = -n\pr_\nu \phi$, and noting that $\tilde{\nu}=e^{-\phi}\nu$, we find
\[
\begin{split}
   e^{2\phi} \wt{\Ric}_M( \tilde{\nu}, \tilde{\nu}) &= \Ric_M(\nu,\nu)-(n-1)(\bar{\nabla}^2\phi(\nu,\nu) - (\pr_\nu \phi)^2) 
   \\& \qquad - (\Lap^g \phi -n(\pr_\nu\phi)^2  + \bar{\nabla}^2\phi(\nu,\nu) +(n-1)(|\nabla^{g}\phi|_g^2 + (\pr_\nu\phi)^2))
   \\& = \Ric_M(\nu,\nu)-\Lap^g \phi -n \bar{\nabla}^2\phi(\nu,\nu) +n(\pr_\nu \phi)^2  -(n-1)|\nabla^{g}\phi|_g^2,
\end{split}
\]
where we have expanded $\bar{\Lap} \phi = \Lap^g \phi -n(\pr_\nu\phi)^2  + \bar{\nabla}^2\phi(\nu,\nu) $ using a basis of $TM$ adapted to $\Sigma$. We also have
\[
e^{2\phi}|A^{\tilde{g}}|_{\tilde{g}}^2 = |A^g|_g^2 - n(\pr_\nu \phi)^2, 
\]
and
\[
\begin{split}
    e^{2\phi}\Lap^{\tilde{g}}(e^\phi u) &= \Lap^g(e^\phi u) + (n-2)g(\nabla^{g}\phi, \nabla^{g}(e^\phi u))
    \\&= e^\phi(\Lap^g u + 2g(\nabla^{g}u,\nabla^{g}\phi) + |\nabla^{g}\phi|_g^2 u + (\Lap^g \phi)u)
    \\&\qquad + (n-2)e^\phi g(\nabla^{g}\phi, \nabla^{g}u) +(n-2) |\nabla^{g}\phi|_g^2 e^{\phi} u
    \\&=e^{\phi} \big(  \Lap^g u + n g(\nabla^{g}u,\nabla^{g}\phi) + (n-1)|\nabla^{g}\phi|_g^2 u + (\Lap^g \phi)u \big).
\end{split}
\]
Combining everything, we see that the quadratic terms in $\bar\nabla \phi$ cancel exactly, leaving
\[
\begin{split}
e^{\phi}L^{\tilde{g}}(e^{\phi}u) &= \Lap^g u + ng(\nabla^{g}u,\nabla^{g}\phi) 
 + \big(|A^g|_g^2 +\Ric_M(\nu,\nu) -n \bar\nabla^2\phi(\nu,\nu)\big)u.
\end{split}
\]
Recalling that $f=-n\phi$ yields the desired result. 

\end{proof}

\subsection{Spherical caps}
\label{sec:cap-model}

Throughout this article, we denote the standard round metric (of constant curvature $1$) on $\mathbb{S}^N$ by $g_1$ and the standard flat metric on $\mathbb{R}^N$ by $\delta$. Recall that $B_R = B_R(o)\subset M$ denotes the geodesic ball in $M$, with $\mathbb{B}_R$ reserved for Euclidean balls. Unless otherwise indicated, for balls $B_R\subset \mathbb{S}^N\hookrightarrow\mathbb{R}^{N+1}$ we take the centre to be $o=e_0$, and denote by $\rho$ the (spherical) distance from $o$. Then, in the notation of Section \ref{sec:capillary-notation}, $(M^3, S) = (\mathbb{S}^3, \partial B_R)$ satisfies $\bar{\eta} = \pr_\rho$ on $\partial B_R$ and $k = k^{\partial B_R} = \cot R \, g_1 $.

This project was motivated, in part, by an attempt to study the free boundary minimal surface problem in $\mathbb{B}^3$ as part of a continuous family containing the minimal surface problem in $\mathbb{S}^3$. Let $\kappa_R := \sin^2 R$ and let $\mathbb{M}^N_{\kappa_R}$ denote the space form of curvature $\kappa_R$. Accordingly, for $R>0$ we may consider the pair \[(\mathbb{M}^N_{\kappa_R}, \pr B_{R\csc R}).\] This is precisely the rescaling of the pair $(\mathbb{S}^N, \pr B_R)$ such that $\partial B_{R \csc R}$ is isometric to the unit $(N-1)$-sphere. Note that $\kappa_R \to 0$ as $R \to 0$ and $R \csc R \to 1$ so that in the limit one obtains $(\mathbb{R}^N, \partial \mathbb{B}_1)$.

Each of the settings $(\mathbb{M}^N_{\kappa_R} , \pr B_{R\csc R})$ has a conformal model by the hemisphere. To be precise, for $R>0$, the sphere $\mathbb{M}^N_{\kappa_R}$ is isometric to a Riemannian manifold \[M^N_R:=(\mathbb{S}^N, h_R^2 g_1),\qquad h_R = \frac{1}{1+\cos R\cos \rho},\]
and under this identification $B_{R\csc R} \subset \mathbb{M}^N_{\kappa_R}$ is mapped to $B_\pit \subset \mathbb{S}^N$, where as usual $\rho$ is the spherical distance on $\mathbb{S}^N$. (The precise constructions are fairly standard, and deferred to our companion paper \cite{NZ25b} - particularly Appendix B - where the reader may find a more comprehensive review of the relevant conformal transformations.) For $R=0$, stereographic projection gives that $\mathbb{R}^N$ is isometric to $(\mathbb{S}^N \setminus \{-o\}, h_0^2 g_1)$, with $\mathbb{B}^N$ being mapped to $B_\pit$. 

As per Section \ref{sec:conf-weight-corres} above, surfaces are minimal with respect to $h_R^2 g_1$ if and only if they are weighted minimal with respect to $g_1$, with weight \[f_R := -n\log h_R = n\log(1+\cos R\cos \rho). \]

To summarise, we will have the following equivalent problems:

\begin{enumerate}[(A)]
\item FBMS in $(\mathbb{S}^N, \pr B_R)$ for $R>0$, or in $(\mathbb{R}^N, \pr\mathbb{B}_1)$ for $R=0$.
\item FBMS in $(\mathbb{M}^N_{\kappa_R} , \pr B_{R\csc R})$, where $\mathbb{M}^N_{\kappa_R}$ is the space form of curvature $\kappa_R= \sin^2 R$.
\item Weighted ($f_R=n\log(1+\cos R\cos \rho)$) FBMS in $(\mathbb{S}^N, \pr B_\pit)$.
\end{enumerate}

As the equivalences are all conformal, we note that the above problems give equivalent \textit{two}-parameter families if one expands the contact angle condition to $\gamma\in(0,\pit]$. 

For results at a given $R>0$, we find it most familiar to present our results in terms of non-rescaled setting (A). When studying these problems as part of a continuous family, however, we find it more thematic to work in either of the rescaled settings (B, C), particularly as they include the limiting case $R\to 0$ in a direct sense.

\subsection{$(R,\gamma)$-minimal surfaces}

We make the following definition for notational convenience:

\begin{definition}\label{def:Rgamma-cap}
Suppose $\Sigma^n$ is a compact, smooth manifold with boundary. Let $R \in (0, \pi)$ and  $\gamma \in (0,\frac{\pi}{2}]$.

 An $(R,\gamma)$-minimal immersion is a 2-sided proper minimal immersion $x : (\Sigma, \partial \Sigma) \looparrowright (\mathbb{S}^{n+1}, \partial B_R)$ which contacts $\partial B_R$ with constant angle $\gamma$. This means (as $\bar{\eta}=\pr_\rho$) that for the outward pointing unit conormal $\eta$ and some choice of unit normal $\nu$ along $\Sigma$, one has 
\begin{equation}
    \langle \nu, \partial_\rho \rangle = \cos \gamma, \qquad \qquad \langle \eta, \partial_\rho \rangle = \sin \gamma.
\end{equation}

An $(R, \gamma)$-minimal embedding is an $(R, \gamma)$-minimal immersion which is additionally an embedding. 
\end{definition}

\begin{remark}
For much of this article, we will focus on the $n=2$ case. For brevity, we will often refer to surfaces $\Sigma^2$ as above to be $(R,\gamma)$-minimal surfaces. 
\end{remark}

We compute some basic boundary relations for $(R,\gamma)$-minimal immersions. We assume $B_R\subset \mathbb{S}^3$ is centred at $o=e_0$. 

\begin{lemma}\label{lem:boundary-Rgamma-cap}
Suppose $x : (\Sigma^n, \partial \Sigma) \looparrowright (\mathbb{S}^{n+1}, \partial B_R)$ is an $(R, \gamma)$-minimal immersion. Then on $\partial \Sigma$, 
\begin{align}
\langle x, e_0\rangle &= \cos R, \\
\pr_\eta \langle x, e_0\rangle &= -\sin R \sin\gamma, \\
\langle \nu, e_0\rangle &= - \sin R \cos \gamma, \\
\partial_{\eta} \langle \nu, e_0 \rangle & = \tan \gamma \, A(\eta, \eta) \langle \nu, e_0\rangle. 
\end{align}
If $R = \frac{\pi}{2}$, and $a \in \mathbb{R}^{n+1}$ is a unit vector orthogonal to $e_0$, then 
\begin{equation}
\partial_{\eta} \langle \nu, a \rangle = -\cot \gamma \, A(\eta, \eta) \langle \nu, a \rangle.
\end{equation}
\end{lemma}
\begin{proof}
The first equation is clear. Now, on $\partial B_R$, the tangential projection of $e_0$ to $T\mathbb{S}^n$ is given by $-\sin R \, \partial_\rho$. 

The contact angle condition (and our standard choice of normal) gives that $\partial_\rho = \cos \gamma\,  \nu +  \sin \gamma \, \eta $ on $\pr\Sigma$. 
This immediately implies 
\[ \pr_\eta \langle x,e_0\rangle = \langle \eta,e_0\rangle = -\sin R \langle \eta, \pr_\rho\rangle = -\sin R \sin \gamma\]
and 
\[
\langle \nu, e_0 \rangle = -\sin R \langle \nu, \partial_\rho \rangle = - \sin R \cos \gamma. 
\]
Moreover, the tangential projection of $\partial_\rho$ to $T\Sigma$ is given, along $\pr\Sigma$, by $\sin \gamma \, \eta$. Therefore 
\[
\partial_\eta \langle \nu, e_0 \rangle = \langle A(\eta), e_0 \rangle = -\sin R \langle A(\eta), \partial_\rho^{\top, \Sigma}\rangle = - \sin R \sin \gamma A(\eta, \eta)  = \tan \gamma A(\eta, \eta) \langle \nu, e_0 \rangle. 
\]

Finally, suppose that $a$ is an ambient unit vector orthogonal to $e_0$, and that $R = \frac{\pi}{2}$. When the boundary of $\Sigma$ lies in $\partial B_{\frac{\pi}{2}}$, we have $-e_0 = \partial_\rho =\cos \gamma \, \nu +  \sin \gamma \, \eta $. Hence, the orthogonality of $a$ to $e_0$ implies $\langle a, \eta \rangle = -\cot\gamma \langle a, \nu \rangle$. Thus
\[
\partial_\eta \langle \nu, a \rangle = \langle A(\eta), a \rangle = A(\eta, \eta) \langle \eta, a \rangle  = - \cot \gamma \, A(\eta, \eta) \langle \nu, a\rangle.
\]
\end{proof}

\section{Intersection and two-piece properties in radially conformal manifolds}
\label{sec:radial}

In this section, we generalise half-space intersection and two-piece properties for minimal hypersurfaces in certain manifolds which admit a radial symmetry. These methods have been pursued in the literature, including in \cite{Ros95,assimos2020intersection, LM21, NZ24}; our goals for this section are to present the results in a more comprehensive and unified manner. In particular, we generalise our results in \cite{NZ24} to a class of rotationally symmetric ambient manifolds.  

As the rest of this article is concerned with minimal surfaces in spherical caps, we summarise the key results in that setting as follows: 

\begin{proposition}
\label{prop:radial-sphere}
    Let $R\leq \pit$ and suppose $(\Sigma,\pr\Sigma)\hookrightarrow (B_R, \pr B_R)$ is an embedded free boundary minimal surface constrained in $B_R\subset \mathbb{S}^3$. Let $\rho$ be the distance from the centre $o$ of $B_R$ and consider any $a \in\mathbb{S}^3$ such that the boundary $\pr B_{\pit}(a)$ of the hemisphere $B_\pit(a)$ passes through $o$. 
    
    \begin{enumerate}[(a)]
    \item If $\Sigma'$ is another such surface, then $\Sigma,\Sigma'$ intersect in the closed hemisphere $\overline{B_\pit(a)}$; that is, $\Sigma\cap\Sigma' \cap \overline{B_\pit(a)}\neq \emptyset$. 
    \item $\Sigma \setminus \pr B_\pit(a)$ consists of two connected components. 
    \item Suppose that $\Sigma$ has genus 0 and $b$ boundary components. If $b=1$ then $\Sigma$ is totally geodesic. If $b\geq 2$, then $o\notin\Sigma$ and $g_1(\pr_\rho,\nu)$ does not vanish on the interior of $\Sigma$.
	\end{enumerate}
\end{proposition}

\subsection{A class of radially conformal manifolds}
\label{sec:radial-setup}

We now describe the more general setup for this section. We consider compact ambient manifolds $(M^{n+1},\bar{g})$ which are conformal to either the Euclidean unit ball $\mathbb{B}^{n+1}$ or the round sphere $\mathbb{S}^{n+1}$ by a radial conformal factor (about $o\in M$). As usual let $B_R=B_R(o)$ be the geodesic ball in $M$ (with respect to $\bar{g}$).

We also assume that 

\begin{enumerate}[(A)]
\item $\Ric_M$ is either (strictly) largest in the radial direction, or is constant. 
\end{enumerate}

In the present setting, a constant Ricci tensor implies $(M^{n+1}, \bar{g})$ is a isometric to a geodesic ball in a space form. 

There are two convenient models for such manifolds $M$: a conformal model and a warped product model. In the first model, there is $\bar{\rho}$ so that the (closure of) the geodesic ball $B_{\bar{\rho}}$ is precisely $M$. Then $(B_{\bar{\rho}}, \bar g) \simeq (\mathbb{B}_{\bar{r}}^{n+1}, e^{2\phi(r)}\delta)$, where $r$ is the radial coordinate on $\mathbb{B}_{\bar{r}}^{n+1}$. In terms of $\phi$, the structure assumption translates (see Remark \ref{rem:ricci}) to 

\begin{enumerate}[(A$'$)]
\item $\left(\frac{\phi'}{r}\right)' - r\left(\frac{\phi'}{r}\right)^2 < 0,$ or $\left(\frac{\phi'}{r}\right)' - r\left(\frac{\phi'}{r}\right)^2 \equiv 0$.
\end{enumerate}

By the standard transformation, $d\rho = e^{\phi(r)}dr$ and $h(\rho) = re^{\phi(r)}$, we may see that $M$ has the warped product structure $\bar{g}=d\rho^2 + h(\rho)^2 g_1$, where by slight abuse of notation $g_1$ denotes the round metric on $\mathbb{S}^n$. In terms of $h$, the structure assumption translates to

\begin{enumerate}[(A$''$)]
\item $-\frac{h''}{h} > \frac{1-(h')^2}{h^2},$ or $-\frac{h''}{h} \equiv \frac{1-(h')^2}{h^2}$.
\end{enumerate}
Note that (A$''$) has the opposite sign to assumption (H4) in \cite{Bre13}. 

Also fix
\[
R\in(0,\bar{\rho}], \qquad S := \pr B_R \subset M.
\]

We will be interested in 2-sided embedded minimal hypersurfaces $\Sigma^n$ in $M$, which contact $S$ at angle $\pit$ along $\pr\Sigma \subset S$. If $\Sigma$ is constrained in $B_R$, it is equivalent to simply take $M=B_R$, so that $S=\pr M$ and $R=\bar{\rho}$. 

\begin{remark}
Our aim in this section is to present a reasonably general setting for which the intersection and two-piece properties hold. For instance, the equality case of structure condition (A) precisely corresponds to space forms of constant curvature, which are represented (up to scaling) by warping factors $h = \sin \rho, \rho, \sinh\rho$, respectively. Another setting covered by condition (A) is the case of Gaussian minimal hypersurfaces, or equivalently self-shrinkers in $\mathbb{R}^{n+1}$. In this case, we have $\phi(r) = -\frac{r^2}{4n}$, whence $\frac{\phi'}{r} = -\frac{1}{2n}$ and so $\left(\frac{\phi'}{r}\right)' - r\left(\frac{\phi'}{r}\right)^2 =- \frac{r}{4n^2}<0$.

We expect our result to hold for still more general settings - for instance, if the strictness in condition (A) is relaxed, but $(M,\bar g)$ is real analytic - but we have neglected these edge cases in favour of a relatively unified approach. 
\end{remark}

\subsection{Intersection properties}

\label{sec:frankel}

We next describe how the methods of \cite{NZ24} extend to embedded minimal hypersurfaces $\Sigma$, where $\Sigma, M$ and $S=\pr B_R$ are as in Section \ref{sec:radial-setup}. In \cite{NZ24}, we extended two of the classical proofs of the Frankel intersection property: one based on the classification of stable minimal hypersurfaces via the stability inequality, and another using a length-variation argument that relies on the non-negativity of the Bakry-\'{E}mery-Ricci curvature after a certain conformal change. The former method is more geometric, while the latter more analytic. In what follows, we pursue the length-variation argument, but the reader interested in the stability approach will find some details at the end of this section and in Appendix \ref{sec:stable-details}.

It is helpful to recall that the normal on $S$ (with respect to $\bar{g}$) is $\bar{\eta}=\pr_\rho$.

In this subsection, we will use the conformal model $(M, \bar{g}) \simeq (\mathbb{B}_{\bar{r}}^{n+1}, e^{2\phi(r)}\delta)$. We remark that when $\pr M=\emptyset$, correspondingly $\bar{r}=\infty$, our computations will still be justified; alternatively, one may instead perform the analogous computations using $\mathbb{S}^{n+1}$ as the conformal model. These alternative arguments are omitted for the sake of brevity.

For $a\in \mathbb{S}^n\subset\mathbb{R}^{n+1}$, we define $x_a:= \langle x,a\rangle$, and subsets 
\begin{align*}
    \mathcal{H}_a &:= \{x_a > 0\} \subset \mathbb{R}^{n+1}, \\
    D_a&:= \pr \mathcal{H}_a, \\
    S_a &: = \partial M \cap \mathcal{H}_a
\end{align*}
By abuse of notation, we will identify each of these with their lifts to $M$. Further, let 
\[
N_a := M \cap \mathcal{H}_a.
\]
Then $M = N_{-a} \cup D_a \cup N_a$, for any $a \in \mathbb{S}^n$. 

We now analyse certain test functions. Consider $u_a: M\to \mathbb{R}$ defined (in terms of the conformal model) by
\begin{equation} \label{eq:ua}
u_a:=  e^{\phi}x_a.
\end{equation} 
Likewise, denote $\nu_{\delta} = e^\phi\nu_{\bar{g}}$ the unit normal of $\Sigma$ with respect to $\delta$, and denote $\pr_r$ the radial unit vector on $\mathbb{R}^{n+1}$. In this section we will write $\langle \cdot,\cdot\rangle$ for the pairing with respect to $\delta$.

\begin{lemma}\label{lem:BER}
Let $(M^{n+1},\bar{g})$ be as in Section \ref{sec:radial-setup}, and consider $u_a$ as above. On $N_a$ define a metric $\tilde{g} := u_a^{-2} \bar{g}$ and $f = -n \log u_a $. Then 
\begin{enumerate}[(i)]
\item $(N_a, \tilde{g})$ is isometric to a hyperbolic half-space with totally geodesic boundary given by $S_a$.
\item The $(-n)$-Bakry-\'Emery-Ricci tensor of $(N_a, \tilde{g}, f)$ is given by 
\[
\Ric^{-n}_{\tilde{g}, f} = - n \Big(\phi'' - \frac{\phi'}{r} - (\phi')^2\Big) dr \otimes dr. 
\]
\item Suppose $\Sigma$ is a compact, properly embedded minimal hypersurface contained in $\overline{N_a}$, such that $\pr \Sigma \subset D_a \cup S_a$, and that $\Sigma$ has contact angle $\pit$ along $\pr\Sigma \cap S_a$. 

Then in $(N_a, \tilde{g})$, after the conformal change, $\Sigma$ is a complete, properly embedded $f$-minimal hypersurface with contact angle $\pit$ along $\partial \Sigma \cap S_a$. 
\end{enumerate}
\end{lemma}

\begin{proof}

Identify $(M, \bar{g})$ with its conformal model. For (i), we simply observe that 
\[
\tilde{g} := u_a^{-2} \bar{g} = x_a^{-2} \delta.
\]
It follows that $(N_a, \tilde{g}) = (\mathbb{B}^{n+1}_{\bar{r}} \cap \{x_a > 0\}, x_a^{-2} \delta)$ is isometric to a hyperbolic half-space with totally geodesic boundary. In particular, $(N_a, \tilde{g})$ is geodesically convex. 

Let $\psi = - \log x_a$ so that $\tilde{g} = e^{2\psi} \delta$. Note from $e^{-f} = u_a^n = e^{n\phi} x_a^n$, we obtain $f = n \psi - n\phi$. For (ii), we begin by noting that $\Ric_{\tilde{g}} = -n \tilde{g}$.  On the other hand, using Appendix \ref{sec:conformal-change}, 
\[
\nabla^2_{\tilde{g}} f = \nabla_\delta^2 f - df \otimes d\psi  -d\psi \otimes df + \delta (df, d\psi) \delta. 
\]
Now $d\psi = -\frac{1}{x_a} a$ and $\nabla^2_{\delta} \psi =  \frac{1}{x_a^2} a \otimes a $. Similarly, $d\phi = \phi' \partial_r$ and $\nabla^2_{\delta} \phi = \big(\phi'' - \frac{\phi'}{r}\big)dr \otimes dr + \frac{\phi'}{r} \delta$. One may compute that
\begin{align*}
    \nabla^2_\delta f &= \frac{n}{x_a^2} a \otimes a -n \big(\phi'' - \frac{\phi'}{r}\big)dr \otimes dr -n \frac{\phi'}{r} \delta,
\end{align*}
and 
\begin{align*}
    \delta(df, d\psi)\delta  &= n |d\psi|^2_{\delta} \,\delta + \frac{n}{x_a} \phi' \delta(\partial_r,  a) \delta \\
    &= \frac{n}{x_a^2} \delta  + n \frac{\phi'}{r}\delta \\
    &= n \tilde{g} + n \frac{\phi'}{r} \delta, 
\end{align*}
Finally, we have 
\begin{align*}
\frac{1}{n} df \otimes df 
& = -\frac{n}{x_a^2} a \otimes a + df \otimes d\psi + d\psi \otimes df + n(\phi')^2 dr \otimes dr. 
\end{align*}
Putting everything together and cancelling terms yields (ii):
\begin{align*}
\Ric_{\tilde{g}, f}^{-n} =  \Ric_{\tilde g}+  \nabla^2_{\tilde{g}} f  + \frac{1}{n} df \otimes df  
 & = - n \Big( \phi'' - \frac{\phi'}{r} - (\phi')^2 \Big) dr \otimes dr. 
\end{align*}
Finally, for (iii) observe that the area measures induced on $\Sigma$ satisfy
\[
d\mu_{\Sigma}(\bar{g}) = d\mu_{\Sigma}(u_a^{-2} \tilde{g}) = u_a^{-n} d\mu_{\Sigma}(\tilde{g})= e^{-f} d\mu_{\Sigma} (\tilde{g}). 
\]
Hence, minimality with respect to the metric $\bar{g}$ corresponds to $f$-minimality with respect to the metric $\tilde{g}$. Under the conformal change, $\Sigma$ continues to contact $S_a$ with angle $\pit$. Moreover, as the boundary $D_a$ is infinitely far (in the $\tilde{g}$-metric) from any interior point of $N_a$, the hypersurface $\Sigma$ (strictly, $\Sigma \setminus D_a$ is complete.)
\end{proof}

For the Frankel property, we need a classification of totally geodesic hypersurfaces in $M$. 

\begin{lemma}\label{lem:geodesic-classify}
Let $(M^{n+1},\bar g)$ and $S=\pr M = \pr B_R$ be as in Section \ref{sec:radial-setup}, satisfying condition (A). If $\Sigma$ is totally geodesic, then $\Sigma$ is either an equator, that is $\Sigma = D_b \cap \overline{N_a}$ for some $b \in \mathbb{S}^n$, or else $\Sigma = \partial B_{\hat{r}}$ for the unique value of $\hat{r} \in (0, \bar{r}]$  such that $\phi'(\hat{r}) + \frac{1}{\hat{r}} = 0$. 
\end{lemma}

\begin{proof}
    As the metric on $\Sigma$ is $g = e^{2\phi} \delta |_{\Sigma}$, the conformal change formulas in Appendix \ref{sec:conformal-change} imply
    \[
    A^g = e^\phi(A^\delta + (\partial_\nu \phi) \delta), \qquad H^g = e^{-\phi}(H^\delta + n (\partial_\nu \phi)).
    \]
    Thus $A^g \equiv 0$ if and only if $A^\delta = \frac{1}{n} H^\delta \delta$. In other words,  $\Sigma$ is totally geodesic with respect to $\bar{g}$ if and only if $\Sigma$ is totally umbilic with respect to $\delta$. In fact, as one additionally has $H^\delta= -n \phi' \langle \partial_r, \nu \rangle$, we conclude that $\Sigma$ must either be a hyperplane through the origin or else a centred sphere. 

    The curvature $\partial B_r \subset M$ is constant and given by $e^{\phi(r)} (\phi'(r) + r^{-1})$. Consequently, $\partial B_r$ is totally geodesic if and only if $y(r) := r\phi'(r)  = -1$. Note that $y(0) = 0$ and  $y' =r \phi'' + \phi'$.  Expanding condition $(\mathrm{A}')$ gives 
    \[
    0 >  \frac{1}{r^2} \big(r\phi'' - \phi' - r (\phi')^2) = \frac{1}{r^2} \Big(y' - \frac{y(y+2)}{r}\Big). 
    \]
    It follows from the differential inequality that if for some $\hat{r}$ one has $y(\hat{r})= -1$, then $y'(r) < 0$ for all $r \in [\hat{r}, \bar{r}]$. Consequently, $\partial B_r$ is totally geodesic for at most one value of $r \in (0, \bar{r}]$. 
\end{proof}

We remark that following seems to be the only argument in this section that requires the minimal hypersurfaces to be \textit{constrained} in the (half-)ball. 

\begin{theorem}[Half-space Frankel]
\label{thm:frankel}
Let $(M^{n+1},\bar g)$ and $S=\pr M = \pr B_R$ be as in Section \ref{sec:radial-setup}, satisfying condition (A).

Let $\Sigma_i$, $i=0,1$, be compact, properly embedded minimal hypersurfaces contained in the half-ball $B_R\cap \overline{N_a}$, such that $\pr \Sigma_i \subset D_a \cup (S\cap \mathcal{H}_a)$, and that $\Sigma_i$ has contact angle $\pit$ along $\pr\Sigma_i \cap S$. 

If the boundaries do not intersect, $\pr\Sigma_0 \cap \pr \Sigma_1=\emptyset$, then the minimal hypersurfaces must intersect in the interior, $\Sigma_0 \cap \Sigma_1\neq \emptyset$. 
\end{theorem}
\begin{proof}
The proof is essentially the same as the one given in Section 6.4 of \cite{NZ25a}, which proves the theorem in the special case that the condition in ($\mathrm{A}'$) holds with equality. 

Under the hypotheses of the theorem, and in view of Lemma \ref{lem:BER}, after conformal change, we may regard $\Sigma_0, \Sigma_1$ as complete, properly embedded, $f$-minimal hypersurfaces in hyperbolic half-space $(\mathbb{H}^{n+1}_+, g_{\mathrm{hyp}})$, each meeting the totally geodesic boundary $\partial \mathbb{H}^{n+1}_+$ of the half-space at angle $\pit$. Let $\Omega$ denote the region between $\Sigma_0$ and $\Sigma_1$, and consider the sum of the distance functions 
\[
u = d_{\Sigma_0} + d_{\Sigma_1}
\]
taken with respect to the hyperbolic metric. The assumption that $\partial \Sigma_0$ and $\partial \Sigma_1$ were disjoint before the conformal change implies that $u(x) \to \infty$ as $x \to \infty$ after the conformal change. As $\Sigma_0$ and $\Sigma_1$ do not intersect in the interior nor along their free boundaries, $u$ must attain a positive minimum somewhere in $\overline{\Omega}$. 

Now, Proposition 25 of \cite{NZ25a} (applied to $(\mathbb{H}^{n+1}, g_{\mathrm{hyp}}, f)$ with $N = \mathbb{H}^{n+1}_+$) asserts that $u$ must be constant, and hence both $\Sigma_0, \Sigma_1$ are totally geodesic with respect to the conformal metric $e^{-\frac{2}{n}f} g_{\mathrm{hyp}} \cong \bar{g}$, as in Lemma \ref{lem:BER}.  But if $\Sigma_0, \Sigma_1$ are totally geodesic in $(M^{n+1}, \bar{g})$, then they must intersect by the classification in Lemma \ref{lem:geodesic-classify}.  This is a contradiction and completes the proof.  
\end{proof}

From Theorem \ref{thm:frankel} we may deduce the following two-piece property:

\begin{corollary}[Two-piece property]
\label{cor:two-piece}
Let $(M,\bar{g})$ be as in Section \ref{sec:radial-setup}, satisfying condition (A). Let $(\Sigma,\pr\Sigma) \hookrightarrow (B_R, \pr B_R)$ be a connected, compact, embedded free boundary minimal hypersurface constrained in $B_R\subset M$.

Fix any $a\in \mathbb{R}^{n+1}$. Then either $\Sigma$ is the totally geodesic disk $D_a \cap M$, or $\Sigma\setminus D_a$ consists of precisely two connected components. 
\end{corollary}
\begin{proof}
We claim that $\Sigma \cap \overline{N_a}$ is nonempty. This follows trivially from the half-space Frankel Theorem \ref{thm:frankel}, but one can also apply the maximum principle using the mean-convex foliation provided by Lemma \ref{lem:mc-foliation} as barriers.

Now $\Sigma \cap \overline{N_a}$ has connected components $\Sigma^{(i)}$, $i=1,\cdots k$. If $k>1$, then by Theorem \ref{thm:frankel}, $\pr\Sigma^{(1)}, \pr\Sigma^{(2)}$ must intersect, say at $p$. Since $\Sigma$ was embedded, the intersection must be smooth and tangential to $D_a$ at $p$. But this is only possible if $\Sigma$ lies to one side of $D_a$ near $p$, and it follows from the strong maximum principle that $\Sigma$ must be contained in $D_a$. 
\end{proof}

One may alternatively use a classification of stable hypersurfaces as in \cite{NZ24} to deduce the half-space intersection property. This argument, which uses a solution of a Plateau-type problem, seems to require the additional assumption that $S=\pr B_R$ is convex in $M$. In the remainder of this subsection we will discuss this classification of stable hypersurfaces, and then we will include the alternative half-space intersection argument in Appendix \ref{sec:stable-details} as it may be more geometrically familiar than the proof given above.

The following is the relevant notion of stability adapted to the half-space setting. Suppose that $\Sigma$ is a compact, embedded minimal hypersurface contained in the half-space $\overline{N_a}$, so that $\pr \Sigma \subset \pr \overline{N_a} = D_a \cup (S\cap \mathcal{H}_a)$, and further suppose that $\Sigma$ has contact angle $\pit$ along $\pr \Sigma\cap S$. One may read this as $\Sigma$ having free boundary along the round portion $S$ and fixed boundary along the flat portion $D_a$. As in Section \ref{sec:prelim-stability}, the index form is

\[ Q^{\mathrm{A}}(u) := -\int_\Sigma uL u\, d\mu_g + \int_{\pr\Sigma \cap S} u(\pr_\eta - k_S(\bar \nu,\bar \nu)) u\, d\sigma_g,\]
where recall $L = \Lap + |A_\Sigma|^2 + \Ric_M(\nu,\nu)$, $\nu = \nu_{\bar{g}}$ is a unit normal on $\Sigma$ in $M$, $\bar{\nu} = \bar{\nu}_{\bar{g}}$ is the unit normal of $\partial \Sigma$ in $S$, and $k_S$ is the second fundamental form of $S$. We will drop the subscripts where clear from context. As in \cite{NZ24}, we say that $\Sigma$ is \textit{stable for one-sided variations fixing $D_a$} if $Q^{\mathrm{A}}(u) \geq 0$ for all $C^2$ functions $u$ such that $u\geq 0$ and $u|_{D_a}=0$. 

Very similar computations to Lemma \ref{lem:BER} yield the following:

\begin{lemma}
\label{lem:stability-test}
Let $(M,\bar{g})$ be as in Section \ref{sec:radial-setup}. Fix $a\in\mathbb{R}^{n+1}$, and consider $u_a$ as above. If $\Sigma$ is minimal in $M$ then
\begin{equation}
\label{eq:L-test-fn} L u_a =   -n e^{ -2\phi}  \left(\phi''-\frac{\phi'}{r} - (\phi')^2 \right) \langle\nu_\delta, \pr_r \rangle^2 u_a + |A^g_\Sigma|_g^2 u_a
\end{equation}

Noting that $S=\pr B_R$ is totally umbilic, we have 
\begin{equation}
    \label{eq:N-test-fn}
    \pr_{\bar{\eta}} u_a = \kappa u_a, \qquad \text{where }k_S=\kappa g_S.
\end{equation}

In particular, if $\Sigma$ contacts $S$ at angle $\pit$ along $\pr \Sigma$, then the conormal derivative satisfies $(\pr_{\eta_g} - k_S(\nu_{\bar g},\nu_{\bar{g}}))u_a =0$.
\end{lemma}

These properties of $u_a$ allow one to deduce a classification of stable minimal hypersurfaces in the half-space $\overline{N_a}$, whose proof is also deferred to Appendix \ref{sec:stable-details}:

\begin{proposition}
\label{prop:half-space-bernstein}
Let $(M,\bar{g})$ and $S=\pr B_R$ be as in Section \ref{sec:radial-setup}, satisfying condition (A). Let $\Sigma$ be a compact, embedded minimal hypersurface contained in the half-space $\overline{N_a}$, such that $\pr \Sigma \subset D_a \cup (S\cap \mathcal{H}_a)$, and that $\Sigma$ has contact angle $\pit$ along $\pr \Sigma\cap S$. 

If $\Sigma$ is stable for one-sided variations which fix $D_a$, then $\Sigma$ must be totally geodesic. In fact, except for the case when $(M,g)$ has constant positive sectional curvature and $S$ is strictly concave, $\Sigma$ must be an equator, that is, $\Sigma = D_b \cap \overline{N_a}$ for some $b\in\mathbb{S}^n$. 
\end{proposition}

\subsection{Topological implications}
\label{sec:nodal}

In this section, we consider minimal surfaces $\Sigma^2$ in $M^3$, where $(M,\bar{g})$ and $S=\pr B_R$ follow the same setup as in Section \ref{sec:radial-setup} (with $n=2$). We will assume that $\Sigma$ contacts $S$ at angle $\pit$ along $\pr \Sigma$. 

The main result of this section is that if such a surface $\Sigma$ is embedded and has genus 0, then it is either a totally geodesic disc, or must be a radial graph. This fact has already been observed in Euclidean space in \cite[Corollary 4.2]{KM22}. An analogous result was crucial in classifying genus-zero shrinkers in Gaussian space \cite{Bre16}. The proof of such a result relies on two important ingredients: the two-piece property and Courant's nodal domain theorem.

\begin{definition}
Let $\Sigma$ be a compact Riemann surface. We say that $Z\subset \Sigma$ is a network with boundary if it is the union of finitely many smooth curves $\gamma_i$, such that:
\begin{itemize}
\item if $p$ is an endpoint of some $\gamma_i$ that is not a closed loop, then either $\gamma_i$ meets $\pr\Sigma$ tranversely at $p$, or $\gamma_i$ meets some other $\gamma_j$ at $p$;
\item for any pair $\gamma_i,\gamma_j$, the interiors are disjoint (in particular each $\gamma_i$ is embedded);
\item at any intersection point $p$, the curves that meet at $p$ do so at equal angles. 
\end{itemize} 

We may also consider the surface $\tilde{\Sigma}$ obtained by collapsing each boundary component to a point. The image $\tilde{Z}$ under this map is an embedded network in $\tilde{\Sigma}$ with nodes given by the intersection points and the (collapsed) boundary components (that contained a point of $Z$). This is the \textit{reduced nodal graph} considered by Karpukhin et. al. \cite{KKP}.

We say that $Z$ is regular if the degree of each interior node is even, and the degree of each boundary node in $\tilde{Z}$ is at least 2. 
\end{definition}

The following local description of the nodal set is rather well-known, but we include the details and a general statement for completeness:

\begin{lemma}
\label{lem:nodal-local}
Let $\Sigma$ be a connected, compact Riemann surface. Suppose that $u:\Sigma\to \mathbb{R}$ is not identically zero, and satisfies $\Lap u =  -V u$ and $\pr_\eta u = c u$ on $\pr \Sigma$, where $V$ is a smooth function on $\Sigma$ (up to the boundary) and $c\in\mathbb{R}$. 

Consider any $p\in Z:= u^{-1}(0)$. There is a neighbourhood $U$ of $p$ so that $Z\cap U$ consists of $m$ arcs, where $m$ is the vanishing order of $u$ at $p$, and the angle between any adjacent pair of arcs is $\alpha=\frac{\pi}{m}$. If $p\in\pr\Sigma$ then each arc intersects $\pr\Sigma$ transversely at a half-integer multiple of $\alpha$.

In particular, $Z$ is a regular network with boundary. 
\end{lemma}
\begin{proof}
If $p$ in the interior, this follows from the classical theorem of Bers. So consider $p\in\pr\Sigma$. Following \cite[Proof of Theorem 2.3]{FS16}, we may take local conformal coordinates centred at $p$, so that $\Sigma$ corresponds to the upper half-disk and $p$ corresponds to the origin. Write $g = F^2\delta$, where $\delta$ is the Euclidean metric.

Then $u$ satisfies $\Lap_\delta u = - V F^2 u$ on the upper half-disc, and $\pr_y u = -c G u$ on the $x$-axis, where $G(x)=F(x,0)$. Then as $G$ only depends on $x$, we compute that $v:= e^{c Gy}u$ satisfies 
\[ \nabla^\delta v = e^{c Gy}(\nabla u + \tilde{b}), \qquad \tilde{b}=c(G'ye_x +Ge_y),\]
\[\Lap_\delta v = - V F^2 v + \tilde{b} \cdot \nabla^\delta v + e^{c Gy}G'' y  ,\]
\[ \pr_y v = e^{c Gy}(-c Gu + c Gu)=0.\]
Thus $v$ satisfies an elliptic equation, whose coefficients are smooth up to the boundary, with zero Neumann boundary conditions. The even reflection of $v$ is at least $C^{1,1}$, so using classical unique continuation as in \cite[Proof of Theorem 2.3]{FS16} (i.e. similar to the interior case) gives the stated local description of $Z$. The angle condition at boundary nodes follows from the reflection symmetry. In particular, $Z$ is a network with boundary. 

Note that any interior crossing is the junction of $k$ smooth curves, hence has degree $2k$. If there is a boundary component $\Gamma$ with just a single zero $p\in Z\cap \Gamma$, then as $\Gamma$ is connected $u$ cannot change sign on $\Gamma$. This implies that an even number of arcs of $Z$ must be incident at $p$. Thus, $Z$ is regular. 
\end{proof}

Again, we spell out the (somewhat well-known) consequence that for genus 0 surfaces, a two-piece property places restrictions on the (reduced) nodal graph:

\begin{lemma}
\label{lem:nodal-genus-0}
    Let $\Sigma$ be a compact Riemannian surface of genus 0. Let $u$ be as in Lemma \ref{lem:nodal-local} and $Z:=u^{-1}(0)$. Suppose that $\Sigma\setminus Z$ consists of exactly two connected components. Then $Z$ does not have any interior nodes, and each vertex of $\tilde{Z}$ has degree 0 or 2. 
\end{lemma}
\begin{proof}
Note that if $\gamma$ is a closed loop, or a curve with both endpoints on the same boundary component, then $\Sigma\setminus \gamma$ has two boundary components. (Glue a disk to each boundary component, and use that a closed loop separates $S^2$.) 

By the two-piece assumption, $\{u>0\}$ and $\{u<0\}$ are connected. Suppose for the sake of contradiction that $Z$ has an interior crossing point $p$. Then in particular, there is a neighbourhood $U$ of $p$ so that $U\setminus Z$ consists of disjoint components on which $u$ alternates in sign. In particular, there are two disjoint components on which $u>0$. Since $\{u>0\}$ is connected, we can connect these by a loop $\gamma \subset \{u>0\}\cup\{p\}$ based at $p$. As $\Sigma$ has genus 0, $\Sigma\setminus \gamma$ must consist of two connected components. But by the alternating signs, there is a component of $\{u<0\}$ in both components of $\Sigma\setminus \gamma$, which contradicts connectedness of $\{u<0\}$. 

Similarly, consider a boundary component $\Gamma$ of $\Sigma$. Then there is a neighbourhood $U$ of $\Gamma$ so that $U\setminus Z$ consists of $k$ disjoint components on which $u$ alternates in sign. In particular, $k$ must be even. If $k\geq 4$, then at least two components of $U\setminus Z$ have $u>0$. These must be connected by a curve $\gamma$ (with endpoints on $\Gamma$)  through $\{u>0\}$. But then $\Sigma\setminus \gamma$ must consist of two connected components, and by the alternating signs, there is a component of $\{u<0\}$ on either side of $\Sigma\setminus \gamma$. This contradicts connectedness of $\{u<0\}$. We conclude that either 0 or 2 arcs of $Z$ have endpoints on $\Gamma$. 
\end{proof}

We now construct `support functions' on the ambient manifolds $(M^3,\bar{g})$ in Section \ref{sec:radial-setup}. Consider $u: M\to \mathbb{R}$ defined (in terms of the conformal model $\bar{g} = e^{2\phi}\delta$) by
\[ u:=  e^{\phi} \langle x,\nu_\delta\rangle,\] where again $\nu_\delta = e^\phi\nu_{\bar{g}}$ is the unit normal of $\Sigma$ with respect to $\delta$.

In the following proposition, we say that $\Sigma$ has the two-piece property if it satisfies the conclusions of Corollary \ref{cor:two-piece}. Recall, in particular, that Corollary \ref{cor:two-piece} established the two-piece property for \textit{constrained} FBMS in settings where the pair $(M, S=\pr M)$ satisfies condition (A).

\begin{proposition}
\label{prop:radial graph}
Let $(M^3,\bar{g})$ and $S=\pr B_R$ be as in Section \ref{sec:radial-setup}. Let $(\Sigma^2, \partial \Sigma) \hookrightarrow (M, S)$ be a connected, compact, minimal embedding that contacts $S$ with angle $\pit$ along $\partial \Sigma \subset S$. 
Assume that $\Sigma$ has the two-piece property. 

Suppose that $\Sigma$ has genus 0 and $b$ boundary components. If $\Sigma$ is not a totally geodesic disc, then $o \not\in \Sigma$ and $\bar{g}(\pr_\rho,\nu_{\bar{g}})$ does not vanish on the interior of $\Sigma$.
\end{proposition}

\begin{proof}
Let $u_a=e^\phi x_a$ be as in Lemma \ref{lem:stability-test}. By the result of that lemma, we have $(\Lap^g_\Sigma + V)u_a=0$ and that $(\pr_{\eta_g}-c)u_a =0$, where 
\begin{align*}
V &= \Ric_M(\nu_{\bar{g}}, \nu_{\bar{g}}) +n e^{-2\phi}(\phi'' -r^{-1} \phi' - (\phi')^2) \langle\nu_\delta, \partial_r\rangle^2, \\
c & = k_S(\bar{\nu}_{\bar{g}},\bar{\nu}_{\bar{g}})=e^{-\phi(\bar{r})}\left(\phi'(\bar{r}) + \frac{1}{\bar{r}}\right). 
\end{align*}
In particular, $u$ satisfies the assumptions of Lemma \ref{lem:nodal-local}. So for any $a$, we have the following dichotomy - either:

\begin{enumerate}[(i)]
\item $u_a\equiv 0$, so $\Sigma$ is a piece of the hyperplane $\Gamma_a:= \{x_a\}=0$ in the conformal model, and is a totally geodesic disc;
\item $u_a\not \equiv 0$, so by Lemma \ref{lem:nodal-local}, the two-piece property and Lemma \ref{lem:nodal-genus-0}, we have that $Z_a:= u_a^{-1}(0)$ is a regular network with boundary, and cannot have an interior crossing point.
\end{enumerate}

By supposition $\Sigma$ is not totally geodesic, so item (ii) must hold for any $a$. We claim that if either $o\in \Sigma$ or $\bar{g}(\pr_\rho, \nu_{\bar{g}})=0$ at an interior point, then there is an $a\neq 0$ for which $u_a$ has a second order zero at some $y \in \Sigma$. Given this claim, Lemma \ref{lem:nodal-local} implies that $Z_a$ must contain at least 2 arcs crossing at $y$. This contradicts item (ii), so we would conclude that $o\notin\Sigma$ and that $\bar{g}(\pr_\rho,\nu_{\bar{g}})$ does not vanish on the interior, as desired. 

To prove the claim, first suppose $o \in \Sigma$. In the conformal model $o$ maps to $0$, so we let $a=\nu_\delta(o)$, and certainly $x_a=\langle 0,a\rangle=0$ at $o$. Furthermore, \begin{equation}
\label{eq:u-a-grad} e^{2\phi}\nabla^g u_a = x_a \nabla^{\delta} e^\phi + e^\phi \nabla^{\delta}x_a= x_a \nabla^{\delta} e^\phi + e^\phi a^{\top\Sigma}.\end{equation} As $x_a=0$ and $a$ is normal at $o$, in particular we have $\nabla^g u_a(o)=0$. 

Suppose that $\bar{g}(\pr_\rho,\nu_{\bar{g}})=0$ at some $y \in \Sigma$. In terms of the conformal model, this means that at $y$ we have $0= \langle\pr_r, \nu_\delta\rangle = \frac{x_a}{r}$, where we have set $a=\nu_\delta(y)$. In particular, $u_a(y)=0$. Again by (\ref{eq:u-a-grad}), as $x_a=0$ and $a$ is normal at $y$, in particular we have $\nabla^g u_a(y)=0$. 

\end{proof}

\begin{remark}
In fact, the above method allows us to classify embedded free boundary minimal discs, at least when $\Ric_M \leq 0$ or if $M=\mathbb{S}^3$. (Of course, for space forms, even immersed free boundary minimal discs are classified; see Section \ref{sec:hopf}.)

Indeed, consider $b=1$. Since $\Sigma$ contacts $S$ at angle $\pit$ along $\pr\Sigma$, we have that $g(\pr_\rho,\nu_{\bar{g}})=0$ at \textit{any} $y$ \textit{on the boundary} $\pr \Sigma$. Choose such a $y$ and set $a=\nu_\delta(y)$ as above. Then Lemma \ref{lem:nodal-local} again implies that $Z_a$ must contain at least 2 arcs crossing at $y$. But by Lemma \ref{lem:nodal-genus-0}, the only boundary component $\pr \Sigma$ must have degree exactly 2 in the reduced graph $\tilde{Z}_a$. So the nodal set $Z_a$ must consist of a single loop, which intersects $\pr\Sigma$ tranversely at the single point $y$. 

From the local description, it follows that $u_a$ does not change sign on $\pr \Sigma$, that is, $\pr \Sigma$ is contained in the half-space $\overline{N_a}$. By either Lemma \ref{lem:sph-one-side} or Lemma \ref{lem:conf-one-side} below, it follows that $\Sigma$ is contained in the half-space boundary $D_a$ and is thus a totally geodesic disc. 
\end{remark} 

\begin{lemma}
\label{lem:sph-one-side}
Let $(\Sigma^k,\pr\Sigma)$ be a FBMS in $(\mathbb{S}^{N}, \pr B_R)$. Fix $a \perp e_0$, where $o=e_0$ is the centre of the ball $B_R$. If $u_a \geq 0$ on $\pr\Sigma$, then $u_a\equiv 0$ on $\Sigma$. 
\end{lemma}
\begin{proof}
We may work directly with the Euclidean coordinate $x_a =e^{-\phi} u_a$. By the Green's identity
\[ 0= \int_\Sigma (x_0(\Lap+k)x_a - x_a(\Lap+k)x_0) = \int_{\pr\Sigma} (x_0 \pr_\eta x_a - x_a \pr_\eta x_0) = -\frac{1}{\sin R} \int_{\pr\Sigma}  x_a.\]
As $x_a\geq0$ on $\pr\Sigma$, the only way equality can hold is if $x_a|_{\pr\Sigma}\equiv 0$. 
But then $x_a$ satisfies $(\Lap+k)x_a=0$ on $\Sigma$, with $x_a = \pr_\eta x_a=0$ on $\pr\Sigma$.  Thus $x_a$, hence $u_a$, must vanish identically. 
\end{proof}

\begin{lemma}
\label{lem:conf-one-side}
Let $(M^{n+1},\bar g)$ be as in Section \ref{sec:radial-setup}, satisfying conditions (A).

Assume that $\Ric_M \leq 0$. Let $x : (\Sigma^n, \partial\Sigma) \hookrightarrow (M, \partial B_R)$ be a connected, compact, minimal embedding that contacts $S = \partial B_R$ with angle $\pit$ along $\partial \Sigma \subset S$.

In conformal coordinates, if $u_a \geq 0$ on $\pr\Sigma$, then $u_a  \equiv 0$ on $\Sigma$. 
\end{lemma}
\begin{proof}
Consider the conformal model $(M^{n+1}, \bar g) \simeq (\mathbb{B}^{n+1}_{\bar{r}}, e^{2\phi} \delta)$. Let us begin by showing that $\Sigma$ must be constrained in $B_R$. By assumption $\Ric_M(\partial_r, \partial_r) \leq 0$. So - recalling that $g = e^{2\phi} \delta|_{\Sigma}$ - Remark \ref{rem:ricci} implies $r\phi'' + \phi' = (r\phi')' \geq 0$. Since $r\phi'\big|_{r = 0} = 0$, we must have $r\phi'\geq 0$ for all $r$. In particular, $\phi'(r) + r^{-1} \geq r^{-1}> 0$ for all $r \in (0, \bar{r}]$, i.e. $\partial \mathbb{B}^{n+1}_r$ is convex with respect to $e^{2\phi} \delta$ for all $r \in (0, \bar{r}]$. 

Applying the conformal change formula from Appendix \ref{sec:conformal-change}, we obtain 
\[
\Delta r^2 =e^{-2\phi}(\Lap^\delta r^2 + (n-2) \langle \nabla^\delta \phi, \nabla^\delta r^2\rangle\big) = e^{-2\phi}(2n+2(n-2)r\phi' |\pr_r^{\top\Sigma}|^2)>0.
\]
By the maximum principle, $r^2$ cannot attain an interior maximum and hence $r^2 \leq \max_{\partial \Sigma} r^2$. In terms of the original metric $\bar{g}$, this implies $\rho \leq R$, i.e. $\Sigma$ is constrained. 

Continuing, via Lemma \ref{lem:stability-test}, we compute
\begin{align}\label{eq:brdy_lemma}
    \big(\Lap + \mathrm{Ric}_M(\nu_{\bar{g}}, \nu_{\bar{g}})\big)u_a & = Lu_a - |A^g|^2_g u_a 
    = -2 e^{-2\phi}r\Big(\Big(\frac{\phi'}{r}\Big)'  - r\Big(\frac{\phi'}{r}\Big)^2\Big) \langle\nu_\delta, \partial_r\rangle^2 u_a. 
\end{align}
If we are in the equality case of ($\mathrm{A}'$), then $(M, \bar{g})$ is a geodesic ball in either $\mathbb{R}^n$ or $\mathbb{H}^n$ and we have $(\Delta + \Ric_M(\nu_{\bar g}, \nu_{\bar g}))u_a = 0$. In this case, the (strong) maximum principle implies $u_a \equiv 0$ on $\Sigma$. So let us suppose for the remainder of the proof that we have the strict inequality $\big(\frac{\phi'}{r}\big)'  - r\big(\frac{\phi'}{r}\big)^2 < 0$ in ($\mathrm{A}'$).

Now if $u_a$ is not everywhere nonnegative on $\Sigma$, then it attains a negative minimum at an interior point $y \in \Sigma$. Since $u_a(y) < 0$, we have $r > 0$. Moreover, we claim $\langle\nu_\delta, \partial_r\rangle|_y \neq 0$. Indeed, since $\nabla^g u_a(y) = 0$, we must have that $\nu_\delta$ is parallel to $\nabla^g u_a$ at $y$. Consequently, if $\langle\nu_\delta, \partial_r\rangle|_y = 0$, then $\langle\nabla^g u_a, \partial_r\rangle|_y = 0$. By  \eqref{eq:u-a-grad}, we have $e^{2\phi} \nabla^g u_a = u_a \phi' \partial_r + e^\phi a^{\top\Sigma}$. Thus 
\[
e^{2\phi} \langle\nabla^g u_a, \partial_r\rangle =  (u_a \phi' + e^\phi \langle a^{\top \Sigma}, \partial_r\rangle) =  u_a(\phi' + r^{-1}). 
\]
As we have seen above, $\Ric_M \leq 0$ implies $\phi' + r^{-1} \neq 0$. 

Thus, at the minimum $y$, the right hand side of \eqref{eq:brdy_lemma} is strictly negative, but as $\Ric_M \leq 0$, the left hand side is nonnegative. This is a contradiction, so we must have $u_a \geq 0$ on all of $\Sigma$. The strong maximum principle implies $u_a > 0$ on the interior of $\Sigma$. 

It follows that $x : (\Sigma, \partial \Sigma) \hookrightarrow (B_R, \partial B_R)$ is an embedding such that $\Sigma \setminus D_a$ is connected. Now the two-piece property, Corollary \ref{cor:two-piece}, implies that $\Sigma$ must be the totally geodesic disk $D_a \cap B_R$.
\end{proof}

\subsubsection{Extra topological remarks}

To close this section, we record some general topological ideas that may be of interest. The key is Cheng's work, particularly \cite[Lemma 3.1]{Ch76}, which we reproduce as follows:

\begin{lemma}[\cite{Ch76}]
Let $\Sigma$ be a compact Riemann surface of genus $g$. Suppose that $Z= \bigcup_{i=1}^k \gamma_i$, where the $\gamma_i:S^1\to \Sigma$ are injective, piecewise $C^1$ curves, and that $\gamma_i \cap \gamma_j$ consists of a finite number of points for each $i\neq j$. If $k\geq 2g+1$, then $\Sigma\setminus Z$ has at least $k-2g+1$ connected components. 
\end{lemma}

Cheng's work also applies to surfaces with boundary. Explicitly, we can record:

\begin{corollary}
\label{cor:cheng}
Let $\Sigma$ be a compact Riemann surface of genus $g$ and $Z$ a network with boundary. Suppose that the reduced nodal graph $\tilde{Z}$ contains $k$ distinct cycles. If $k\geq 2g+1$, then $\Sigma\setminus Z$ has at least $k-2g+1$ connected components.
\end{corollary}

This gives an quick proof of Lemma 4.3, modulo the following easy lemma on graphs:

\begin{lemma}
Let $G$ be a finite graph and $v_1\in G$ such that $\deg v_1\geq 1$ and every other vertex has degree at least 2. Then $G$ contains a cycle.
\end{lemma}
\begin{proof}
Starting at $v_1$, traverse along any edge to some $v_2$. Either $v_2$ is a new vertex, or we have found a cycle. Since $v_2$ has degree at least 2, we may traverse along a new edge, and continue this process. As there are finitely many vertices this process must terminate, upon which we have found a cycle. 
\end{proof}

Following the work of Seo \cite{Seo23}, we then have:

\begin{proof}[Alternative proof of Lemma \ref{lem:nodal-genus-0}]
We argue as in \cite{Seo23}:

By Corollary \ref{cor:cheng}, since $g=0$ the graph $\tilde{Z}$ must contain at most 1 cycle. 

Now suppose that $\tilde{Z}$ contains a node $v_1$ of degree at least 3. As $\tilde{Z}$ is regular, each node has degree at least 2, so in particular we may find a shortest cycle $C$ containing $v_1$. Then there is an edge $e=v_1v_2$ not contained in $C$. 

Let $G$ be the graph obtained by deleting $e$, and $G'$ the connected component containing $v_2$. Note that $G'$ cannot contain any vertices of $C$, otherwise there would be a cycle containing $e$ (hence distinct from $C$). But the degree of $v_2$ in $G'$ is at least 1, and the degree of every other vertex is at least 2, so $G_{v_2}$ contains a cycle which is disjoint from $C$. This is also a contradiction, so we conclude that each node in $\tilde{Z}$ has degree exactly 2. 

In particular, this rules out any interior crossings (which have degree at least 4), and also shows that $\tilde{Z}$ is a cycle graph. 
\end{proof}

\section{Hopf differential}
\label{sec:hopf}

In this section, we review the Hopf differential technique and its consequences for minimal surfaces in a space form that contact a geodesic sphere at a constant angle. For compact surfaces with simple topology, the Hopf differential gives useful information about the surface geometry. This method was previously used by Nitsche \cite{Ni85} (for free boundary minimal surfaces in $\mathbb{B}^3$), Fraser-Schoen \cite{FS15} (for free boundary minimal discs in $\mathbb{B}^n$) and Chodosh-Edelen-Li \cite{CEL24} (for capillary minimal discs in $\mathbb{S}^3_+$). We emphasise that our method is precisely the same, but we will point out that it extends to more general assumptions, and gives useful information about annuli. 

Specifically, suppose $(\Sigma^k, \partial \Sigma) \looparrowright (M^N, S)$ is an immersion. This yields standard orthogonal decompositions (suppressing the pullback over $\Sigma,\pr\Sigma,S$ respectively for ease of notation): 
\begin{align*}
TM &= T\Sigma \oplus N_M \Sigma, &  TS &= T\partial \Sigma \oplus N_S \partial \Sigma,\\T\Sigma &= T\partial \Sigma\oplus C^{\infty}(\Sigma) \eta, & TM &= TS \oplus C^{\infty}(M) \bar{\eta},
\end{align*}
where we have decorated the normal bundles $N_M \Sigma$ and $N_S \partial \Sigma$ to indicate the ambient space in which they sit. Let $\nabla^{\perp\Sigma}$, $\nabla^{\perp\partial \Sigma}$, and $\bar \nabla^S$ denote the connections induced on $N_M \Sigma$, $N_S \partial \Sigma$, and $TS$ respectively. (Note that $\bar{\nabla}$ induces $\nabla^{\perp\Sigma}$ and $\bar \nabla^S$ induces $\nabla^{\perp\partial \Sigma}$.) Let $\eta^S := \eta - \bar{g}(\eta, \bar{\eta}) \bar{\eta}$ denote the projection of the conormal $\eta$ to $TS$. Because $\eta$ is normal to $T\partial \Sigma$ and $\bar{\eta}$ is normal to $TS$, we may regard $\eta^S$ as a section of $N_S\partial \Sigma$. 

In what follows, we consider immersions $(\Sigma^k, \partial \Sigma) \looparrowright (M^N, S)$ that have both: 
\begin{enumerate}
    \item parallel mean curvature vector, $\nabla^{\perp\Sigma} \mathbf{H} = 0$; and
    \item parallel projected conormal; $\nabla^{\perp\partial \Sigma}  \eta^S = 0$.
\end{enumerate}

There are satisfying parallels between these conditions and familiar conditions on the conormal $\eta$ and the mean curvature $\mathbf{H}$. In particular, in any codimension, just as minimality ($\mathbf{H} = 0$) implies (1), the free boundary condition ($\eta^S = 0$) implies (2).  Moreover, these conditions generalise the notions of constant mean curvature and constant contact angle in codimension one, respectively:
Indeed, in codimension one $\nabla^{\perp\Sigma} \mathbf{H} = 0$ if and only if $\mathbf{H} = - H \nu$ for some constant $H$ (which may differ on each connected component of $\Sigma$).  Similarly, in codimension one, we find that $\nabla^{\perp\partial \Sigma} \eta^S = 0$ if and only if $\eta^S$ is a constant multiple of $\bar{\nu}$ on each component of $\pr\Sigma$, where $\bar{\nu}$ is a choice of unit normal of $\pr\Sigma$ in $S$.

In what follows, we will only consider surfaces ($k=2$); let $\mathbf{A}$ denote the vector-valued second fundamental form and $\mathring{\mathbf{A}} =\mathbf{A}-\frac{1}{2}\mathbf{H}g$ its traceless component. The Hopf differential proof is intrinsic and does not require the immersion to be constrained to the geodesic ball. 

\begin{proposition}
\label{prop:hopf}

Let $\mathbb{M}^N\in \{\mathbb{S}^N, \mathbb{R}^N, \mathbb{H}^N\}$. Suppose $(\Sigma^2, \partial \Sigma) \hookrightarrow (\mathbb{M}^N, \partial B_R)$ is a 2-sided proper immersion with parallel mean curvature vector, $\nabla^{\perp\Sigma} \mathbf{H} = 0$ and parallel projected conormal $\nabla^{\perp\partial \Sigma} \eta^S = 0$, where $S = \partial B_R$. Suppose that $\Sigma$ is diffeomorphic to a compact Riemann surface with genus 0 and $b$ boundary components.

If $b=1$, then $\Sigma$ is totally geodesic.

If $b=2$, then $|\mathring{\mathbf{A}}|$ does not vanish on $\Sigma$. 
\end{proposition}

\begin{proof}
We include a brief proof for the convenience of the reader. 

Since $b \leq 2$, the uniformization theorem implies we can find a conformal parametrization of $\Sigma$ into $\mathbb{M}$ given by 
\[
\Phi : \Omega \to \Sigma \looparrowright  \mathbb M,
\]
 where $\Omega=\{ z  \in \mathbb{C} : r_0 \leq |z| \leq 1 \}$ for some $r_0 \in [0, 1)$.

Introduce polar coordinates $z = re^{i\theta}$ on the complex plane. By assumption, the metric $g$ on $\Sigma$ is given in our (globally defined) polar coordinates by $g = \lambda^2 (dr^2 + r^2 d\theta^2)$ where $\lambda> 0$ is the conformal factor. Let $\partial_z :=\frac{1}{2} e^{-i\theta}(\partial_r - \frac{i}{r} \partial_{\theta}) $ and $\partial_{\bar{z}} := \frac{1}{2} e^{i\theta}(\partial_r + \frac{i}{r} \partial_\theta) $ denote the Wirtinger vector fields, $dz = e^{i\theta}(dr + i r\, d\theta)$ and $d\bar{z} = e^{-i\theta}(dr - i r\, d\theta)$ the corresponding forms, and extend the second fundamental form $\mathbf{A} : T\Sigma \otimes \mathbb{C} \times T\Sigma\otimes \mathbb{C} \to N_M \Sigma\otimes \mathbb{C}$ complex linearly. Similarly extend the metric $g$, the connections, and the Christoffel symbols. 

We begin by noting that $g = \lambda^2\, dz \otimes d\bar z$. Because the metric is conformal to (flat) $\Omega\subset \mathbb{C}$, it follows that $\nabla_{\partial_{\bar{z}}} \partial_z = \nabla_{\partial_z} \partial_{\bar z} = 0$. In particular, $\Gamma_{z\bar{z}}^z = \Gamma_{z\bar{z}}^{\bar{z}} = 0$. Differentiating the metric, we find 
\[
0 = \nabla_z g_{z\bar z} = \partial_z g_{z \bar z} -\Gamma_{zz}^z g_{z\bar{z}} = 2 \lambda (\partial_z \lambda) - \lambda^2 \Gamma_{zz}^z .
\]
On the other hand, for the second fundamental form we have
\begin{align*}
4z^2 \mathbf{A}_{zz} & = r^2 \mathbf{A}_{rr} - \mathbf{A}_{\theta\theta} + 2 i \, r \mathbf{A}_{r\theta}, \\
4\mathbf{A}_{z\bar{z}} &=  \mathbf{A}_{rr} + \frac{1}{r^2} \mathbf{A}_{\theta \theta} = \lambda^2 \mathbf{H}. 
\end{align*}
As a consequence, by the Codazzi equation, 
\begin{align*}
 {\nabla}^{\perp\Sigma}_{\bar{z}} \mathbf{A}_{zz} &= {\nabla}^{\perp\Sigma}_z \mathbf{A}_{z \bar{z}}\\
& = (\partial_z A_{z\bar z} )^{\perp\Sigma} - \Gamma_{zz}^z A_{z\bar z}\\
& = \frac{1}{4}\Big(2\lambda(\partial_z\lambda) \mathbf{H} + \lambda^2 \nabla^{\perp\Sigma}_z \mathbf{H} -\lambda^2\Gamma^z_{zz} \mathbf{H}\Big)=0,
\end{align*}
by our assumption of parallel mean curvature. 

Next, we consider the boundary condition. On $\partial \Sigma$, conflating $\partial_r, \partial_\theta$ with their pushforwards $\Phi_\ast \partial_r, \Phi_\ast \partial_\theta$, we note that $\partial_\theta$ is tangent to $S$ and $\partial_r = \pm \lambda \eta$ (with sign depending upon the component of $\partial \Sigma$). Now we observe 
\[
\bar{\nabla}_{\partial_\theta} \partial_r = \pm \bar{\nabla}_{\partial_\theta} (\lambda\, \eta)  = \pm \big((\partial_\theta\lambda) \eta + \lambda \big( \bar \nabla_{\partial_\theta}\eta)\big).
\]
Let us write $u = \bar{g}(\eta, \bar \eta)$, so we have the orthogonal decomposition $\eta = \eta^S + u \bar{\eta}$. We calculate 
\begin{align*}
    \bar \nabla_{\partial_\theta}\eta &= \bar{\nabla}_{\partial_{\theta}} \eta^S + (\partial_\theta u)\bar{\eta} + u \nabla_{\partial_\theta} \bar{\eta} \\
    & = \bar \nabla^S_{\partial_\theta} \eta^S  - k_S(\partial_\theta, \eta^S)\bar{\eta} + (\partial_\theta u) \bar{\eta} +u \,k_S(\partial_\theta)\\
    & = \bar \nabla^S_{\partial_\theta} \eta^S  + (\partial_\theta u) \bar{\eta} +u (\ct R) \partial_\theta,
\end{align*}
where we have used the umbilicity of $S=\pr B_R$ to obtain $k_S(\partial_\theta, \eta^S) = - (\ct R)\bar{g}( \partial_\theta, \eta^S) = 0$. 
Taking the projection to $N_M \Sigma$, we find along $\pr\Sigma$ that
\begin{align*}
    \mathbf{A}_{r\theta} &= (\bar{\nabla}_{\partial_\theta} \partial_r)^{\perp\Sigma} = \nabla^{\perp\partial \Sigma}_{\partial_\theta} \eta^S  + (\partial_\theta u) (\bar\eta)^{\perp\Sigma} = (\partial_\theta u) (\bar\eta)^{\perp\Sigma},
\end{align*}
by the assumption of parallel projected conormal. Let us show that parallel projected conormal also implies that $\pr_\theta u = 0$. Indeed, as $|\eta^S|^2 +u^2=|\eta|^2=1$, we have
\[
-2 u(\partial_\theta u) = \partial_\theta |\eta^S|^2 =\langle \nabla^S_{\partial_\theta} \eta^S, \eta^S \rangle =  \langle \nabla^{\perp\partial \Sigma}_{\partial_\theta} \eta^S, \eta^S \rangle = 0,
\]
in view of the fact that $\eta^S$ is a section of in $N_S\partial \Sigma$. 
That is, $\pr_\theta(u^2)=0$, so by continuity $u$ is constant on each component of $\pr\Sigma$ (and in particular $\pr_\theta u=0$). We have thus shown that

\[
\mathbf{A}_{r\theta} |_{\partial \Sigma} = 0.
\]

Now, as in Fraser-Schoen \cite{FS15}, let us consider the complex-valued function 
\[
\phi = 16 z^4 \mathbf{A}_{zz} \cdot \mathbf{A}_{zz} =  \big(|r^2\mathbf{A}_{rr} - \mathbf{A}_{\theta\theta}|^2 -4 |\mathbf{A}_{r\theta}|^2\big) + 4 i  \langle r^2 \mathbf{A}_{rr} - \mathbf{A}_{\theta \theta}, r\mathbf{A}_{r\theta} \rangle.
\]
Here ``$\cdot$" indicates the complex-linear extension of the real inner product on $N_M \Sigma \times \mathbb{C}$. The observations above imply $\partial_{\bar{z}} \phi = 0$ on $\Omega$ and $\mathrm{Im}(\phi) = 0$ on $\partial \Omega$. It follows that $\phi \equiv c \geq 0$ and $\mathbf{A}_{r\theta} = 0$ on $\Omega$. With these conclusions, we note that (using that $e_1 = \lambda^{-1} \partial_r$, $e_2 = \lambda^{-1} r^{-1} \partial_\theta$ are orthonormal)
Note that 
\begin{align*}
|\mathring{\mathbf{A}}|^2 &= \big|\mathbf{A}_{11} - \frac{1}{2} \mathbf{H}\big|^2+ \big|\mathbf{A}_{22} - \frac{1}{2} \mathbf{H}\big|^2 + |\mathbf{A}_{12}|^2= \frac{1}{2}\lambda^{-4} \big|\mathbf{A}_{rr}- \frac{1}{r^2} \mathbf{A}_{\theta\theta}\big|^2 + \lambda^{-2} |\mathbf{A}_{r\theta}|^2, 
\end{align*}
and consequently 
\[
\phi = |r^2 \mathbf{A}_{rr} - \mathbf{A}_{\theta \theta}|^2 = 2r^4\lambda^4|\mathring{\mathbf{A}}|^2.
\]

\textbf{Claim:} If $s$ is a holomorphic section of $N_M\Sigma\otimes \mathbb{C}$ which vanishes on $\pr\Sigma$, then $s\equiv 0$.

\textit{Proof of claim:} It is known that any complex vector bundle over a disk or an annulus\footnote{Indeed, an annulus is homotopy equivalent to $\mathbb{S}^1$, any orientable bundle over $\mathbb{S}^1$ is trivial, and any complex vector bundle induces an orientation on its underlying real vector bundle.} is trivial. In particular, $\Phi^*(N_M\Sigma\otimes \mathbb{C})$ admits a global frame $s_1,\cdots,s_{n-2}$. As in Fraser-Schoen \cite{FS16}, we may write $s=\sum_j f_j s_j$ for some functions $f_j$ satisfying \[\begin{cases}
    \pr_{\bar{z}}f_j  + \sum_k a_{kj} f_j =0, & \Omega \\
    f_j =0,& \pr \Omega
\end{cases}\]
This implies $f_j\equiv 0$ for all $j$, hence $s\equiv 0.$

If $\Sigma$ is a disc, then we must have $\phi \equiv \phi(0)=0$. In particular, $\mathring{\mathbf{A}}_{zz}=0$ on $\pr\Omega$. Applying the claim to $z^2 \mathring{\mathbf{A}}_{zz}$ gives that $\mathring{\mathbf{A}}_{zz}\equiv 0$ on $\Omega$. The remainder of the proof of \cite[Theorem 2.2]{FS16} shows that $\Sigma$ is a totally umbilic surface in a 3-dimensional, totally geodesic submanifold of $M$. 

In particular, by the classification of such surfaces, $\Sigma$ must be (a piece of):
\begin{itemize}
    \item A spherical cap in $\mathbb{S}^3\subset \mathbb{M}=\mathbb{S}^n$;
    \item A plane or spherical cap in $\mathbb{R}^3\subset \mathbb{M}=\mathbb{R}^n$;
    \item A totally geodesic disc, horosphere or geodesic sphere in  $\mathbb{H}^3\subset \mathbb{M}=\mathbb{H}^n$.
\end{itemize}

If $\Sigma$ is an annulus, then we still have $\phi\equiv c \in \mathbb{R}_{\geq 0}$. If $c=0$, then as in the disk case, $\Sigma$ will be a piece of one of the above surfaces. None of those surfaces contain an immersed annulus $(\Sigma,\pr\Sigma)\looparrowright(B_R,\pr B_R)$, so this is a contradiction. Hence $c > 0$, which (as $r> 0$ on the annulus) implies that $|\mathring{\mathbf{A}}| > 0$.
\end{proof}

Recall that, in the codimension 1 setting, we may work with the scalar-valued second fundamental form. In this case, we have the following additional result:

\begin{corollary}
\label{lem:same-sign}
Suppose that $\Sigma$ is diffeomorphic to a compact annulus with boundary and that $x: (\Sigma, \partial \Sigma) \looparrowright  (\mathbb{S}^3, \partial B_R)$ is an $(R, \gamma)$-minimal immersion. Then either $A(\eta, \eta) > 0$ on both components of $\partial \Sigma$ or else $A(\eta, \eta) < 0$ on both components of $\partial \Sigma$. 
\end{corollary}

\begin{proof}
As in the proof of Proposition \ref{prop:hopf} above, we have that $|A|$, and hence $A_{rr}$, are nowhere vanishing. In particular, we either have $A_{rr} > 0$ or else $A_{rr} < 0$ on all of $\Omega$. On the other hand, because the parametrization is conformal, $\eta$ is a nonzero multiple of $\pr_r$ at every point on $\partial \Sigma$, so $A(\eta,\eta)$ is a positive multiple of $A_{rr}$ at every such point and thus has the same sign as $A_{rr}$, which is itself the same on each boundary component. 
\end{proof}

\section{Polar surfaces of constant contact angle}
\label{sec:dual}

\subsection{Recalling Lawson's polar varieties}
 Let us recall Lawson's polar varieties \cite{Law70}. Suppose $\Sigma$ is a compact surface (with or without boundary). Consider a two-sided minimal immersion in $\mathbb{S}^3$ and a choice of Gauss map $\tilde{x} \in \{\nu, - \nu\}$ 
\[
x : \Sigma \to \mathbb{S}^3, \qquad \qquad \tilde{x} : \Sigma \to \mathbb{S}^3.
\]
Since $\Sigma$ is a minimal surface, in any local coordinate system, one has $d\nu(\partial_i) = A_i^j dx(\partial_j)$ and so the pullback $\tilde{g} := \tilde{x}^\ast g_1$ satisfies $\tilde{g}_{ij} = A^2_{ij} = \Psi g_{ij}$ where $\Psi=\frac{1}{2} |A|_g^2$. In particular, away from umbilic (hence flat) points,  $\tilde{x}$ defines a new immersion of $\Sigma$. In fact, since the umbilic points of a minimal surface in $\mathbb{S}^3$ are isolated whenever $\Sigma$ is not totally geodesic (as, for instance, umbilic points are zeroes of the holomorphic Hopf differential) and since away from umbilic points 
\[
(\tilde{\Delta} + 2) \tilde{x} = \pm \frac{2}{|A|^2}(\Delta + |A|^2) \nu  = 0, 
\]
we see that $\tilde{x} : \Sigma \to \mathbb{S}^3$ defines a new \textit{branched} minimal immersion. Lawson called this new immersion the polar variety to $\Sigma$. We will refer to it as the \textit{polar dual}. 

The polar dual of the polar dual of $\Sigma$ is just the original immersion up to composition with the antipodal map. 
Indeed, at a non-flat point $p$, the shape operator $d\nu$ gives an isomorphism of the tangent space $T_p\Sigma$. In particular, the vectors $x(p)$ and $\nu(p)$ still span the normal space to the dual, so as the double dual $\tilde{\nu}$ must be orthogonal to $\nu$, it must be $\tilde{\nu}(p) = \pm x(p)$.

In the construction to this point, we are free to choose either sign - for now, let us choose $ \tilde{\nu}(p)=x(p)$. With this choice of normal to the polar dual, it is easy to see that
\[
\tilde{A}_{ij} = \langle d \tilde{\nu}(\partial_i), d\tilde{x}(\partial_j)\rangle = \langle dx(\partial_i), d\nu(\partial_j)\rangle = A_{ij}. 
\] 
Hence
\[
\tilde{\Psi} = \frac{1}{2} |\tilde{A}|^2_{\tilde{g}} = \frac{1}{2} \tilde{g}^{ij} \tilde{g}^{kl}\tilde{A}_{ik}\tilde{A}_{jl} = \Psi^{-1},
\]

Lawson's polar duals are most degenerate on pieces of a totally geodesic $2$-spheres $\mathbb{S}^2$ in $\mathbb{S}^3$ - taking the polar dual maps these surfaces to singleton points. On the other hand, when $\Sigma$ is a closed minimal torus, or when $\Sigma$ a free boundary minimal annulus, the Hopf differential argument implies that there are no branch points, so Lawson's polar dual gives a genuine minimal immersion. 

\subsection{Polar duals of $(R, \gamma)$-minimal surfaces}

In this subsection, we examine Lawson's polar dual on $(R, \gamma)$-minimal surfaces. We fix coordinates so that $o=e_0$ is the centre of the ball $B_R \subset \mathbb{S}^3\subset \mathbb{R}^4$, and let $\rho$ be the (spherical) distance from $o$, so that $\bar{\eta}=\pr_\rho$. Having fixed $n=2$, recall from Definition \ref{def:Rgamma-cap} that an $(R, \gamma)$-minimal surface is a 2-sided minimal immersion $x : (\Sigma, \partial \Sigma) \looparrowright (\mathbb{S}^3, \partial B_R)$ that satisfies, on the boundary $\partial \Sigma$,
\[
\langle \nu, \partial_\rho \rangle = \cos \gamma, \qquad \langle \eta, \partial_\rho \rangle = \sin \gamma
\]
where $\gamma \in (0, \frac{\pi}{2}]$ (with our standard choice of $\nu$). 

Recall that, as $\pr B_R$ is totally umbilic, by Lemma \ref{lem:2ff-diag} we have $A(\eta,v)=0$ for any $v$ tangent to $\pr \Sigma$, and hence (along $\pr \Sigma$) \[A(\eta) = A(\eta,\eta)\eta.\] 

We now observe that polar duality preserves the class of immersed $(R, \gamma)$-minimal annuli. Recall that, by Corollary \ref{lem:same-sign}, either $A(\eta, \eta) > 0$ on both components of $\partial \Sigma$ or else $A(\eta, \eta) < 0$ on both components of $\partial \Sigma$.

\begin{lemma}
\label{lem:dual-imm-ann}
Suppose that $\Sigma$ is diffeomorphic to a compact annulus with boundary. Suppose that $x: (\Sigma, \partial \Sigma) \looparrowright  (\mathbb{S}^3, \partial B_R)$ is an $(R, \gamma)$-minimal immersion.  Then there exists $\varepsilon \in \{\pm1\}$ so that the \textit{polar dual} $\tilde{x} = \varepsilon \nu$ defines a $(\tilde{R}, \tilde{\gamma})$-minimal immersion $(\Sigma, \partial \Sigma) \looparrowright (\mathbb{S}^3, \partial B_{\tilde{R}})$, where
\begin{equation}\label{eq:dualRgamma}
(\tilde{R}, \tilde{\gamma}) = \left(\cos^{-1}(-\varepsilon\sin R \cos \gamma),\sin^{-1}\left(\frac{\sin R\sin\gamma}{\sin\tilde{R}}\right) \right).
\end{equation}

\end{lemma}

\begin{proof}
We have already seen in general that $\tilde{x}$ is a branched minimal immersion. Since by assumption $\Sigma$ has genus-zero and two boundary components, Proposition \ref{prop:hopf} implies $|A|$ is non-vanishing and hence $\tilde{x}$ has no branch points. 

We now show that $\tilde{x} : \Sigma \to \mathbb{S}^3$ satisfies $\tilde{x}(\partial \Sigma)\subset \partial B_{\tilde{R}}$ and that $\tilde{x}(\partial \Sigma)$ contacts $\partial B_{\tilde{R}}$ at constant angle $\tilde{\gamma}$ for suitable $\tilde{R}, \tilde{\gamma}$. 

By Lemma \ref{lem:boundary-Rgamma-cap} above, we have (along $\pr\Sigma$) that

\[
    \langle \tilde{x}, e_0 \rangle = \varepsilon\langle \nu,e_0\rangle = -\varepsilon \sin R \cos \gamma,
\]
so taking $\tilde{R} = \cos^{-1}(-\varepsilon \sin R \cos \gamma)$ gives $\tilde{x}(\pr\Sigma)\subset \pr B_{\tilde{R}}$. To ensure contact in the correct orientation, we choose $\varepsilon \in \{\pm1\}$ so that
\[
\varepsilon A(\eta,\eta) > 0, \qquad \text{ on } \pr \Sigma.
\]
This choice is possible because, by Corollary \ref{lem:same-sign}, the sign of $A(\eta, \eta)$ is constant on all of $\partial \Sigma$. 

To check the contact angle condition, we will consider the corresponding conormals. For this portion, it is prudent to distinguish between the intrinsic conormal $\eta$ and its pushforward $\pr_\eta x$, and similarly for the polar dual. 
By the discussion above, we have $\pr_\eta \tilde{x} = \varepsilon \pr_\eta \nu = \varepsilon A(\eta,\eta)\pr_\eta x$. As $\pr_{\tilde{\eta}}\tilde{x}$ and $\pr_\eta x$ are the (extrinsic) unit conormals, it follows that the intrinsic conormal of the polar dual is $\tilde{\eta} = \frac{1}{\varepsilon A(\eta,\eta)} \eta$, so that $\pr_{\tilde{\eta}}\tilde{x}=\pr_\eta x$. 

The outer normal $\tilde{\bar{\eta}}$ to $B_{\tilde{R}}$ at $\tilde{x}$ will be $\pr_\rho(\tilde{x}) = -\frac{1}{\sin \tilde{R}} e_0^{\top,\mathbb{S}^3}(\tilde{x})$, so again by Lemma \ref{lem:boundary-Rgamma-cap} we can compute the contact angle condition along $\pr\Sigma$:

\[ \langle \pr_{\tilde{\eta}}\tilde{x}, \pr_\rho(\tilde{x}) \rangle = -\frac{1}{\sin \tilde{R}}\langle \pr_\eta x, e_0\rangle = \frac{\sin R \sin \gamma}{\sin \tilde{R}}.\]

Note that, by definition of $\tilde{R}$, we have  

\[\frac{\sin^2 R \sin^2 \gamma}{\sin^2 \tilde{R}} = \frac{\sin^2 R \sin^2 \gamma}{1-\sin^2 R \cos^2 \gamma} = \frac{\sin^2 R \sin^2 \gamma}{\cos^2 R + \sin^2 R\sin^2 \gamma} \leq 1. \]

In particular, there exists $\tilde{\gamma} \in (0,\pit]$ so that $\sin\tilde{\gamma} \sin \tilde{R}= \sin R \sin \gamma$, which completes the proof. \end{proof}

We note the following special cases:

\begin{itemize}
\item For an $(R,\pit)$-minimal surface, we have $(\tilde{R},\tilde{\gamma})=(\pit, R)$.
\item For a $(\pit, \gamma)$-minimal surface, we have $(\tilde{R},\tilde{\gamma})$ is either $(\pi-\gamma, \pit)$ or $(\gamma,\pit)$ depending on $\varepsilon$ (that is, the sign of $A(\eta,\eta)$.)
\end{itemize}

\subsection{Duals of embedded annuli}

Finally, we observe that for \textit{constrained and embedded} free boundary minimal \textit{annuli}, the polar dual is also constrained and embedded. This is based on Ros' observation \cite[Theorem 5]{Ros95} that embedded minimal tori have embedded polar duals. (Of course, by Brendle's resolution \cite{Br12} of the Lawson conjecture, the only such torus is the Clifford torus - which is self-dual, up to ambient isometry.)

\begin{proposition}
\label{prop:emb-dual}
Let $R\in (0,\pit]$. Suppose $\Sigma$ is diffeomorphic to an annulus and that  $x : (\Sigma , \partial \Sigma)\hookrightarrow (B_R, \partial B_R)$ is an $(R, \pit)$-minimal embedding constrained in $B_R\subset\mathbb{S}^3$. Then its polar dual is a $(\pit,R)$-minimal embedding constrained in $B_\pit$.
\end{proposition}

\begin{proof}
In the case $R=\pit$, we may consider the surface $\tilde{\Sigma}$ consisting of $\Sigma$ and its reflection through the equator $\pr B_\pit$. The doubled surface $\tilde{\Sigma}$ will be a $C^2$ minimal (hence smooth) torus in $\mathbb{S}^3$, and the result follows from \cite[Theorem 5]{Ros95}, or from the fact that $\tilde{\Sigma}$ must be the Clifford torus. 

Henceforth, we assume that $R<\pit$.

In view of Lemma \ref{lem:dual-imm-ann}, as $\Sigma$ is compact, we need only to show that $\tilde{x}:=\varepsilon\nu$ is injective and that $\tilde{x}$ maps the interior of $\Sigma$ to the interior of the hemisphere $B_\pit$. 

By Proposition \ref{prop:radial graph}, $\Sigma$ is a radial graph, so $\langle \nu, e_0\rangle$ does not vanish on the interior of $\Sigma$ (and $\langle \nu,e_0\rangle \equiv 0$ on $\pr \Sigma$). As $\tilde{x}= \varepsilon \nu$ was chosen so that $\pr_\eta \langle\tilde{x} , e_0\rangle <0$, it follows that $\langle \tilde{x}, e_0\rangle$ is positive on the interior of $\Sigma$ and vanishes on $\pr\Sigma$, that is, $\tilde{x}=\varepsilon\nu$ is constrained in $B_\pit$. 

It remains to show that $\nu$ is injective. We split this into two main parts, each involving a different auxiliary map with codomain of the same dimension as $\Sigma$:
\begin{enumerate}
    \item In the first part, we show that $\nu|_{\pr\Sigma}$ is injective. To do so, we adapt an argument of Ros \cite{Ros95} by constructing a doubled and projected surface in $\mathbb{R}^3$, and considering its Gauss map $\nu''$ (which has values in $\mathbb{S}^2$). 
    \item In the second part, we define a projected Gauss map $\mathbf{n} : \Sigma \to \mathbb{S}^2$, show that $\mathbf{n}$ agrees with $\nu$ on $\partial \Sigma$, that $\mathbf{n}$ is injective, and finally that this implies injectivity of $\nu$.  
\end{enumerate}
The details of these arguments are given in the following subsections.
\end{proof}

\subsubsection{Part 1}

Denote by $\Xi:\mathbb{R}^3\to\mathbb{S}^3\setminus {-e_0}$ the standard (inverse) stereographic projection, and let $\Phi_R: \mathbb{S}^3 \to\mathbb{S}^3$ be the conformal dilation about $e_0$, taking $B_R$ to $B_\pit$. Note that conformal dilation is conjugate to a Euclidean dilation, namely $\Phi_R =   \Xi \circ \mathcal{D}_\lambda \circ\Xi^{-1}$
for a particular $\lambda$ depending on $R$.

The surface $\Sigma_R := (\Phi_R)_*\Sigma$ is a free boundary (but not minimal) surface in $B_\pit$ and let $\tilde{\Sigma}_R$ denote its even reflection across $\{x_0 =0\}$. Then the doubled surface $\tilde{\Sigma} = \Sigma_R \cup \tilde \Sigma_R$ is a closed ($C^2$) surface of $\mathbb{S}^3$. Finally, (recall that $e_0\notin \Sigma$), we let:
\begin{itemize}
    \item  $\Sigma' = (\Xi^{-1})_* \tilde{\Sigma} \subset \mathbb{R}^3$ be the stereographic projection of the reflected surface, 
    \item $\nu' : \Sigma' \to \mathbb{S}^2$ the Gauss map of $\Sigma'$, and 
    \item $\Gamma' = (\Xi^{-1} \circ \Phi_R)_\ast \partial \Sigma \subset \partial \mathbb{B}_1$ denote the image of the boundary of $\Sigma$ under the mapping.
\end{itemize}

Note that, by the conjugation relation above, the `undoubled' portion of $\Sigma'$ is also given by $(\mathcal{D}_\lambda)_* (\Xi^{-1})_*\Sigma$, that is, by dilating the stereographic projection of $\Sigma$. It is contained in $\mathbb{B}_1$ and has boundary $\Gamma'$.

We begin by proving two claims. \\

\noindent \textbf{Claim 1:} \textit{If $\nu$ is noninjective on $\partial \Sigma$, then $\nu'$ is noninjective on $\Gamma'$. }\\

\noindent \textit{Proof of Claim 1:} Note that Euclidean dilation leaves the Gauss map invariant. Therefore, along $\Gamma'$, $\nu'$ may be computed as the Gauss map of $(\Xi^{-1})_*\Sigma$. In particular, as stereographic projection is conformal, $\nu'$ will be the unit vector in the direction of the pushforward of $\nu$ by $\Xi^{-1}$. 

For $x\in \mathbb{S}^3\setminus\{ -e_0\}$ the stereographic projection is given by $\Xi^{-1}(x) = \frac{x-x_0e_0}{1+x_0}$. For $v\perp x$, the differential is given by
\[
(d\Xi^{-1})_x(v) = \frac{1}{1+x_0}v - \frac{\langle v,e_0\rangle}{(1+x_0)^2}(x+e_0).
\] 
Along $\pr\Sigma$, the free boundary condition gives $\nu \perp \pr_\rho$ and hence $\langle \nu,e_0\rangle=0$. Thus under stereographic projection, the normal $\nu$ pushes forward to the vector $\frac{\nu}{1+x_0}$. By the above discussion, it follows that if $\nu$ is noninjective on $\partial \Sigma$, then $\nu'$ must be noninjective on $\Gamma'$. \qed \\

\noindent \textbf{Claim 2:} \textit{Both $\det d\nu' < 0$ on $\Gamma'$ and $\det d\nu < 0$ on $\partial \Sigma$. }\\

\noindent \textit{Proof of Claim 2:} The determinants $\det d\nu'$ and $\det d\nu$ are the determinants of the second fundamental forms at corresponding points of $\Sigma'$ and $\Sigma$. The surface $\Sigma' \cap \overline{\mathbb{B}_1}$ is the image of $\Sigma$ under the radially symmetric conformal map $\Xi^{-1} \circ \Phi_R$. In particular, the second fundamental form $A_{\Sigma'}$ at $x' \in \Sigma' \cap \overline{\mathbb{B}_1}$ is related to the second fundamental form $A_{\Sigma}$ at the corresponding $x \in \Sigma$ by the conformal change transformation formula \eqref{eq:2ff-conf-change} with a radial conformal factor $e^{\phi(\rho)}$. As above, the free boundary condition gives $\nu\perp\pr_\rho$ along $\pr\Sigma$, so that $\pr_\nu \phi(\rho)=0$. It follows that $\det A_{\Sigma'}$ and $\det A_\Sigma$ have the same signs along $\Gamma'$ and $\pr\Sigma$. In particular, the signs are negative as $\det A_\Sigma = - |A_\Sigma|^2/2 <0$ by minimality of $\Sigma$. \qed\\

We now continue with the proof of Part 1. 

Suppose that $p_1,p_2\in\pr\Sigma$ have the same normal $\nu(p_1)= \nu(p_2):=v$. Let $x_v := \langle x,v\rangle$ and let $\mathbb{S}^2_v := \{x_v = 0\} \subset \mathbb{S}^3$ be the  equator defined by $v$. By the free boundary condition, $\langle v,e_0\rangle=0$, so as in Section \ref{sec:nodal}, $u:= x_v|_{\Sigma}$ defines a Steklov eigenfunction on $\Sigma$ and has a critical zero at each $p_i$. Note also that $-e_0 \in \mathbb{S}^2_v$. By Lemma \ref{lem:nodal-genus-0}, $Z:=u^{-1}(0) = \Sigma \cap \mathbb{S}^2_v$ consists of arcs that are smooth on the interior, and meet the boundary transversely. Moreover, the degree of each boundary component in the reduced graph $\tilde{Z}$ is at most 2. Let $\Gamma_i$ be the boundary component containing $p_i$. Then, as $p_i$ is a critical zero (hence contributes 2 to the degree), it follows that $Z\cap \Gamma_i$ can only contain the single point $p_i$. As there are only 2 boundary components, we deduce from Lemma \ref{lem:nodal-genus-0} that $Z$ consists of precisely two smooth arcs intersecting each other orthogonally, and $\pr\Sigma$ transversely at angle $\pi/4$, at each $p_i$. 

Note that, by the two-piece property (see Proposition \ref{prop:radial-sphere}), $\Sigma \setminus Z$ has two connected components that we denote $U_\pm$. As $u|_{\pr\Sigma}$ is certainly not identically 0, Lemma \ref{lem:sph-one-side} implies that $u$ must change sign on $\pr \Sigma$. It follows that $Z$ must be a non-contractible loop in $\Sigma$, and $u$ have opposite signs on each boundary component of $\pr\Sigma$ (and opposite signs on $U_\pm$). In particular, after a relabeling, we must have $\partial U_+ = Z \cup \Gamma_1$ and $\partial U_- = Z \cup \Gamma_2$. 

Let $W = \tilde{\Sigma} \cap \mathbb{S}^2_v$ be the even reflection of $(\Phi_R)_\ast Z$. As $Z$ consists of two smooth arcs that meet orthogonally (and with angle $\pi/4$ with respect to each $\Gamma_i$) at each $p_i$, it follows that $W$ is a union of two smooth arcs that meet at the points $\tilde{p}_i := \Phi_R(p_i) \in \partial B_{\pit}$ orthogonally. In  particular, the unit normal of $\tilde{\Sigma}$ only takes the value $\pm v$  at the points $\tilde{p}_i$ ($v$ is only tangent to $\mathbb{S}^3$ along $\mathbb{S}^2_v$ and $\tilde \Sigma$ is tangent to $\mathbb{S}^2_v$ only at the points $\tilde{p}_i$). 

Next, note that $\tilde{\Sigma} \setminus W$ consist of two connected pieces. Indeed, $\tilde{\Sigma} \setminus W$ is the union of the sets $\tilde{U}_+ := \tilde{\Sigma} \cap \{x_v > 0\}$ and $\tilde{U}_- := \tilde{\Sigma} \cap \{x_v < 0\}$. Evidently, $\tilde{U}_{+}$ and $\tilde{U}_-$ are disjoint. To see that $\tilde{U}_+$ is connected, observe that $\tilde{U}_+$ is the union of $(\Phi_R)_\ast U_+$ and its reflection (both connected), which are connected through the boundary component $(\Phi_R)_\ast\Gamma_1 \subset \partial B_{\pit}$, fixed by the reflection. Similarly $\tilde{U}_-$ is connected. 

In summary, we have that $\tilde{\Sigma}$ is a disjoint union $\tilde{\Sigma} = \tilde{U}_+ \cup \tilde{U}_- \cup W$ with $\tilde{U}_\pm$ connected, $\partial \tilde{U}_{\pm} = W$, and $W$ the union of two smooth arcs meeting orthogonally at two points $\tilde{p}_i \in \partial B_{\pit}$. The unit normal of $\tilde \Sigma$ only takes the values $\pm v$ at these points. 

We now finally consider the stereographic projection mapping $\tilde{\Sigma}$ to $\Sigma'$. Under stereographic projection from $-e_0$, the sphere $\mathbb{S}^2_v$ maps to a plane $\mathbb{R}^2_w$ with normal $w \in \mathbb{S}^2$ (the normalized image of $v$ under $\Xi^{-1}$). The points $\tilde p_i$ map to points $p_i' \in \Gamma' \cap \partial \mathbb{B}_1$ and, by Claim 1, $\nu'(p_i') = w$.  It follows from our discussion that 
\[
Z' := \Sigma' \cap \mathbb{R}^2_w
\]
consists of two smooth arcs meeting orthogonally at a pair of points $p_i' \in \Gamma' \cap \partial \mathbb{B}_1$, $i = 1, 2$ and that 
\[
\Sigma' \setminus Z' = U_+' \cup U_-'
\]
is the union of two open connected subsets. The unit normal $\nu'$ of $\Sigma'$ must attain the value $\pm w$ at more than just the points $p_i'$ after projection. 
However, by Claim 2 above, non-vanishing of $\det d \nu'$ at the points $p_i'$ and the fact that $\mathbb{R}^2_w$ intersects $\Sigma'$ transversally elsewhere ensures that there exists an open neighborhood $\Sigma' \cap \mathbb{R}^2_w$ such that $\nu'$ attains $w$ only at $p_i'$ in this neighborhood. 

The situation we have obtained is impossible. To see this, for $\varepsilon > 0$, consider the closed slab $\mathbb{R}^2_w(\varepsilon)$ of points with distance $\varepsilon$ to $\mathbb{R}^2_w$. If $\varepsilon > 0$ is sufficiently small, then the surface $\Sigma'(\varepsilon) := \Sigma' \cap \mathbb{R}^2_w$ satisfies:
\begin{itemize}
    \item The only points in $\Sigma'(\varepsilon)$ where $\nu'$ attains the values $\pm w$ are at the points $p_1'$ and $p_2'$.
    \item For a suitable choice of $\varepsilon > 0$ small, $\partial \Sigma'(\varepsilon)$ is a smooth curve meeting $\partial \mathbb{R}^2_w(\varepsilon)$ transversally in a smooth Jordan curve. 
\end{itemize}
We addressed the former point above. The latter point is a consequence of the fact that each $U_\pm'$ is a disk (and Sard's theorem) - see Appendix \ref{sec:ros-details} for details. In particular, $\Sigma'(\varepsilon)$ has genus-one (by Gauss-Bonnet). 

Using the second point above, we finally consider a \textit{closed} surface $\Sigma'' \subset \mathbb{R}^3$ obtained by (smoothly) capping off $\Sigma'(\varepsilon)$ along $\partial \Sigma'(\varepsilon)$ with a pair of disks in $\mathbb{R}^3 \setminus \mathbb{R}^2_w(\varepsilon)$. These caps can be glued in such a way that $\nu''$, the Gauss map of $\Sigma''$, attains the values $\{\pm w\}$ at exactly four points: $w$ at each point $p_1'$ and $p_2'$; $w$ at the maximum of the $w$-coordinate $x_w$, which occurs on one cap; and $-w$ at the minimum of $x_w$, which occurs on the other cap. 

We have already noted that $\det d\nu'' |_{p_i'} = \det d\nu' |_{p_i'} < 0$ at each point $p_i'$. On the other hand, the extreme points of $x_w$ on each cap are convex points, so we must have $\det d \nu'' >0$ at those cap points. Thus $w$ and $-w$ are regular values of $\nu'' : \Sigma'' \to \mathbb{S}^2$, and so the degree of $\nu''$ may be calculated above either point. By the discussion above, both points $p_i'$ contribute $-1$ to the degree $\mathrm{deg}(\nu'', w)$, while the cap points contribute degree $1$ to $\mathrm{deg}(\nu'', w)$ and $\mathrm{deg}(\nu'', -w)$ respectively. Our accounting, however, now gives that $\mathrm{deg}(\nu'', w) =1-2= -1$ while $\mathrm{deg}(\nu'', -w) = 1$. This is the desired contradiction, and we conclude that $\nu$ is injective on $\partial \Sigma$, as was to be shown.

\subsubsection{Part 2}

Let $\pi_{e_0^\perp}(v) = v - \langle v,e_0\rangle e_0$ be the orthogonal projection away from $e_0$. Let 
\[
\mathbf{n} := \frac{\pi_{e_0^\perp}(\nu)}{|\pi_{e_0^\perp}(\nu)|} = \frac{\nu - \nu_0 e_0}{\sqrt{1 - \nu_0^2}},
\]
where $\nu_0 = \langle \nu,e_0\rangle$. In this section, we will prove that $\nu$ is injective using that $\nu|_{\partial \Sigma}$ is injective and the map $\mathbf{n}$. 

We claim that $\nu_0^2<1$, so that $\mathbf{n}$ is a well-defined map $\Sigma \to \mathbb{S}^2$. Indeed, if $\nu_0^2=1$ then we must have $\langle x,e_0\rangle=0$ and $e_0^{\top\Sigma}=0$. The former can only happen if $R=\pit$, and $x\in \pr\Sigma$, but in this case the free boundary condition gives that $e_0 = -\eta$ is tangent to $\Sigma$, which contradicts the second condition. In fact, since $\langle \nu, e_0 \rangle = 0$ on $\partial \Sigma$, we have $\mathbf{n}|_{\pr\Sigma} = \nu|_{\pr\Sigma}$.

We next compute the differential of $\mathbf{n}$. Fix any $p \in \Sigma$ and consider a normal coordinate frame $\partial_i$, orthonormal at $p$. We find that 
\[
(1-\nu_0^2)^{\frac{1}{2}}\langle \pr_i \mathbf{n}, \pr_j\rangle = A_{ij} - \frac{A(\pr_i, e_0^\top)\langle \pr_j , e_0^\top\rangle}{1-\nu_0^2} = A_{ij} - A_{ik} U^k_j = A_{ik}(\delta^k_j - U^k_j),
\]
where $U$ is the matrix at $p$ with components $U^k_j = \delta^{kl}\langle \partial_l, e_0^\top \rangle \langle \partial_j,e_0^\top\rangle(1-\nu_0^2)^{-1}$. As $\mathrm{tr}_g A = 0$ and $|A| > 0$ (see Proposition \ref{prop:hopf}), the eigenvalues of $A$ are $\pm \Psi$ with $\Psi = \frac{1}{\sqrt{2}}|A| > 0$ and $A$ is nondegenerate. On the other hand, $U$ is bounded in operator norm by $\frac{|e_0^\top|^2}{1-\nu_0^2} = \frac{1 - \nu_0^2 - x_0^2}{1-\nu_0^2}$. As $R < \pit$, we have $x_0 > 0$ on $\Sigma$. So it follows $\delta^k_j - U^k_j > 0$ and hence that $d\mathbf{n}$ is nondegenerate.

Finally, we put injectivity $\mathbf{n}|_{\partial \Sigma} = \nu|_{\partial \Sigma}$ and nondegeneracy of $d\mathbf{n}$ on the interior of $\Sigma$ together to get that $\mathbf{n}$ is injective.

\begin{lemma}
\label{lem:covering}
Let $\Sigma$ be a connected Riemann surface with genus 0 and $b\geq 2$ boundary components. Suppose that $F: \Sigma\to \mathbb{S}^2$ is a smooth map whose differential is nondegenerate and whose restriction to the boundary $F|_{\pr \Sigma}$ is injective. Then $F$ is injective.
\end{lemma}
\begin{proof}
As $dF$ is nondegenerate, it follows that the restriction $F|_{\mathring{\Sigma}}$ to the interior is a local diffeomorphism, and that $F|_{\pr \Sigma}$ smoothly embeds $\pr \Sigma$ as a disjoint union of $b$ circles $F(\pr \Sigma) = \bigsqcup_{i=1}^b \Gamma_i$ in $\mathbb{S}^2$. The complement $\mathbb{S}^2\setminus F(\pr \Sigma) = \bigsqcup_{i=0}^b \Omega_i$, where $\Omega_i$ is a disk for $i>0$, and $\Omega_0$ is genus 0 with $b$ boundary components. 

Let $S = F(\mathring{\Sigma})$ be the image of the interior. As $F|_{\mathring{\Sigma}}$ is a local diffeomorphism (in particular, an open map), $S$ is open in $\mathbb{S}^2$ and $\pr S \subset F(\pr \Sigma)$. 

For each $\Omega_i$, we claim that either $\Omega_i \cap S=\emptyset$ or $\Omega_i \subset S$. Indeed, $\Omega_i\cap S$ is clearly open. As $F(\pr\Sigma) \cap \Omega_i =\emptyset$, we have $S\cap \Omega_i = F(\Sigma)\cap \Omega_i$, but the latter is relatively closed in $\Omega_i$ as $\Sigma$ is compact. Thus $\Omega_i\cap S$ is open and closed in the connected region $\Omega_i$, which establishes the claim. 

Now suppose that $\Omega_i \subset S$ for $i>0$. Then $\Omega_i$ is contractible, and it follows from the homotopy lifting property that the boundary component $\Gamma_i = \pr \Omega_i$ must be contractible in $\Sigma$. This contradicts the assumed topology of $\Sigma$, so we conclude that $\Omega_i \cap S = \emptyset$ for $i>0$. As $S$ is open, it follows that $F(\pr \Sigma) \cap S = \emptyset$ and hence $S=\Omega_0$. 

It follows that the preimage of a (sufficiently small) neighbourhood of any boundary point $p\in F(\pr \Sigma)$ must consist of components diffeomorphic to half-discs (not discs), and so $F$ is a covering map on the whole compact surface $\Sigma$. As $\overline{S}$ is connected, the degree of the covering is constant. As $dF$ is nondegenerate, we may assume without loss of generality that $\det dF>0$, and so we may compute the degree at any value by simply counting pre-images. In particular, the degree must be 1 on $\pr \Sigma$, so (as $\mathbb{S}^2$ is connected) it must be 1 everywhere, and we conclude that $F$ is injective. 
\end{proof}

The lemma implies $\mathbf{n} : \Sigma 
\to \mathbb{S}^2$ is injective. As $\mathbf{n}$ factors through $\nu$, this completes the proof of Proposition \ref{prop:emb-dual}.

\section{Uniqueness problems for minimal annuli}
\label{sec:continuity}

In this section, we explore a continuous family of uniqueness problems for minimal surfaces. Specifically, this family will include free boundary minimal surfaces in geodesic balls $B_R \subset \mathbb{S}^3$, with the limiting cases being free boundary minimal surfaces in $B_{\frac{\pi}{2}} \subset \mathbb{S}^3$ (equivalently, closed minimal surfaces in $\mathbb{S}^3$ with a reflection symmetry) and free boundary minimal surfaces in $\mathbb{B}_1 \subset \mathbb{R}^3$. Our goal is to provide a clear, unified presentation for readers interested in these settings, and to suggest potential avenues for which this unified approach may be illuminating.

Recall that, in our terminology, `FBMS' refers to minimal immersions $(\Sigma,\pr\Sigma) \looparrowright (M,S)$ which contact $S$ with angle $\pit$ along $\pr\Sigma \subset S$. 

Before continuing, we briefly remind the reader of some well-known rigidity problems. The setting of closed minimal surfaces in $\mathbb{S}^3$ is the most well-studied. Here we highlight the Lawson conjecture \cite{Law70b}, which concerns the uniqueness of embedded minimal tori, and was proven by Brendle \cite{Br12}:

\begin{theorem}[Lawson conjecture; proven in \cite{Br12}]
The unique embedded minimal torus in $\mathbb{S}^3$ is the Clifford torus. 
\end{theorem}

Given a closed minimal surface $\Sigma \looparrowright \mathbb{S}^3$, let $Q^{\mathrm{A}}$ denote its index form (the reader may consult the formula in Section \ref{sec:prelim-stability}, noting that the boundary is empty). We denote by $\ker Q^{\mathrm{A}}$ the kernel of the index form $Q^{\mathrm{A}}$, and we refer to this as the space of \textit{Jacobi fields}. The nullity is $\nul(\Sigma) = \dim \ker Q^{\mathrm{A}}$. Let $\mathcal{K}_0$ denote the space of Killing fields on $\mathbb{S}^3$, that is, the space of rotational vector fields. As is well-known, if $K\in\mathcal{K}_0$ then the normal projection $\langle K, \nu \rangle$ must be a Jacobi field since $K$ generates a 1-parameter subgroup of ambient isometries.

In \cite[Page 128]{yau1986nonlinear}, \cite[Problem 31]{yau1990open} \cite[Section I.G]{yau2000review}, Yau asked whether there are continuous families of embedded minimal surfaces in $\mathbb{S}^3$ that are nontrivial (are not related by ambient isometries). The following is the infinitesimal version of Yau's question:

\begin{question}
\label{conj:nullity-sph}
For any closed embedded minimal surface $\Sigma\subset\mathbb{S}^3$, are all Jacobi fields induced by Killing fields?
\end{question} 

For FBMS in the Euclidean ball $\mathbb{B}^3$, the analogue of the Lawson conjecture is the following conjecture of Nitsche and Fraser-Li, which concerns uniqueness of embedded minimal \textit{annuli}: 

\begin{conjecture}[\cite{Ni85,FL14}]
The unique embedded free boundary minimal annulus in $\mathbb{B}^3$ is the critical catenoid.
\end{conjecture}

Henceforth, we use the abbreviation `FBMA' for `free boundary minimal annulus'. We remark that the critical catenoid is known to be unique amongst rotationally symmetric embedded minimal annuli (see \cite[Section 3]{FS11}, \cite[Corollary 3.9]{KL21}).

\subsection{Families of rigidity problems} 
\label{sec:fR-surfaces}

To study FBMS in $(\mathbb{S}^3, \pr B_R)$ as $R$ varies, it is convenient to recall the various equivalent formulations from Section \ref{sec:cap-model}:

\begin{enumerate}[(A)]
\item FBMS in $(\mathbb{S}^3, \pr B_R)$ for $R>0$, or in $(\mathbb{R}^N, \pr\mathbb{B}_1)$ for $R=0$.
\item FBMS in $(\mathbb{M}^3_{\kappa_R} , \pr B_{R\csc R})$, where $\mathbb{M}^3_{\kappa_R}$ is the space form of curvature $\kappa_R= \sin^2 R$.
\item Weighted ($f_R=2\log(1+\cos R\cos \rho)$) FBMS in $(\mathbb{S}^3, \pr B_\pit)$.
\end{enumerate}

We emphasise that by default $\mathbb{S}^3$ will carry the standard metric $g_1$ (as in (C)).  

In this section we will refer to the pairs $(\mathbb{M}^3_{\kappa_R} , \pr B_{R\csc R})$ in (B) as \textit{rescaled caps}. 

The equivalent problems above are each invariant under a 3-dimensional group of ambient isometries. Analogous to the closed case, in each setting above (by abuse of notation) we denote the Lie algebra of this 3-dimensional group as $\mathcal{K}_0$. Concretely, in setting (C), this group is precisely the subgroup $\mathrm{SO}(4)^o$ of rotations that fix $o=e_0$ (and consequently preserve the hemisphere $B_\pit$, as well as the radial functions $h_R, f_R$); the Lie algebra $\mathcal{K}_0$ is the space of rotational vector fields in 2-planes orthogonal to $e_0$.

We say that $\Sigma$ is \textit{rotationally symmetric} if it is invariant under the flow of a nontrivial element of $\mathcal{K}_0$. Equivalently, there is a nonzero vector field $K\in\mathcal{K}_0$ for which $\langle K,\nu\rangle \equiv0$ on $\Sigma$. 

The following uniqueness problem is a natural analogue of the Lawson and Nitsche conjectures, stated in terms of formulation (B):

\begin{conjecture}
\label{conj:lawson-R}
Let $R\in[0,\pit]$ and consider the space form $\mathbb{M}^3_{\kappa_R}$ of curvature $\kappa_R= \sin^2 R$. Suppose $(\Sigma,\pr\Sigma)\hookrightarrow(B_{R\csc R}, \pr B_{R\csc R})$ is an embedded FBMA constrained in $B_{R\csc R}\subset \mathbb{M}^3_{\kappa_R}$. Then $\Sigma$ is rotationally symmetric. 
\end{conjecture}

There is a subtle point as to whether one should include the hypothesis that $\Sigma$ is \textit{constrained} in $B_{R\csc R}$, or if it should be a conclusion. For $R=0$ any such FBMS must be constrained, but this is not clear for $R>0$. However, it is not difficult to show that embedded rotationally symmetric examples must be constrained (see Corollary \ref{cor:rot-constrained} below), so we find it plausible that Conjecture \ref{conj:lawson-R} may be true without the constraint hypothesis. It would be interesting to know, for any genus, if any embedded minimal surface must be constrained.

The special case $R=\pit$ corresponds to the Lawson conjecture and follows quite readily from Brendle's work (see Lemma \ref{lem:fb-lawson}). On the other hand, the special case $R=0$ is precisely the Nitsche conjecture. For $R\in (0,\pit)$, Lima-Menezes \cite{LM23} were able to establish uniqueness of free boundary minimal annuli under an additional `immersed by first eigenfunctions' assumption. In \cite{Me23} (see also \cite{NZ25b} for the unconstrained setting), it was observed that the latter condition is essentially a condition on the spectral index $\ind(Q^{\mathrm{S}})$ (recall the definitions in Section \ref{sec:prelim-stability}). As a result, one has the following result, which is essentially a restatement of \cite[Proposition 7.4]{NZ25b} (and \cite[Theorem 6.6]{FS16}) in terms of formulation (B):

\begin{proposition}[\cite{LM23,Me23,NZ25b}]
\label{prop:FB-first-efns}
Let $R\in[0,\pit]$ and suppose $(\Sigma,\pr\Sigma)$ is an embedded free boundary minimal annulus in the rescaled cap $(\mathbb{M}^N_{\kappa_R} , \pr B_{R\csc R})$. If $\ind(Q^{\mathrm{S}})\leq 1$ then $\Sigma$ is rotationally symmetric. 
\end{proposition} 

We remark that uniqueness within the rotational class does not seem to be clear for $R\in(0,\pit)$, unlike the cases $R=0,\pit$. We find that the problem within the rotational setting is of a somewhat distinct character, so we will not discuss it further in this work. (In \cite{dO24}, de Oliveira constructs a rotational FBMA in $B_{R(a)}\subset \mathbb{S}^3$ for each $a\in(-\frac{1}{2},0)$, and it is expected that the map $a\mapsto R(a)$ should be bijective, but proving this appears to be somewhat subtle - see his Remark 3.5.)

Given a FBMS $(\Sigma,\pr\Sigma)\looparrowright (\mathbb{M}^3_{\kappa_R} , \pr B_{R\csc R})$, we again define the space of \textit{Jacobi fields} to be the kernel, $\ker Q^{\mathrm{A}}$, of the index form $Q^{\mathrm{A}}$. (Note that elsewhere, `Jacobi fields' may simply refer to functions satisfying the Jacobi equation on the interior, but being in $\ker Q^{\mathrm{A}}$ also comprises a boundary condition.) The nullity is $\nul(\Sigma) = \dim \ker Q^{\mathrm{A}}$. 

As above, if $K\in\mathcal{K}_0$, the normal projection $\langle K, \nu \rangle$ must be a Jacobi field (as $K$ generates a continuous family of ambient isometries preserving the free boundary setting).

The following is an analogue of Question \ref{conj:nullity-sph}, stated in terms of formulation (B):

\begin{question}
\label{conj:nullity}
Suppose $(\Sigma,\pr\Sigma)$ is an embedded FBMS in the rescaled cap $(\mathbb{M}^3_{\kappa_R} , \pr B_{R\csc R})$. Are all Jacobi fields on $\Sigma$ induced by Killing fields in $\mathcal{K}_0$?
\end{question}

For \textit{annuli}, an affirmative answer to Question \ref{conj:nullity} should be a weaker version of the rotational symmetry Conjecture \ref{conj:lawson-R}. As such, we are most interested in this special case of annuli. In the remainder of this section, we will discuss a continuity approach towards Conjecture \ref{conj:lawson-R}.

\begin{definition}
Let $I_{\mathrm{rot}}$ be the set of $R\in[0,\pit]$ such that every embedded FBMA $(\Sigma,\pr\Sigma)\hookrightarrow(B_{R\csc R} , \pr B_{R\csc R})$, constrained in $B_{R\csc R}\subset \mathbb{M}^3_{\kappa_R}$, is rotationally symmetric.
\end{definition}

Note that certainly $\pit \in I_{\mathrm{rot}}$, as we may reflect an embedded FBMA $\Sigma \hookrightarrow B_\pit \subset \mathbb{S}^3$ through the equator $\pr B_\pit$ to obtain an embedded minimal torus in $\mathbb{S}^3$, which must be the Clifford torus by Brendle's resolution of the Lawson conjecture \cite{Br12}. That is, we have:

\begin{lemma}
\label{lem:fb-lawson}
Let $(\Sigma,\pr\Sigma) \hookrightarrow (B_\pit, \partial B_\pit)$ be an embedded FBMA constrained in  $B_\pit \subset \mathbb{S}^3$. Then $\Sigma$ is isometric to half of the Clifford torus. 
\end{lemma}

We will see in Section \ref{sec:I-open} that $I_{\mathrm{rot}}$ is open in $[0,\pit]$. The essential point is that, having reframed the work of Lima-Menezes \cite{LM23} in terms of $\ind(Q^{\mathrm{S}})$ (Proposition \ref{prop:FB-first-efns}), openness will follow from lower semicontinuity of the index. For this we use the rescaled formulation (B). 

Determining whether $I_{\mathrm{rot}}$ is closed is more subtle. The naive idea is that a non-rotational example could lead to non-rotational examples for nearby values of $R$ by the implicit function theorem, but this idea is complicated by the possibility that the original example has nontrivial nullity. Nevertheless, we observe that if Question \ref{conj:nullity} is \textit{affirmative for constrained annuli}, then this idea succeeds, and $I_{\mathrm{rot}}$ is closed. Of course, we expect that Question \ref{conj:nullity} may be a much more difficult conjecture, but we find it interesting to make this observation concrete and also to clarify the technical apparatus. This will be carried out in Section \ref{sec:I-closed}, in which we use the weighted formulation (C). 

\subsubsection{Rotationally symmetric minimal annuli}
\label{sec:rot-sym-annuli}

Before moving on to the continuity discussion, we close this subsection by returning to the discussion of the constrained condition for rotationally symmetric minimal annuli. We remark that there may be more direct proofs, but that the following is rather efficient given the consequences of the two-piece property that we have previously established in this paper. 

Recall that an $(R, \gamma)$-minimal surface is a radial graph if $o \not \in \Sigma$ and $g_1(\nu, \partial_\rho) \neq 0$ on the interior $\Sigma$. 

\begin{lemma}
\label{lem:rot-constrained}
Let $(\Sigma,\pr\Sigma)$ be a rotationally symmetric $(R,\gamma)$-minimal annulus, $R\leq \pit$. Then $\Sigma$ is constrained in $B_R$ if and only if it is a radial graph.

\end{lemma}
\begin{proof}

Recall that by Remark \ref{rmk:constrained}, for $R\leq \pit$ the property of being constrained is equivalent to $x_0\geq 0$ on $\Sigma$. Also recalling that $e_0^{T\mathbb{S}^3} = - \sin \rho \,\partial_\rho$, the property of being a radial graph is equivalent to that $\nu_0 = -\sin \rho\, g_1(\nu, \partial_\rho)$ does not change sign on $\Sigma$. Thus, it is sufficient to show that $x_0$ does not change sign if and only if $\nu_0$ does not change sign.

Without loss of generality we may assume $\Sigma$ is symmetric about the $x_2x_3$-plane. As in \cite{Br13survey}, a rotationally symmetric (about the $x_2x_3$-plane) minimal immersion in $\mathbb{S}^3$ takes the form 
\[x(s,t) = (r(t)\cos t, r(t)\sin t, \sqrt{1-r(t)^2}\cos s, \sqrt{1-r(t)^2}\sin s),\]
where $r$ is a smooth function with values in $(0,1)$ satisfying 
\begin{equation}
    \label{eq:min-ode}
 r(t)(1-r(t)^2)r''(t) = (1-2r(t)^2)(2r'(t)^2+r(t)^2(1-r(t)^2). 
 \end{equation}
Moreover, there is a choice of unit normal that satisfies \[\nu_0 = \langle \nu,e_0\rangle=\frac{ r(t)(1-r(t)^2)\cos t + r'(t)\sin t }{\sqrt{r'(t)^2 + r(t)^2 (1-r(t)^2)} }.\] 

We claim that 
\[ \pr_t x_0 = 0 \Leftrightarrow \pr_t \nu_0 =0 \Leftrightarrow r'(t)\cos t = r(t)\sin t.\]
Indeed, we can directly compute \[\pr_t x_0 = r'(t)\cos t - r(t)\sin t,\] and using the minimality ODE (\ref{eq:min-ode}),
\[ \pr_t \nu_0 = \frac{r(t)^2(r'(t)\cos t - r(t)\sin t)}{\sqrt{r'(t)^2+r(t)^2(1-r(t)^2)}}.\]

Having established the claim, we see that the sign of $\nu_0$ at a critical point is the same as the sign of 
\[ r(t)(1-r(t)^2)\cos t +r'(t)\sin t  =  r(t)\cos t \left((1-r(t)^2)+r(t)\tan^2 t\right),\]
and hence the same sign as $r(t)\cos t = x_0$. 

Now suppose that $\nu_0$ does not change sign. By the above discussion, we see that $x_0$ must have the same sign at all interior critical values. But now $x_0=\cos R \geq 0$ on $\pr\Sigma$, so (using the constant contact angle condition when $\cos R=0$, as $\sin\gamma>0$) $x_0$ must be positive somewhere in the interior of $\Sigma$. Then $x_0$ must have a positive maximum on $\Sigma$, so, by the claim, $x_0$ also has a nonnegative minimum.  Thus $x_0\geq 0$ on $\Sigma$, so $\Sigma$ must be constrained. 

Similarly, for the converse suppose that $x_0$ does not change sign. Then again by the above discussion, $\nu_0$ has the same sign at all interior critical values. 
But $\nu_0$ is a constant on $\pr\Sigma$, either $\pm\sin R \cos\gamma$, so arguing similarly to the above we will find that $\nu_0$ does not change sign on $\Sigma$. (Indeed, if $\sin R\cos\gamma >0$ then $\nu_0$ has a strict sign on the boundary and so has that sign on the interior. If $\sin R\cos \gamma=0$ then Lemmas \ref{lem:boundary-Rgamma-cap} and Corollary \ref{lem:same-sign} imply that the derivative of $\nu_0$ has a constant sign along the boundary, and $\nu_0$ has that sign on $\Sigma$.)

For the converse direction, we could also use the dual as follows: Suppose that $\Sigma$ is a radial graph. Consider the polar dual $\tilde{\Sigma}$, which is a rotationally symmetric $(\tilde{R},\tilde{\gamma})$-minimal annulus. By the dual construction, $\tilde{x}_0 = \varepsilon \nu_0$ does not change sign, so we must have $\tilde{R}\leq \pit$ and that $\tilde{\Sigma}$ is constrained. Applying the above argument to $\tilde{\Sigma}$, we deduce that $\tilde{\Sigma}$ is a radial graph, hence $x_0 =  \pm  \tilde{\nu}_0$ does not change sign, so $\Sigma$ is also constrained. 
\end{proof}

Note that any rotationally symmetric embedded FBMS $(\Sigma,\pr\Sigma)\looparrowright(\mathbb{S}^3,\pr B_R)$ must satisfy the two-piece property (that is, the conclusions of Corollary \ref{cor:two-piece}): Using a rotation by $\pi$, the number of components of $\Sigma$ in each half-space is the same, hence the same as the number of components of $\Sigma$ itself. 
By Proposition \ref{prop:radial graph}, it follows that any such $\Sigma$ is a radial graph. As any radial graph must be embedded, we have shown:

\begin{corollary}
\label{cor:rot-constrained}
Let $(\Sigma,\pr\Sigma)$ be a rotationally symmetric FBMA in $(\mathbb{S}^3, \pr B_R)$, $R\leq \pit$. Then $\Sigma$ is constrained in $B_R$ if and only if it is embedded. 
\end{corollary}

\subsection{Continuity of spectral index and openness of $I_{\mathrm{rot}}$}
\label{sec:I-open}

In this subsection, we show that $I_{\mathrm{rot}}$ is open. This argument relies on semicontinuity properties of the (spectral) index, and does not require the nullity conjecture.

\begin{proposition}
The set $I_{\mathrm{rot}}$ is open in $[0,\pit]$. 
\end{proposition}
\begin{proof}

Suppose $R \in I_{\mathrm{rot}}$. It suffices to show that for any sequence $R_j \in [0,\pit]$, $R_j \to R$, if $(\Sigma_j, \pr\Sigma_j)$ are embedded FBMA in $B_{R_j\csc R_j} \subset \mathbb{M}^3_{\kappa_{R_j}}$, then there are arbitrarily large $j$ for which $\Sigma_j$ is rotationally symmetric. 

Indeed, suppose we have $R_j, \Sigma_j$ as above. By the compactness theory for free boundary minimal surfaces of fixed topology (see Appendix \ref{sec:compactness}), after passing to a subsequence, we may assume that the $\Sigma_j$ converge smoothly to an embedded FBMA $(\Sigma, \partial \Sigma) \hookrightarrow (B_{R \csc R}, \pr B_{R\csc R})$ in $\mathbb{M}_{\kappa_R}$. (Note that we may apply case (i) of Lemma \ref{lemma:compactness} if $R<\pit$, and case (ii) when $R=\pit$.)

\textbf{Claim:} Index plus nullity is upper semicontinuous, that is, \[\mathrm{ind}_0(Q^{\mathrm{S}}_\Sigma) \geq \limsup_{j\to\infty} \mathrm{ind}_0(Q^{\mathrm{S}}_{\Sigma_j}).\] 

The proof of the claim is rather standard, but it will be most precise to parametrise $\Sigma_j$ over $\Sigma$ (for large $j$), that is, we have a sequence of embeddings $x_j:(\Sigma,\pr\Sigma)\hookrightarrow(B_{R_j \csc R_J} , \pr B_{R_j\csc R_j})$ in $\mathbb{M}^3_{\kappa_{R_j}}$. In this way, the forms $Q^{\mathrm{S}}_{\Sigma_j}$ may be defined over $C^\infty(\Sigma)$.

It suffices to show that if $\ind_0(Q^{\mathrm{S}}_{\Sigma_j})\geq k$ on a subsequence, then $\ind_0(Q^{\mathrm{S}}_\Sigma)\geq k$. Indeed, suppose there is such a subsequence and let $\mathcal{V}_j$ be a $k$-dimensional subspace of $C^\infty(\Sigma)$ on which $Q^{\mathrm{S}}_{\Sigma_j}$ is negative semidefinite. We may diagonalise $Q^{\mathrm{S}}_{\Sigma_j}$ on each $\mathcal{V}_j$, giving an $L^2_{\Sigma_j}$-orthonormal (and $Q^{\mathrm{S}}_{\Sigma_j}$-orthogonal) basis of eigenfunctions $u_{j,i}$. Taking a further subsequence, these will converge smoothly to an $L^2_\Sigma$-orthonormal set $u_i = \lim_{j\to\infty} u_{i,j}$. Moreover, by the smooth convergence of the embeddings $x_j$ and the explicit form of $Q^{\mathrm{S}}$, it is easy to verify that $Q^{\mathrm{S}}_{\Sigma}$ is negative semidefinite on $\mathcal{V} = \spa\{u_i\}$, and thus $\ind_0(Q^{\mathrm{S}}_{\Sigma}) \geq \dim \mathcal{V}=k$ as claimed.

As $R \in I_{\mathrm{rot}}$, we know that $\Sigma$ must be rotationally symmetric. For $R\in(0,\pit)$, \cite[Theorems 2 and 4]{LM23} implies that the first two Steklov eigenspaces are $\spa\{x_0\}$ (with $Q^{\mathrm{S}}_{\Sigma}(x_0) < 0$) and $\spa\{x_1,x_2,x_3\}=\mathrm{ker}(Q^\mathrm{S}_{\Sigma})$, and it follows that $\ind_0(Q^{\mathrm{S}}_\Sigma)=4$. For $R=0$, $\Sigma$ must be the critical catenoid, for which the first two Steklov eigenspaces are $\spa \{1\}$ (with $Q^{\mathrm{S}}_{\Sigma}(1) < 0$) and $\spa\{x_1,x_2,x_3\}= \mathrm{ker}(Q^\mathrm{S}_{\Sigma})$ \cite{FS11}, which again implies $\ind_0(Q^{\mathrm{S}}_\Sigma)=4$. For $R=\pit$, $\Sigma$ must be half the Clifford torus, which similarly has $\ind_0(Q^{\mathrm{S}}_\Sigma)=4$.

From these observations and the claim, it follows that $\mathrm{ind}_0(Q^{\mathrm{S}}_{\Sigma_j}) \leq 4$, hence $\ind(Q^{\mathrm{S}}_{\Sigma_k})\leq 1$ for sufficiently large $k$. By Proposition \ref{prop:FB-first-efns}, it follows that each such $\Sigma_k$ is rotationally symmetric. 

\end{proof}

\begin{corollary}
There exists $\underline{R}$ so that for all $R\in (\underline{R},\frac{\pi}{2}]$, every embedded FBMA $(\Sigma,\pr\Sigma)\hookrightarrow(\mathbb{M}^3_{\kappa_R} , \pr B_{R\csc R})$ is rotationally symmetric.
\end{corollary}

\subsection{The inverse function theorem and closedness of $I_{\mathrm{rot}}$}
\label{sec:I-closed}

We have the following perturbation argument for FBMS in rescaled caps $(\mathbb{M}^N_{\kappa_R} , \pr B_{R\csc R})$.

\begin{proposition}\label{prop:ift-caps}
Let $R \in [0, \frac{\pi}{2}]$. Suppose that $x : (\Sigma, \pr\Sigma) \looparrowright (\mathbb{M}^N_{\kappa_R} , \pr B_{R\csc R})$ is a FBMS. Assume every element of $\ker Q^{\mathrm{A}}$ is induced by a Killing field in $\mathcal{K}_0$. 

Then there exists $\delta>0$ and a smooth family of FBMS $\tilde{x}_{\tilde{R}}: (\Sigma, \partial \Sigma) \looparrowright (\mathbb{M}^N_{\kappa_{\tilde{R}}} , \pr B_{\tilde R\csc \tilde R})$ for $\tilde{R} \in (R-\delta,R+\delta)\cap[0,\pit]$.

Additionally, if $x$ is an embedding, then $\tilde{x}$ is an embedding. If $x$ is not rotationally symmetric, then $\tilde{x}$ is not rotationally symmetric. If $x$ is constrained, then $\tilde{x}$ is constrained.
\end{proposition}

\begin{proof}
We have stated the proposition in terms of setting (A) of Section \ref{sec:fR-surfaces}, but as discussed in that section we may recast the problem in terms of setting (C), that is, for weighted $f_R$-minimal surfaces in $(\mathbb{S}^3, \partial B_\pit)$. So we suppose $x : (\Sigma, \partial \Sigma) \looparrowright (\mathbb{S}^3, \partial B_\pit)$ is a weighted FBMS with respect to $f_R = (1 + \cos R \cos \rho)^{-1}$. 

To set out notation for the remainder of this proof, we will fix $x,R$ and aim to produce weighted FBMS in $(\mathbb{S}^3,\pr B_\pit)$ with respect to $f_{R+s}$, which we refer to as $f_{R+s}$-FBMS. We recall that $\bar{\nabla}$ denotes the Levi-Civita connection on $(\mathbb{S}^3, g_1)$ and $\nabla$ the connection on $(\Sigma, g)$ with $g = x^\ast g_1$. We show that for $s$ near $0$, we can find an $f_{R+s}$-FBMS,  $\tilde{x} : (\Sigma, \partial \Sigma) \looparrowright (\mathbb{S}^3, \partial B_\pit)$.

As usual, let $\nu$ denote a choice of unit normal on $\Sigma$ and $\eta = \partial_\rho$ the outward-pointing conormal of $\partial \Sigma$ in $\Sigma$. Let us define $I = [-R, \frac{\pi}{2}-R]$ and for $\epsilon > 0$ small 
    \[
    \Omega:= \Big\{u\in  C^{2,\alpha}(\Sigma) \times \mathbb{R} : \|u\|_{C^{2,\alpha}} <\epsilon, \;  \; \partial_\eta u = 0 \text{ on } \partial \Sigma\Big\},
    \]
which is a Banach manifold. Here the $C^{2,\alpha}(\Sigma) = C^{2,\alpha}(\Sigma, g_1)$ for some fixed choice of $\alpha > 0$.  For $u \in \Omega$, define $ x_{u} : \Sigma \looparrowright \mathbb{S}^3$ by 
    \begin{align*}
    x_u(p) &= \exp^{\mathbb{S}^3}_{F(p)} (u(p) \nu(p)) = \cos(u(p)) x(p) + \sin(u(p))\nu(p) \in  \mathbb{S}^3 \subset \mathbb{R}^4. 
    \end{align*}
Note that $x_u$ is a proper immersion $(\Sigma,\pr\Sigma) \looparrowright (\mathbb{S}^3, \pr B_\pit)$ if $\epsilon > 0$ is sufficiently small. This follows from the free boundary condition, using $\partial B_{\frac{\pi}{2}}$ is totally geodesic in $\mathbb{S}^3$.  

Moreover, we can check that the Neumann condition implies that $x_u$ contacts $\pr B_\pit$ with angle $\pit$ along $\pr \Sigma$. Indeed, suppose $p \in \partial \Sigma$ and, using $x$ contacts $\pr B_\pit$ at angle $\pit$, choose local coordinates $(x_1, x_2)$ on $\Sigma$ near $p$ such that $\partial_1x(p) = \partial_\rho$ and $\partial_2 x(p)$ is tangent to $\partial B_{\frac{\pi}{2}}$. As by assumption we have $\partial_1 u = \partial_\eta u = 0$ along $\partial \Sigma$, it follows that 
    \begin{align*}
    \partial_1 x_u &= \cos (u) \partial_\rho -\sin(u) (\partial_1 u) x+ \cos (u) (\partial_1 u) \, \nu + \sin(u) A_{11} \partial_\rho \\
    & = (\cos(u) + \sin(u) A_{11})\partial_\rho
    \end{align*}
is tangent to $\Sigma$ in the image of $dx_u$. So for sufficiently small $\epsilon > 0$, we conclude 
    \[
    \langle \nu_u(p), \partial_\rho \rangle  = 0.
    \]

Having defined the immersions $x_u$, we consider the mean curvature $H_u$ of $x_u$ with respect to $\nu_u$ (the vertical unit normal vector, consistent with $\nu$). The analysis is essentially similar to \cite[Proposition 5.6]{DM09}. Consider the map 
    \[
    \mathcal{H} : \Omega  \times I\to C^\alpha(\Sigma)
    \]
given by 
    \[
    \mathcal{H}(u, s) = H_u - g_1\big( \bar{\nabla} f_{R+s} \circ x_u, \nu_u \big). 
    \]
By construction, $\mathcal{H}(0, 0) = 0$. We would like to show that for all $s$ sufficiently close to $0$, there exists a $u_s$ (close to zero) such that $\mathcal{H}(u_s, s) = 0$. This is an application of the implicit function theorem under the assumption that the kernel $Q^A$ is spanned by vector fields generating ambient isometries. The remaining details can be found in Appendix \ref{sec:ift-caps}.  One point worth highlighting is that the differential certainly has kernel corresponding to $\ker (Q^{\mathrm{A}})$, even under our assumption that the Jacobi fields are induced by Killing fields in $\mathcal{K}_0$. Using a trick of \cite{DM09} (see Lemma \ref{lem:killing-Hf-orthogonality}), however, we are able to ensure that the implicit function theorem applies and gives the desired solutions for nearby $R$. 
\end{proof}

\begin{corollary}
    If the nullity Conjecture \ref{conj:nullity} holds for constrained annuli, then set $I_{\mathrm{rot}}$ is closed. 
\end{corollary}

\begin{proof}
    Suppose $R \not \in I_{\mathrm{rot}}$. Then there exists an embedded FBMA $(\Sigma,\pr\Sigma)\hookrightarrow(B_{R\csc R}, \pr B_{R\csc R})$  constrained in $B_{R\csc R}\subset \mathbb{M}^3_{\kappa_R}$, which is not rotationally symmetric. By assumption, the only Jacobi fields of $\Sigma$ are induced by Killing fields in $\mathcal{K}_0$. The previous proposition then implies that there exist embedded FBMAs $\tilde{x}_{\tilde{R}} : (\Sigma,\pr\Sigma)\hookrightarrow(B_{\tilde R \csc \tilde R} , \pr B_{\tilde R\csc \tilde R})$ for $\tilde{R} \in (R- \delta, R + \delta) \cap [0, \frac{\pi}{2}]$, constrained in $B_{\tilde{R}\csc \tilde{R}}\subset \mathbb{M}^N_{\kappa_{\tilde{R}}}$, none of which are rotationally symmetric. It follows that $[0, \frac{\pi}{2}]\cap (R- \delta, R + \delta) \subset [0, \frac{\pi}{2}] \setminus I_{\mathrm{rot}}$. Then $[0, \frac{\pi}{2}] \setminus I_{\mathrm{rot}}$ is open, so $I_{\mathrm{rot}}$ is closed. 
\end{proof}

\appendix

\section{Formulae for conformal change}
\label{sec:conformal-change}

Fix a Riemannian manifold $(M^{n+1}, \bar{g})$ and an isometrically immersed hypersurface $(\Sigma^n, g) \looparrowright (M^{n+1}, \bar{g})$ with unit normal $\nu$. Suppose $\phi \in C^{\infty}(M)$. Let $\tilde{g} = e^{2\phi} \bar{g}$ denote the conformal metric on $M$ and also, by minor abuse of notation, the induced conformal metric on $\Sigma$. 

The conformal change equation for the Ricci curvature is: 
\[\Ric_{\tilde{g}} = \Ric_g - (n-1)(\nabla_{\bar{g}}^2 \phi - d\phi\otimes d\phi) - (\Lap_{\bar{g}} \phi + (n-1) |d\phi|^2_{\bar{g}}) \bar{g}.\] 
For any $u \in C^{\infty}(M)$, its Hessian and Laplacian transform according to the equation: 
\begin{align*}\nabla^2_{\tilde{g}}u &= \nabla^2_{\bar{g}} u - d\phi \otimes du - du\otimes d\phi + \bar{g}(\nabla^{\bar{g}} u,\nabla^{\bar{g}}\phi) \bar g,\\
\Delta_{\tilde{g}} u &= e^{-2\phi}\big(\Delta_{\bar{g}} u + (n-1) \bar{g}( \nabla^{\bar g} \phi,\nabla^{\bar g} u)\big).
\end{align*}

For the immersed hypersurface, we recall that the Laplacian operators acting on $u \in C^{\infty}(M)$ and its restriction $u |_{\Sigma}$ are related via: 
\[
\Delta_{\bar g} u = \Delta_g( u|_{\Sigma}) + (\nabla_{\bar{g}}^2 u)(\nu, \nu) + H^g\cdot \partial_\nu u,
\]
with $H^g := \sum_i \bar{g}(\bar{\nabla}_{e_i} \nu, e_i)$ and $\partial_\nu u = \bar{g}(\bar \nabla u, \nu )$. 

After conformal change, the unit normals of the immersion are related by the formula $\tilde{\nu} = e^{-\phi} \nu$. The second fundamental form of the immersion is given by
\begin{equation}\label{eq:2ff-conf-change}
A^{\tilde{g}}_{ij} = e^{\phi}(A^g_{ij} + (\pr_\nu \phi)g_{ij}).
\end{equation}
Consequently, 
\[H^{\tilde{g}} = \tilde{g}^{ij}A^{\tilde{g}}_{ij} = e^{-\phi}(H^g + n(\pr_\nu \phi)),  \]
and
\[|A^{\tilde{g}}|_{\tilde{g}}^2 = \tilde{g}^{ij} \tilde{g}^{kl} A^{\tilde{g}}_{ij} A^{\tilde{g}_{kl}} = e^{-2\phi}( |A^g|_g^2 + 2(\pr_\nu\phi)H^g + n(\pr_\nu \phi)^2 ). \]

\begin{remark}\label{rem:ricci}
If $g = e^{2\phi} \delta$, then the formula above and a standard computation give
\begin{align*}
\Ric_M& = -(n-1)\Big(\phi'' + \frac{\phi'}{r} \Big) dr^2  -\Big(\phi'' + (2n-3) \frac{\phi'}{r} +(n-2)(\phi')^2 \Big) r^2 g_{\mathbb{S}^{n-1}}.
\end{align*}
In particular, the radial component of the Ricci curvature is largest if and only if
\begin{align*}
& -(n-1)\Big(\phi'' + \frac{\phi'}{r} \Big)  >  -\Big(\phi'' + (2n-3) \frac{\phi'}{r} +(n-2)(\phi')^2 \Big)
    \iff  \Big(\frac{\phi'}{r}\Big)' - r \Big(\frac{\phi'}{r}\Big)^2 < 0.
\end{align*}
This is condition ($\mathrm{A}'$) in Section \ref{sec:frankel}. 
\end{remark}

\section{Additional results on stable hypersurfaces in half-balls}
\label{sec:stable-details}

In this appendix, we give further details on the classification of stable hypersurfaces in a half-ball $B_R\subset M$, where $M$ in as in Section \ref{sec:radial-setup}, and also complete the resulting alternative proof of the half-space intersection property.

\begin{proof}[Proof of Proposition \ref{prop:half-space-bernstein}]
Note that $u_a$ certainly vanishes on $D_a$ and is positive on $\Sigma\subset \mathcal{H}_a$, so by standard arguments as in \cite[Proposition 14]{NZ24} it is a valid test function for the stability inequality. Using Lemma \ref{lem:stability-test}, the index form applied to $u_a$ is
\[\begin{split}
Q^{\mathrm{A}}(u_a) &= -\int_\Sigma u_a L u_a d\mu_g + \int_{\pr \Sigma} u_a (\pr_\eta - k_S(\bar{\nu}_{\bar{g}},\bar{\nu}_{\bar{g}}))u_a d\sigma_g
\\&= -\int_\Sigma  \left( -ne^{ -2\phi} \left(\phi''-\frac{\phi'}{r} - (\phi')^2 \right) \langle \nu_\delta, \pr_r\rangle^2 + |A^g_\Sigma|^2_g\right) u_a^2 d\mu_g.
\end{split}\]

Now stability gives that $Q^{\mathrm{A}}(u_a) \geq 0$. But by condition (A$'$), we have $\phi''-\frac{\phi'}{r} - (\phi')^2 \leq 0$, so $Q^{\mathrm{A}}(u_a)=0$ and we must have $A^g_\Sigma \equiv 0$. Moreover, if condition (A$'$) is strict, then we must also have $\langle\nu_\delta ,\pr_r\rangle \equiv 0$, that is, $\Sigma$ corresponds to a cone with vertex at the origin. As $\Sigma$ is smooth, it must be a piece of a plane through the origin. 

In the case where condition (A$'$) is identically zero, $M$ has constant sectional curvature $\kappa$, and its totally geodesic hypersurfaces are well-known. If $\kappa \leq 0$ or if $\kappa>0$ and $S$ is weakly convex, then any totally geodesic hypersurface in $B_R$ that contacts the boundary at angle $\pit$ must pass through the centre. (In the only remaining case $\kappa>0$ and $S$ is not convex, it is possible that $\Sigma$ is a totally geodesic equator which does not intersect $S$.)
\end{proof}

We remark that, by standard approximation arguments again as in \cite[Proposition 14]{NZ24}, the result of Proposition \ref{prop:half-space-bernstein} also holds if $\Sigma$ is permitted to have a singular set of Hausdorff dimension less than $n-2$. We have presented the simplified version above to avoid duplicating those technical considerations. 

We now proceed to the alternative proof of Theorem \ref{thm:frankel}, using the classification of stable hypersurfaces in a half-ball. Recall that Theorem \ref{thm:frankel} is stated for minimal hypersurfaces $\Sigma$ contained in the half-ball $B_R \cap \overline{N_a}$ where, in particular, $M=B_R$. We will need the following condition on $S=\pr M = \pr B_R$, which we have written according to each model for convenience:

\begin{enumerate}
\item[(B)] $\pr M$ is (weakly) convex (or empty);
\item[(B$'$)] $\phi'(\bar{r}) + \frac{1}{\bar{r}} \geq 0$; 
\item[(B$''$)] $h'(\bar{\rho})\geq 0$. 
\end{enumerate}

\begin{proof}[Alternative proof of Theorem \ref{thm:frankel}, assuming (B)]
The result follows almost verbatim from the Proofs of Theorems 2 and 5 in \cite[Section 4.2]{NZ24}. To help orient the reader, we include a brief sketch as follows. By the transversality theorem and continuity, it is enough to work under the assumption that $\Sigma_i$ intersect $D_a$ transversely. 

Claim 1: The intersection property holds for the stable hypersurfaces identified by Proposition \ref{prop:half-space-bernstein}.

Proof: If $\Sigma = D_b \cap B_R$, $\Sigma' = D_{b'}\cap B_R$ are equatorial discs, then certainly $\Sigma \cap \Sigma' \cap \overline{N_a}$ contains the centre $o$ for any $a,b,b'$. The establishes the claim except in the case where $M$ is a spherical space form, and $S$ is strictly concave. In this case, without loss of generality we may assume $(M, \bar{g})= (\mathbb{S}^{n+1}, g_1)$. If $\Sigma,\Sigma'$ are totally geodesic equators, then their intersections with $D_a$ are totally geodesic and must themselves intersect (for instance, by the usual Frankel property in $\mathbb{S}^n$). If $\Sigma$ does not contact $S$, then that intersection $\Sigma\cap \Sigma' \cap D_a$ must be contained in $B_R$, which establishes the intersection property. If $\Sigma,\Sigma'$ both contact $S$, then they must be equatorial discs as above. 

Claim 2: For any of the stable hypersurfaces $\Sigma$ identified by Proposition \ref{prop:half-space-bernstein}, there are perturbations to either side, which contact $S$ at angle $\pit$ and have strictly positive mean curvature.

Proof: If $M$ is a spherical space form and $\Sigma$ does not contact $S$, then the desired perturbation is just the parallel foliation near the totally geodesic equator $\Sigma$. For the equatorial discs, we defer the proof to Lemma \ref{lem:mc-foliation} below.

As in \cite{NZ24}, $\bar{\Sigma}_0 \cap \bar{\Sigma}_1=\emptyset$, then we consider the region $\Omega \subset B_R$ bounded by $\Sigma_0, \Sigma_1$ and $D_a$. 

If both $\Sigma_i$ are stable, then by Claim 1 we are done. If one is stable, say $\Sigma_0$, then by Claim 2 we may perturb and replace $\Sigma_0$ by a strict barrier inside of $\Omega$. We may further perturb the boundary $\pr\Sigma_0$ inward to a fixed boundary $\Gamma \subset D_a\cap B_R$ that is nowhere totally geodesic with respect to $\delta$ (in the conformal model). 

Solving the Plateau problem in $\Omega$ with boundary $\Gamma$ yields a solution $\Sigma$ that is stable for one-sided variations fixing $D_a$. (Note that convexity of $\pr M$ ensures that the Plateau problem has a solution that stays in $M$.) But each of the stable hypersurfaces identified by Proposition \ref{prop:half-space-bernstein} intersects $D_a$ in an $(n-1)$-manifold that is totally geodesic with respect to $\delta$, which is a contradiction. 
\end{proof}

\begin{lemma}
\label{lem:mc-foliation}
Let $(M,\bar{g})$ and $S=\pr M=\pr B_R$ be as in Section \ref{sec:radial-setup}, satisfying condition (A). Consider the equatorial disk $\Sigma_0 = D_a \cap M$ for some $a\in\mathbb{S}^n$. Then there is a foliation $\{\Sigma_s\}$ of $M$ such that each $\Sigma_s$ is an embedded hypersurface that is strictly mean convex (mean curvature vector pointing away from $\Sigma_0$) and contacts $S$ with angle $\pit$.
\end{lemma}
\begin{proof}
We use the conformal model $(M^{n+1}, \bar g) \simeq (\mathbb{B}^{n+1}_{\bar{r}}, e^{2\phi} \delta)$. The proof follows as in \cite[Proposition 27]{NZ24}, with a somewhat more general computation. In particular, for any $y\in\mathbb{B}_{\bar{r}}$ there is a conformal translation $\Phi_y:\mathbb{B}_{\bar{r}}\to \mathbb{B}_{\bar{r}}$ mapping $0$ to $y$, given by \[
\Phi_y(x) = \bar{r}^{2} \frac{(\bar{r}^2 + 2 \langle x, y\rangle + |x|^2) y + (\bar{r}^2 - |y|^2) x}{\bar{r}^4 + 2\bar{r}^2  \langle x, y \rangle +|x|^2 |y|^2}. 
\]

We also record \[
|\Phi_y(x)|^2 = \bar{r}^2 \frac{ \bar{r}^2 |x + y|^2}{\bar{r}^4 + 2 \bar{r}^2 \langle x, y\rangle +|x|^2 |y|^2 }, \qquad 
(\Phi_y)^\ast \delta_{ij} =\Big(\frac{\bar{r}^2(\bar{r}^2 - |y|^2)}{\bar{r}^4 + 2 \bar{r}^2 \langle x, y\rangle + |x|^2 |y|^2} \Big)^2 \delta_{ij}. 
\]

Choose coordinates so that $a= e_0 \in \mathbb{S}^n\subset\mathbb{R}^{n+1}$. We set $\Sigma_s = \Phi_s(\Sigma_0)$, for $\Phi_s = \Phi_y$, where $y=se_0$. Then the family $\{\Sigma_s\}$ foliates $M$. As $\Phi_s$ is conformal, each $\Sigma_s$ still contacts $\pr\mathbb{B}_{\bar{r}}$ at angle $\pit$. To show that the mean curvature is strictly positive, without loss of generality we can consider $s\geq 0$, and note that the mean curvature of $\Sigma_s$ with respect to $g = e^{2\phi}\delta$ will be the same as the mean curvature of $\Sigma_0$ with respect to $\Phi_s^* g = e^{2u}\delta$, where 

\[
u(x) = u_s(x):=  \phi(|\Phi_s(x)|) + \log\Big(\frac{\bar{r}^2(\bar{r}^2 - s^2)}{\bar{r}^4 + 2 \bar{r}^2 s x_0 + |x|^2 s^2} \Big).
\]

By the formula for the mean curvature under conformal change (cf. Appendix \ref{sec:conformal-change}), since $\Sigma_0$ is totally geodesic (with respect to $\delta$), it suffices to show that  $\left.(\partial_\nu u)(x) = \frac{\pr u}{\pr x_0}\right|_{x_0=0} <0$ for all $x\in \Sigma_0$ and all $s>0$. Using $|\Phi_s(x)|^2 = \frac{\bar{r}^4 (|x|^2 + 2 x_0 s+ s^2)}{\bar{r}^4 + 2 \bar{r}^2 x_0 s+ |x|^2 s^2}$, we find that 
\[
\left.\frac{\pr u}{\pr x_0}\right|_{x_0=0} = -\frac{2\bar{r}^2 s}{\bar{r}^4+s^2r^2} + \bar{r}^2 \sqrt{\frac{r^2+s^2}{\bar{r}^4+s^2r^2}}\phi'\left(\bar{r}^2 \sqrt{\frac{r^2+s^2}{\bar{r}^4+s^2r^2}}\right) \left( \frac{s}{r^2+s^2} - \frac{\bar{r}^2 s}{\bar{r}^4+s^2r^2}\right),
\] 
where $r=|x|$. Let $t= \bar{r}^2\sqrt{\frac{r^2+s^2}{\bar{r}^4 + s^2r^2}} \in [s,\bar{r}]$. Then the above simplifies to 
\[
\begin{split}
\left.\frac{\pr u}{\pr x_0}\right|_{x_0=0} &= -\frac{2 s}{\bar{r}^2(r^2+s^2)}t^2 +  t \phi'(t) s \frac{\bar{r}^2-t^2}{\bar{r}^2(r^2+s^2)}
\\&= \frac{2 s t^2}{\bar{r}^2(r^2+s^2)}\left( -1  +  \frac{\phi'(t)}{t} \frac{\bar{r}^2-t^2}{2}\right).
\end{split}
\]
Note that certainly $\frac{\phi'(t)}{t} < \frac{2}{\bar{r}^2-t^2}$ for $r$ close to $\bar{r}$, hence $t$ close to $\bar{r}$ (as $\phi$ is smooth and $M$ is compact). But $F(t):=\frac{2}{\bar{r}^2-t^2}$ satisfies $F'(t) - tF(t)^2\equiv 0$, so by condition (A$'$) and comparison theory for first order ODE, we must have $\frac{\phi'(t)}{t} < F(t)$ for all $t$. It follows that $\left.\frac{\pr u}{\pr x_0}\right|_{x_0=0} <0$ for all $s>0$, as desired. 

\end{proof}

\section{Additional details for the proof of Proposition \ref{prop:emb-dual}}\label{sec:ros-details}

Proceeding with the notation introduced in the proof of Proposition \ref{prop:emb-dual}, we can argue as Ros does in \cite{Ros95} to show: \\

\textbf{Claim:} \textit{Each surface $U_\pm'$ is a topological disk.}\\

Indeed, $\Sigma'$ is a torus and so $\chi(\Sigma') = 0$. On the other hand, let $K$ denote the Gauss curvature of $\Sigma'$ and $k_\pm$ the geodesic curvature of $\partial U'_{\pm}$ (with respect to the inward-pointing conormal). Of course $k_- = - k_+$ along $Z'$. 

Let $\theta^\pm_i$ denote the \textit{sum} of exterior angles ($\pi$ minus the interior angle) of $\partial U_\pm'$ at each node $p_i' \in \partial U_\pm'$. Since $\partial U_\pm' = Z'$ consists of smooth arcs meeting orthogonally at each $p_i'$, we must have $\theta^\pm_i = \pi$ (the sum two exterior angles, each $\pi/2$). In fact, since the arcs of $Z'$ meet orthogonally, we must have $\theta^\pm_i = \frac{\pi}{2}$ for $i =1, 2$.

Now Gauss-Bonnet gives that 
\[
\int_{U_\pm'} K \, = 2\pi \chi(U_\pm') - \int_{Z'} k_\pm  - \sum_{i} \theta^\pm_i =2\pi \chi(U_\pm') - \int_{Z'} k_\pm  - 2\pi. 
\]
Adding these two equations gives 
\[
0 = 2\pi \chi(\Sigma') = \int_{\Sigma'} K = \int_{U_+'} K + \int_{U_-'} K  = 2\pi\big(\chi(U_+') + \chi(U_-') \big) -4\pi.
\]
We conclude 
\[
\chi(U_+') + \chi(U_-') = 2.
\]
On the other hand, a surface with boundary has Euler characteristic at most 1 and equality only holds for a disk. Thus, we conclude $U_+'$ and $U_-'$ are topologically disks. 

\section{Compactness for FBMS}
\label{sec:compactness}

In this section, we discuss how the compactness theory of Fraser-Li \cite{FL14} may be extended to a slightly more general setting - essentially, so long as there is a sufficient source of positive curvature to rule out stable surfaces.

First, we need the following version of Reilly's formula, which is the formula in \cite[Lemma 2.6]{FL14} with slightly weakened hypotheses, as in \cite[Theorem 2.1]{JXZ}:

\begin{proposition}[\cite{FL14, JXZ}]
\label{prop:reilly}
Let $\Omega\subset M$ be a bounded domain with piecewise smooth boundary $\pr\Omega = \bigcup_i {\Sigma_i}$, and set $\Gamma = \bigcup_i \pr\Sigma_i$. Assume that $u \in C^\infty(\bar{\Omega}\setminus\Gamma)$, and that $\bar{\nabla}^2 u \in L^2(\Omega)$. Then

\[ \int_{\Omega} \left( |\bar{\nabla}^2 u|^2 - (\bar{\Lap}u)^2 + \Ric_M(\bar{\nabla}u,\bar{\nabla}u)\right) = \sum_i J_{\Sigma_i}(u),\]
where 
\[J_{\Sigma_i}(u)=\int_{\Sigma_i}\left( -\pr_{\nu_i}u \Lap_{\Sigma_i} u + \langle \nabla^{\Sigma_i} u , \nabla^{\Sigma_i} \pr_{\nu_i}u\rangle - A^{\Sigma_i}(\nabla^{\Sigma_i}u, \nabla^{\Sigma_i}u) + H^{\Sigma_i}(\pr_{\nu_i}u)^2\right)\]
and the quantities on the right are defined with respect to the outer unit normal $\nu_i$ on $\Sigma_i$. 
\end{proposition}

Now recall that in \cite{CW83}, Choi-Wang showed that closed minimal hypersurfaces in an ambient space of positive Ricci curvature enjoy a lower bound on their first Laplace eigenvalue. In \cite[Theorem 3.1]{FL14}, Fraser-Li showed that there is an analogous lower bound on the first Steklov eigenvalue of a FBMS, when the ambient space $(M,\pr M)$ has nonnegative Ricci curvature and strictly convex boundary. We briefly observe that when $M$ has \textit{weakly} convex boundary but strictly positive Ricci curvature, we also have a lower bound on the first \textit{Neumann} eigenvalue:

\begin{proposition}\label{prop:choi-wang}
Let $M$ be a compact orientable Riemannian manifold with weakly convex boundary $\partial M$ and $\Ric_M \geq \kappa > 0$. Suppose that $(\Sigma, \partial \Sigma) \hookrightarrow (M, \partial M)$ is an embedded free boundary minimal hypersurface. Then the first Neumann eigenvalue of $\Delta_{\Sigma}$ satisfies $\lambda_1 \geq \kappa/2$. 
\end{proposition}
\begin{proof}
Let $\phi$ be a nonconstant Neumann eigenfunction of $\Lap_\Sigma$, so that $\Lap_\Sigma \phi = -\lambda \phi$ on $\Sigma$, and $\pr_\eta \phi =0$ on $\pr\Sigma$. 

Note that $\Sigma$ separates $M$ into two connected components and denote one of them by $\Omega$, so that $\pr\Omega = \Sigma \cup S$, where $S$ is the appropriate portion of $\pr M$. We may choose $\Omega$ so that $\int_\Sigma A^\Sigma(\nabla^\Sigma \phi, \nabla^\Sigma \phi)\geq 0$ (since it changes sign upon taking the opposite component). 

As $\Sigma$ is minimal, if $u$ is any $C^1$ extension of $\phi$, we will have 
\[
\begin{split}
 J_\Sigma(u) &\leq \int_{\Sigma}\left( -\pr_{\nu}u \Lap_{\Sigma} \phi + \langle \nabla^{\Sigma} \phi , \nabla^{\Sigma} \pr_{\nu}u\rangle \right) \\&= -2 \int_\Sigma \pr_\nu u \Lap_\Sigma \phi + \int_{\pr\Sigma} \pr_\nu u \pr_\eta\phi=2\lambda \int_\Sigma \phi\pr_\nu u.
 \end{split}
 \]

We wish to extend $\phi$ according to the following mixed boundary value problem:
\begin{equation}
\label{eq:mixed-bvp-1}
 \begin{cases}
\bar{\Lap}u = 0 &, \text{ on } \Omega,\\
u = \phi &, \text{ on } \Sigma,\\
\pr_{\bar{\eta}}u = 0 &, \text{ on } S.
\end{cases}
\end{equation}

The key is that $\pit$ is an exceptional angle for the mixed boundary value problem in a domain with corners. For instance, the function $\phi$ has a smooth extension to $\bar{\Omega}$ precisely because it satisfies the compatibility condition $\pr_\eta \phi = 0$ (and $\eta=\bar{\eta}$ along $S$). Moreover, by a reflection argument as in \cite[Section 6]{PT20}, there is a solution of (\ref{eq:mixed-bvp-1}) which is smooth away from the corners, and $W^{2,2}$ (and $C^{1,\alpha}$) on $\bar{\Omega}$ (see also \cite[Proposition 3.5]{GX19}). In particular, Proposition \ref{prop:reilly} applies.

Continuing with the proof, as $\bar{\Lap}u =0$, by the Ricci bound and the discussion above we have
\[
\begin{split}
\kappa \int_{\Omega} |\bar{\nabla}u|^2 &\leq J_\Sigma(u) + J_S(u)  \leq 2\lambda \int_\Sigma \phi\pr_\nu u + J_{S}(u). 
\end{split}
\]

Since $\pr_{\bar{\eta}}u=0$ along $S$, we have \[J_{S}(u) = -\int_{S} A^{S}(\nabla^{S} u, \nabla^{S} u) \leq 0.\]

As $u$ is $W^{2,2}$, integration by parts gives \[\int_{\Omega} |\bar{\nabla}u|^2 = \int_\Sigma u\pr_\nu u + \int_{S} u\pr_{\bar{\eta}}u = \int_\Sigma \phi\pr_\nu u.\] 

Thus \[0\leq (2\lambda-\kappa)  \int_{\Omega} |\bar{\nabla}u|^2.\]

As $\phi$ is not constant we certainly have $\int_{\Omega} |\bar{\nabla}u|^2 >0$, so we conclude that $2\lambda \geq \kappa$ as claimed. 
\end{proof}

\begin{remark}
We pause to remark, as in \cite[Remark 2.7]{FL14}, that the regularity of the solution to (\ref{eq:mixed-bvp-1}) is a crucial and somewhat delicate issue. If $\Omega$ had interior angles $\omega<\pit$, the regularity theory should give that $u\in C^{1,\alpha}(\bar{\Omega})$, with a growth estimate up to the boundary that is sufficient to apply Proposition \ref{prop:reilly} (see \cite[Theorem 3.1]{JXZ}). On the other hand, if the interior angle $\omega\in(\pit,\pi)$ then we only expect a solution of (\ref{eq:mixed-bvp-1}) to be $C^{0,\alpha}$, which may not be sufficient to apply Reilly's formula. (Indeed, consider the region $\Omega$ in the Euclidean $\mathbb{B}^3$ bounded by two flat discs on either side of the origin, so that $\Omega$ has interior angles given by $\omega>\pit$. If Reilly's formula could be applied in this setting, then the alternative proof of the intersection property in \cite[Theorem 2.4]{FL14} would seem to give a contradiction.) 

A helpful local heuristic is the solution of the homogeneous problem in a (2D) sector of opening angle $\omega\in(0,\pi)$. In particular, for mixed Dirichet-Neumann conditions one has the harmonic function $u = r^{\alpha} \sin (\alpha\theta)$ where $\alpha = \frac{\pi}{2\omega}$. For \textit{generic} $\omega$, the exponent $\alpha$ dictates the H\"{o}lder exponent. In the exceptional case $\alpha=1$, of course, this $u$ is \textit{smooth}.
The reader may also consult \cite{AK80, AK82, Li86, Li89} and the discussions in \cite[Appendix A]{JXZ}, \cite[Remark 1.13]{tang2013mixed} for more on the regularity for these mixed boundary value problems. 
\end{remark}

From the work of Fraser-Li \cite{FL14} and the above Neumann eigenvalue estimate, we may deduce the following compactness result: 

\begin{lemma}
\label{lemma:compactness}
Let $(M_i^3, \partial M_i)$ be a sequence of compact, orientable Riemannian $3$–manifolds with boundary that converges smoothly to a compact, orientable Riemannian manifold $(M, \partial M)$. 
Let $\Ric_{M_i}$ and $k_{\partial M_i}$ denote the Ricci curvature of $M_i$ and the second fundamental form of $\partial M_i$, respectively. 
Fix $\kappa > 0$, and suppose that either:
\begin{enumerate}[(i)]
    \item $\Ric_{M_i} \geq 0$ and $k_{\partial M_i} \geq \kappa$ for all $i$; or
    \item $\Ric_{M_i} \geq \kappa$ and $k_{\partial M_i} \geq 0$ for all $i$.
\end{enumerate}
If $(\Sigma_i, \partial \Sigma_i) \hookrightarrow (M_i, \partial M_i)$ is a sequence of embedded FBMS, then after passing to a subsequence, $(\Sigma_i, \partial \Sigma_i)$ converges smoothly to an embedded FBMS $(\Sigma, \partial \Sigma) \hookrightarrow (M, \partial M)$.
\end{lemma}
\begin{proof}
Case (i) is precisely the compactness result of Fraser-Li \cite{FL14} for free boundary minimal surfaces of fixed topology (which applies for varying metrics, see \cite[Theorem 2.3]{MNS17}).

For case (ii), we can still use the theory developed in \cite{FL14}, with a modification to the final step. Indeed, by following the proof of \cite[Theorem 1.2]{FL14}, we have in general that (up to a subsequence) the $\Sigma_i \hookrightarrow M_i$ converge smoothly to some $\Sigma \hookrightarrow M$ away from a finite set of points $\mathcal{S}\subset \Sigma$. When $\mathcal{S} \neq \emptyset$, then for $i$ sufficiently large, $\Sigma_i$ is locally the disjoint union $k \geq 2$ (independent of $i$) graphs over $\Sigma$. In this setting, Fraser-Li construct a sequence of Lipschitz functions $\phi_i \in C^{0, 1} (\Sigma_i)$ such that 
\[
\int_{\partial \Sigma_i} \phi_i = 0, \qquad \lim_{j \to \infty} \int_{\Sigma_i} |\nabla \phi_i |^2 = 0, \qquad \lim_{j \to \infty} \int_{\Sigma_i} \phi^2_i >  0, \qquad \lim_{j \to \infty} \int_{\partial \Sigma_i} \phi_i^2  >  0. 
\]
(Up to shifting by a constant so that the first condition holds, the function $\phi_i$ is essentially $1$ on one of the graphical sheets and $-1$ on the remaining sheets, away from a small neighborhood of $\mathcal{S}$.) 

In \cite{FL14}, corresponding to case (i), this sequence of functions contradicted an estimate for the first Steklov eigenvalue \cite[Theorem 3.1]{FL14}. In our case (ii), the same sequence of functions will instead contradict the Neumann eigenvalue estimate Proposition \ref{prop:choi-wang}, as $\frac{\int_{\Sigma_i} |\nabla \phi_i|^2}{\int_{\Sigma_i}\phi_i^2}\to 0$. This implies that $\mathcal{S}=\emptyset$, which completes the proof of compactness.
\end{proof}

\begin{remark}
As for smooth compactness results in other settings \cite{CM12, SH17},  one may also rule out multiplicity by arguing that the limit surface would have to be stable; indeed, applying the variation by a constant function shows that there are no stable FBMS under the conditions of Lemma \ref{lemma:compactness}.
\end{remark}

\section{Implicit Function Theorem in the Proof of Proposition \ref{prop:ift-caps}}\label{sec:ift-caps}

We begin with the following lemma (compare \cite[Lemma 5.5]{DM09}). 

\begin{lemma}\label{lem:killing-Hf-orthogonality}
Let $\rho$ be the distance to a fixed point $o \in \mathbb{S}^3$ and let $f = f(\rho)$ be a smooth radial function on $\mathbb{S}^3$. Suppose $x : (\Sigma, \partial \Sigma)\looparrowright (\mathbb{S}^3, \partial B_\frac{\pi}{2})$ is any compact, smoothly immersed, $2$-sided surface in $\mathbb{S}^3$ that contacts $\partial B_{\frac{\pi}{2}}$ at angle $\frac{\pi}{2}$ along $\partial \Sigma$. 

Let $K$ be a Killing field on $\mathbb{S}^3$ that generates an isometry of $B_{\frac{\pi}{2}}$ (that is, a rotation which preserves $\partial B_{\frac{\pi}{2}}$). Then 
\[
\int_{\Sigma} H_f \langle K, \nu \rangle\,  e^{-f} d\mu_g = 0 .
\]
\end{lemma}

(Note in particular that we do not have any assumption on the mean curvature.)

\begin{proof}
As $K$ is Killing, in any local orthonormal frame $e_1, e_2$ along $\Sigma$, we have $\langle \bar{\nabla}_{e_i} K, e_i \rangle = 0$. Since $K$ corresponds to an ambient isometry of $B_{\frac{\pi}{2}}$, we note that $\langle K, \partial_\rho \rangle = 0$ and hence $\langle K, \bar{\nabla} f \rangle = 0$ since $f$ is assumed to be radial. Finally, recall that $H_f = H - \langle \nabla f, \nu \rangle$. 

Using these observations, the standard computation gives
    \begin{align*}
   \mathrm{div}_f(K^\top) &= \mathrm{div}_f(K) -  \mathrm{div}_f(K^\perp) \\
   & = \sum_{i =1}^2 \langle \bar{\nabla}_{e_i} K, e_i \rangle  - \langle K, \bar{\nabla} f \rangle -\langle K, \nu \rangle H  + \langle K, \nu \rangle \langle \nu, \bar{\nabla} f \rangle \\
   & = - \langle K, \nu \rangle H_f. 
    \end{align*}
Recall that we let $\eta$ denote the outward pointing conormal of $\partial \Sigma$ in $\Sigma$. Since $x(\Sigma)$ contacts $\partial B_{\frac{\pi}{2}}$ at angle $\frac{\pi}{2}$, we have $\eta = \pm \rho$. Moreover, since $x(\partial\Sigma) \subset \partial B_{\frac{\pi}{2}}$ and $K$ is tangent to $\partial B_{\frac{\pi}{2}}$, we have that $K^\top$ is also tangent to $\partial B_{\frac{\pi}{2}}$. Therefore, $\langle K^\top, \eta \rangle = 0$ along $\partial \Sigma$. Now the divergence theorem gives
    \begin{align*}
     \int_{\Sigma} H_f \langle K, \nu \rangle e^{-f} d\mu_g  = -   \int_\Sigma \mathrm{div}_f(K^\top) e^{-f} d\mu_g = -\int_{\partial \Sigma} \langle K^\top, \eta\rangle d\sigma_g = 0. 
    \end{align*}
This completes the proof. 
\end{proof}

\begin{proof}[Remainder of Proof of Proposition \ref{prop:ift-caps} ]
Let $x, \Sigma, R, f, I, \Omega$ and $\mathcal{H}$ be as before. Let $C^{2, \alpha}_{\mathrm{Neu}}(\Sigma)= \big\{ u \in  C^{2,\alpha}(\Sigma) : \partial_\eta u = 0\big\}$. We denote the linearization of $\mathcal{H}$ in the first component by
    \[
    (\delta_1 \mathcal{H})_{(0,0)} : C^{2, \alpha}_{\mathrm{Neu}}(\Sigma) \to C^{\alpha}(\Sigma),
    \]
and in particular
    \begin{align*}
    L^g_{f} u:= (\delta_1\mathcal{H})_{(0,0)}(u) &:= \frac{d}{dt}\Big|_{t=0}\Big(H_{tu} - g_1\big( \bar{\nabla} f \circ x_{tu}, \nu_{tu} \big)\Big).
    \end{align*}
As in Section \ref{sec:conf-weight-corres}, this is precisely the weighted stability operator
    \begin{align*}
      L^g_{f} u   & := \Delta u - \langle \bar{\nabla} f, \nabla u \rangle + (2 + \bar{\nabla}^2 f (\nu, \nu) + |A|^2) u.    \end{align*}
Note that, as $\pr B_\pit$ is totally geodesic, the weighted index form is precisely $-\int_\Sigma u L^g_f u d\mu_g$, and its kernel is precisely the Neumann kernel $\mathcal{K}:=\ker(L^g_f)\subset C^{2,\alpha}_{\mathrm{Neu}}(\Sigma)$. By our assumption on the kernel, every element of $\mathcal{K}$ is of the form $\langle K,\nu\rangle$, where $K\in \mathcal{K}_0$ is a Killing field that generates rotation of $\mathbb{S}^3$ which fixes $\pr B_\pit$. In particular, we may take an orthonormal basis $\{\tilde{u}_i = \langle K_i, \nu \rangle\}_{i=1}^J$, of $\mathcal{K}$, where each $K_i$ is as above. 

Note that on $\partial \Sigma$,  we have 
    \[
    \partial_\eta \tilde{u}_i = \langle \bar{\nabla}_{\partial_\rho} K_i, \nu \rangle + \langle K_i, \bar{\nabla}_{\partial_{\rho}} \nu \rangle = -\langle \bar{\nabla}_{\nu} K_i, \partial_\rho \rangle + A_{11}\langle K_i, \partial_\rho \rangle = 0,
    \]
using that $\bar{\nabla}_{\nu} K_i$ and $K_i$ are orthogonal to $\partial_\rho$ (the former orthogonality holds because $\Sigma$ contacts $\pr B_\pit$ at angle $\pit$, and the flow of $K_i$ preserves $\partial B_{\frac{\pi}{2}}$). 
Now $L^g_f$ is symmetric in the sense that if $u, v \in C^{2,\alpha}_{\mathrm{Neu}}(\Sigma)$. So, 
    \[
    \int_{\Sigma} (L^g_f u) v \; e^{-f} d\mu_{g} =  \int_{\Sigma} u(L^g_f v) \; e^{-f} d\mu_{g}.
    \]
In particular, if we let 
    \[
    \mathcal{I}^\perp= \Big\{v \in C^\alpha(\Sigma) : \int_\Sigma \tilde{u}_i v \, e^{-f} d\mu_g = 0 \;\; \text{for} \;\; 1\leq i\leq J \Big\}
    \]
denote the $L^2_{e^{-f}}$-orthocomplement of $\mathcal{I} := \mathrm{span}\{\tilde{u}_1, \dots, \tilde{u}_J\} \subset C^\alpha(\Sigma)$, then $\mathrm{Im}(L^g_f) = \mathcal{I}^\perp$.

Define 
    \[
    \Omega^\perp  = \big\{ u \in \Omega :  \int_\Sigma \tilde{u}_i \, u  \, e^{-f} d\mu_g = 0 \;\; \text{for} \;\; 1\leq i\leq J \Big\}
    \]
and 
    \[
   \mathcal{K}^\perp  = \big\{ u \in C^{2,\alpha}_{\mathrm{Neu}}(\Sigma) :  \int_\Sigma \tilde{u}_i \, u  \, e^{-f} d\mu_g = 0 \;\; \text{for} \;\; 1\leq i\leq J \Big\}. 
    \]
To apply the implicit function theorem in the presence of the nontrivial kernel, we consider the map 
    \[
    \mathcal{G} : \Omega^\perp \times \mathbb{R}^J \times I \to C^{\alpha}(\Sigma)
    \]
given by 
    \[
    \mathcal{G}(u, \mathbf{A}, s) := \mathcal{H}(u, s) - \sum_{i=1}^J a_i \langle K_i, \nu_u\rangle ,
    \]
where $\mathbf{A} = (a_1,\cdots, a_J)$. Then, the variation in the first two components
    \[
    (\delta_{12} \mathcal{G})_{(0,\vec{0},0)} : \mathcal{K}^\perp\times \mathbb{R}^J \to C^{\alpha}(\Sigma)
    \]
is given by 
    \[
    (\delta_{12} \mathcal{G})_{(0,\vec{0}, 0)}(u, \vec{b})= \hat{L}(u, \vec{b}) = L^g_f u -\sum_{i=1}^J b_i \tilde{u}_i. 
    \]
The linearization is isomorphism since if $v \in \mathcal{I} \subset C^\alpha(\Sigma)$, then $v = \hat{L}(0, \vec{b})$ for some unique $\vec{b}$ and if $v \in \mathcal{I}^\perp = \mathrm{Im}(L^g_f) \subset C^\alpha(\Sigma)$, then $v = L^g_f u = \hat{L}(u, \vec{0})$ for some  unique $u \in \mathcal{K}^\perp$.

It follows from the implicit function theorem that there exists $\delta > 0$ and a neighborhood $U \subset \Omega^\perp$ such that for $|s| < \delta$, we can find unique $u_s \in U$ and $\{a_i(s)\}_{i=1}^J$ (depending smoothly on $s$) so that 
    \[
    H_{u_s} - g_1(\bar{\nabla} f_{R+s}, \nu_{u_s}) - \sum_{i=1}^J a_i(s) \langle K_i, \nu_{u_s}\rangle = 0. 
    \]
Now integrating this expression against $\langle K_i, \nu_{s} \rangle$ for each of $i$ and applying Lemma \ref{lem:killing-Hf-orthogonality}, we conclude $a_i(s) = 0$ for each $i$. Thus
    \[
    H_{u_s} = g_1(\bar{\nabla} f_{R+s}, \nu_{u_s}).
    \]
In particular, $x_{u_s}$ will be our desired $f_{R+s}$-FBMS. 
    
Finally, the conditions of embeddedness and non-rotational symmetry are open conditions. Thus the deformations $x_{u_s} : \Sigma \looparrowright B_{\frac{\pi}{2}}$ preserve them for all sufficiently small $s$. For the constraint property, we simply note that the exponential map $\exp^{\mathbb{S}^3}$ preserves the hemisphere $B_\pit$, as the latter is geodesically convex.

\end{proof}
\bibliography{low-genus}

\end{document}